\documentclass[12pt]{amsart}

\calclayout

\usepackage{amssymb,amsmath,amsthm, amsfonts}
\usepackage{mathtools, enumitem, wasysym}

\usepackage{xcolor}
\usepackage{hyperref}
\hypersetup{colorlinks=false,linkbordercolor=red,linkcolor=green,pdfborderstyle={/S/U/W 2}}

\usepackage{url}

\usepackage{graphicx}
\usepackage{wrapfig,tikz, tikz-cd}
\usetikzlibrary{arrows, arrows, calc, decorations.markings}

\usepackage{caption}
\usepackage{subcaption}
\usepackage{array}

\newtheorem{theorem}{Theorem}[section]
\newtheorem{proposition}[theorem]{Proposition}
\newtheorem{corollary}[theorem]{Corollary}
\newtheorem{lemma}[theorem]{Lemma}
\newtheorem{conjecture}[theorem]{Conjecture}

\theoremstyle{definition}
\newtheorem{definition}[theorem]{Definition}

\theoremstyle{remark}

\numberwithin{equation}{section}

\newcommand{\cut}{\operatorname{cut}}
\newcommand{\diag}{\operatorname{diag}}

\newcommand{\Sym}{\operatorname{Sym}}
\newcommand{\Pot}{\operatorname{Pot}}
\newcommand{\prim}{\operatorname{prim}}
\newcommand{\Tot}{\operatorname{Tot}}
\newcommand{\un}{\operatorname{un}}
\newcommand{\val}{\operatorname{val}}
\newcommand{\al}{\alpha}
\newcommand{\be}{\beta}

\newcommand{\ph}{\varphi}
\newcommand{\AAA}{\mathbb{A}}

\newcommand{\CC}{\mathbb{C}}
\newcommand{\dd}{\partial}

\newcommand{\ii}{\textbf{i}}
\newcommand{\GG}{\mathcal{G}}

\newcommand{\kk}{\textbf{k}}

\newcommand{\PP}{\mathbb{P}}
\newcommand{\RR}{\mathbb{R}}

\newcommand{\SSS}{\mathbb{S}}
\newcommand{\TT}{\mathcal{T}}
\newcommand{\QQ}{\mathbb{Q}}
\newcommand{\ZZ}{\mathbb{Z}}

\newcommand*\circled[1]{\tikz[baseline=(char.base)]{
\node[shape=circle,draw,inner sep=1](char) {#1};}}

\begin{document}

\title{Potentials for moduli spaces of $A_m$-local systems on surfaces}

\author{Efim Abrikosov}
\address{}
\curraddr{Yale University, Department of Mathematics, 10 Hillhouse Ave, 4th Floor}
\email{efim.abrikosov@yale.edu}

\keywords{Representation Theory, Cluster Varieties, Calabi-Yau categories, Quivers with Potential}

\begin{abstract}
We study properties of potentials on quivers $Q_{\mathcal{T},m}$ arising from cluster coordinates on moduli spaces of $PGL_{m+1}$-local systems on a topological surface with punctures. To every quiver with potential one can associate a $3d$ Calabi-Yau $A_\infty$-category in such a way that a natural notion of equivalence for quivers with potentials (called ``right-equivalence'') translates to $A_\infty$-equivalence of associated categories \cite[Section 8]{KS08}.

For any quiver one can define a notion of a ``primitive'' potential. Our first result is the description of the space of equivalence classes of primitive potentials on quivers $Q_{\mathcal{T}, m}$. Then we provide a full description of the space of equivalence classes of all \emph{generic} potentials for the case $ m = 2$ (corresponds to $PGL_3$-local systems). In particular, we show that it is finite-dimensional. This claim extends results of Gei\ss, Labardini-Fragoso and Schr\"oer (\cite{L08}, \cite{GLS13}) who have proved analogous statement in $m=1$~case.

In many cases $3d$ Calabi-Yau $A_\infty$-categories constructed from quivers with potentials are expected to be realized geometrically as Fukaya categories of certain Calabi-Yau $3$-folds. Bridgeland and Smith gave an explicit construction of Fukaya categories for quivers $Q_{\TT,1}$, see \cite{BS13}, \cite{S13}. We propose a candidate for Calabi-Yau $3$-folds that would play analogous role in higher rank cases, $m > 1$. We study their (co)homology and describe a construction of collections of $3$-dimensional spheres that should play a role of generating collections of Lagrangian spheres in corresponding Fukaya categories.
\end{abstract}

\maketitle

\section{Introduction.}

\subsection{Motivation.}
Informally, a quiver with potential is a finite oriented graph~$Q$ equipped with a possibly infinite formal linear combination of its oriented cycles~$W$. In mathematical literature this object was studied by V. Ginzburg in~\cite{G06} and from a slightly different point of view by Derksen, Weyman and Zelevinsky in~\cite{DWZ07}.

Ginzburg associates to every quiver with potential~$(Q, W)$ a dg-algebra~$\Gamma(Q,W)$ and shows that the category of its finite-dimensional modules has $3d$~Calabi-Yau property (\cite[Definition 3.2.3]{G06}). On the other hand, a well-known source of $3d$ Calabi-Yau $A_\infty$-categories are Fukaya categories of symplectic manifolds. It is an interesting question to relate these two origins of the same categorical structure in specific contexts.

Some potentials on a quiver give rise to equivalent $3d$~Calabi-Yau categories. From that perspective, it is important to understand spaces of potentials modulo some equivalence relation. Corresponding notion of equivalence for different potentials was introduced in \cite{DWZ07} under the name ``right-equivalence''.

In the present work, we deal with quivers describing cluster structure on moduli spaces of $A_m$-local systems on a topological surface with punctures \cite{FG03}\footnote{More precisely, type $A_m$ corresponds to twisted $SL_{m+1}$-local systems for the case of cluster $\mathcal{A}$-varieties and to framed $PGL_{m+1}$-local systems for cluster $\mathcal{X}$-varieties. In this paper, we omit details of this construction and focus on properties of associated quivers.}. Our main goal is to study spaces of potentials up to right-equivalences in this setting.

An important example of potentials for these quivers was given by Goncharov in \cite{G16}. The key property of his construction is that potentials are preserved by a class of combinatorial transformations called ``mutations''. By the general result of Keller and Yang \cite{KY09} this implies, in particular, that there exists a $3d$~Calabi-Yau category assigned to the whole cluster variety, not only a fixed quiver.

\medskip
The first algebraic result concerning right-equivalence classes of potentials on quivers $Q_{\TT, m}$ was given by Labardini-Fragoso \cite{L08} and subsequently by Gei\ss, Labardini-Fragoso and Schr\"oer \cite{GLS13} in the case $m=1$. In particular, they show that the space of equivalence classes of generic potentials on $Q_{\TT, 1}$ is finite-dimensional.

In the present paper we prove that the finite-dimensionality holds for the case of $A_2$-local systems. Hence, there is a finite-dimensional family of $3d$ Calabi-Yau categories associated to any such quiver via Ginzburg's construction. It is conceivable that analogous fact is true for arbitrary rank as well.

\medskip
Geometric side of the problem is to find appropriate symplectic manifolds and realize families of categories of the algebraic origin as Fukaya categories. In $A_1$-case it was done by Bridgeland and Smith, \cite{BS13} and Smith, \cite{S13}. Their manifolds are close relatives of those from the work of Diaconescu, Donagi and Pantev \cite{DDP06}, where symplectic manifolds are constructed from a \emph{holomorphic} quadratic differential on a complex curve. Bridgeland and Smith work with \emph{meromorphic} quadratic differentials which makes certain aspects of their approach quite different. We summarize details of their construction in Section~\ref{CYfoldsm1}. One of the key ideas in \cite{S13} is that finite-dimensionality result of \cite{GLS13} allows to identify a full subcategory of a Fukaya category from only finitely many multiplication coefficients. In Smith's paper these coefficients are described explicitly from the geometry of corresponding symplectic manifolds.

In this paper we propose a generalization of Smith's manifolds to higher rank cases~$A_m$, $m \geq 1$. To construct our open $3d$~Calabi-Yau manifolds we use points of Hitchin base $\mathcal{B}_{S, D, m}\simeq \bigoplus_{k = 2}^{m+1} H^0(S, K_S(D)^{\otimes k})$, where $S$ is an algebraic curve, $D$ is the divisor of marked points (e.g. $m = 1$ corresponds to the case studied by Bridgeland and Smith)\footnote{Note that one may also vary the complex structure on the underlying topological surface, as well as the divisor of marked points $D$, in that case the input becomes purely topological.}. Given a point $\Phi\in\mathcal{B}_{S, D, m}$ we construct an open $3$-dimensional Calabi-Yau manifold $Y_\Phi$ which agrees with \cite{S13} in the case $m = 1$. Following the logic of Smith's paper, our finite-dimensionality result should be sufficient to identify full subcategories of Fukaya categories of $Y_\Phi$ up to $A_\infty$-equivalence. Let us note that a closely related construction of open $3d$ Calabi-Yau manifolds was given in \cite[Section 8]{KS13}.

\smallskip
Therefore to fully understand the interplay between algebra and geometry, it is important to match parameters defining Fukaya categories of open $3d$ Calabi-Yau manifolds (e.g. second cohomology groups) and equivalence classes of potentials for the corresponding quivers. We provide a computation of $H^2(Y_\Phi, \QQ)$ and $H_3(Y_\Phi, \QQ)$ to support the claim that such a matching exists (Proposition \ref{topology}).

\subsection{Spaces of potentials modulo right-equivalences $\Pot(\mathcal T, m)$.} To state our main results about spaces of potentials for quivers of the interest $Q_{\mathcal{T}, m}$ we begin with a brief discussion of some general definitions. Precise description of quivers is postponed till Section \ref{clusterstr}; quivers with potentials are recalled in Section~\ref{QPsec}.

A quiver $Q$ is described by finite sets of vertices and arrows with their incidence relations. Let $\textbf{k}$ be a field of characteristic zero. A \emph{path algebra} $R\langle Q\rangle$ of the quiver over $\mathbf{k}$ is spanned by formal products of composable arrows with multiplication given by concatenation. In particular, there is a path of length zero associated to every vertex of $Q$ and it is idempotent in $R\langle Q\rangle$.

Path algebra possesses a natural completion with respect to path length denoted $R\langle\langle Q\rangle\rangle$. Its elements can be viewed as possibly infinite linear combinations of paths, such that there are only finitely many paths of any given degree. A \emph{potential} is an element of the completed path algebra $R\langle\langle Q\rangle\rangle$ that consists of cyclic paths. Every summand of the potential is considered up to cyclic shifts (see Definition \ref{potdef}).

Any automorphism of $R\langle\langle Q\rangle\rangle$ preserving vertices of the quiver induces a map on potentials. Two potentials on a quiver are called \emph{right-equivalent} if there is an automorphism of the completed path algebra that maps one potential to another.


\begin{figure}
	\centering
	\includegraphics{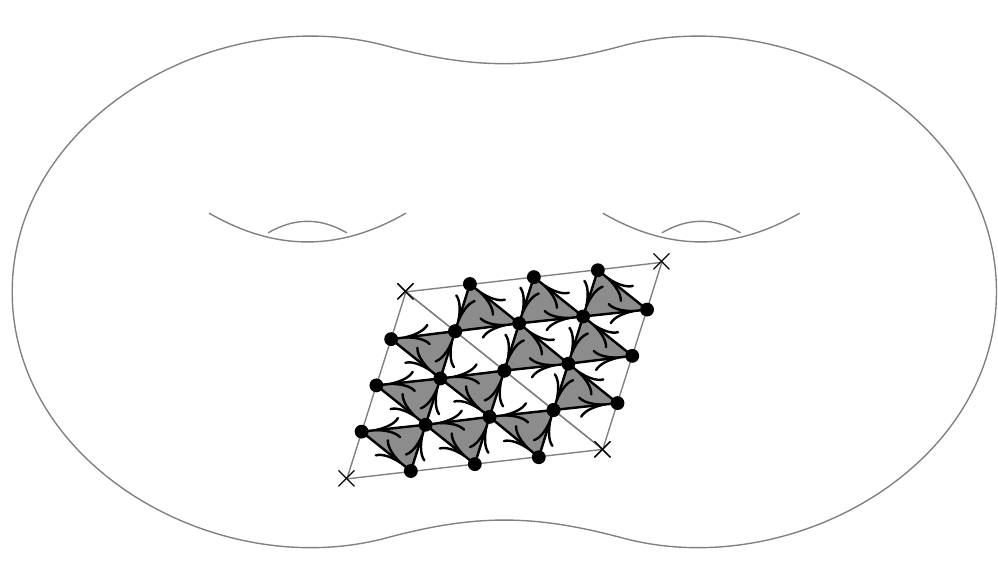}
	\captionof{figure}{Part of quiver $Q_{\TT, 3}$}
	\label{Surface_Quiver}
\end{figure}

We deal with quivers that arise from triangulations of compact oriented topological surfaces $\SSS$ of genus $g$ with $d$ marked points. We require that vertices of triangulations coincide with the set of markings on $\SSS$, and that endpoints of every edge are different. Necessary and sufficient conditions for existence of such triangulations is $d\geq3$ for $g>0$, or $d\geq4$ for $g=0$.

For every triangulation $\TT$ satisfying these requirements and any integer $m\geq 1$  there is a quiver $Q_{\TT, m}$. These quivers were originally discovered by Fock and Goncharov who studied cluster structure on moduli spaces of local systems on topological surfaces \cite{FG03}. Quiver $Q_{\TT,m}$ is obtained by placing a subquiver shown in Figure~\ref{triangle_quiver} in every triangle of the triangulation and identifying $m$ vertices along the shared edge of any two adjacent triangles. A part of such quiver for $m=3$ can be found in Figure~\ref{Surface_Quiver}. Little crosses denote marked points on the surface and gray lines in the interior represent arcs of the triangulation.

For the case $m=1$, we place one vertex at the midpoint of every edge of $\TT$; arrows join three midpoints of sides of every triangle following counterclockwise direction  (according to the orientation of $\SSS$). These are precisely quivers considered by Labardini-Fragoso in~\cite{L08}. For $m>1$ there are also vertices sitting in the interior of triangles.

\smallskip
Our first result uses a concept of a \emph{primitive} potential, it can be defined for an arbitrary quiver (see Section \ref{primdefsec}, Definition~\ref{primdef}). Namely, a potential on a given quiver is primitive if it is a linear combination of cycles from an explicitly defined collection of cycles~$\mathfrak{C}$, where \emph{all} coefficients are different from zero. Therefore, the space of all potentials admits a natural projection onto the span of elements in $\mathfrak{C}$. Throughout the paper, we say that a potential is \emph{generic} if all coefficients of this projection are nonzero (that is, the projection is a primitive potential itself). We stress that usually an arbitrary right-equivalence does not preserve the subspace of primitive potentials.

Let us describe explicitly the collection $\mathfrak{C}$ for quivers $Q_{\TT,m}$. It consists of three types of cycles (see example in Fig.~\ref{primtypes} for a patch around a vertex of the triangulation):
\begin{enumerate}[label=(\roman*)]
\item Boundary cycles of ``black'' regions (oriented counterclockwise);
\item Boundary cycles of ``white'' regions (oriented clockwise);
\item Cycles $L^{(k)}_p$, where $p$ is a puncture of $\SSS$ and $k = 1,2,\ldots m$ (oriented clockwise).
\end{enumerate}
The complement of the quiver embedded in~$\SSS$ is the union of disks. Direction of arrows of the quiver equips boundaries of these disks with orientation (shown in gray and white). Thus, first and second types of cycles of $\mathfrak{C}$ are defined as products of arrows along boundaries of disks with clockwise/counterclockwise oriented boundaries

The third type consists of cycles going around marked points on~$\SSS$ (right panel of Fig. \ref{primtypes}). Given a marked point $p\in\SSS$ and a triangle adjacent to it, for each $1\leq k\leq m$ there is a path of length $k$ parallel to the opposite side of $p$ in the triangle. Composing such paths across all triangles adjacent to $p$ we get a closed cycle of the third type. There are $m$ such cycles around every marked point, and the innermost cycle also belongs to the type associated to disks with clockwise boundary orientation.

Thus, a potential on $Q_{\TT,m}$ is primitive if it is a linear combination of all elements of $\mathfrak{C}$ with nonzero coefficients.

We are ready to state our first result.
\begin{theorem}\label{mn1}
  \begin{enumerate}[label=(\alph*)]
    \item The quotient of the space of primitive potentials for $Q_{\TT, m}$ by the action of the group of right-equivalences preserving this space is isomorphic to $\left(\emph{\kk}^{\times}\right)^{(m-1)d+1}$.
    \item Let $\pi_{\prim}$ denote the projection of the space of generic potentials onto the space of primitive potentials. For any right-equivalence $\varphi$ and for any generic potential $W$ the following holds:
      $$
      \pi_{\prim} \left(\varphi (W)\right) = \pi_{\prim} \left(\varphi(\pi_{\prim}(W))\right),
      $$
  \end{enumerate}
\end{theorem}
\begin{corollary}
  If two generic potentials are right-equivalent then their projections onto the space of primitive potentials are also right-equivalent.

  In other words, the equivalence class of the primitive part of a generic potential described in part (a) of Theorem \ref{mn1} is an invariant of the action of right-equivalences.

\end{corollary}

\proof This is immediate from part (b) of Theorem \ref{mn1}.

\smallskip

\begin{figure}
	\centering
	\begin{minipage}[b]{.33\textwidth}
		\centering
		\includegraphics[width=.9\textwidth]{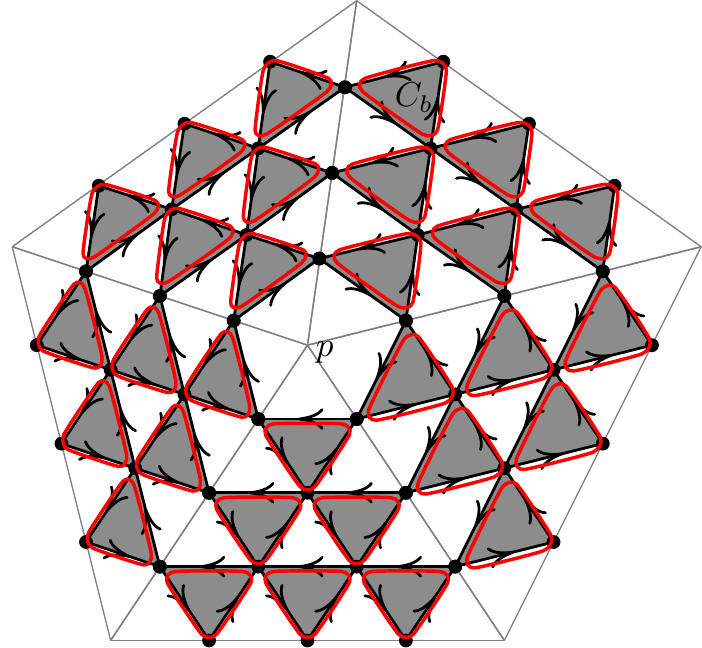}
	\end{minipage}%
	\begin{minipage}[b]{.33\textwidth}
		\centering
		\includegraphics[width=.9\textwidth]{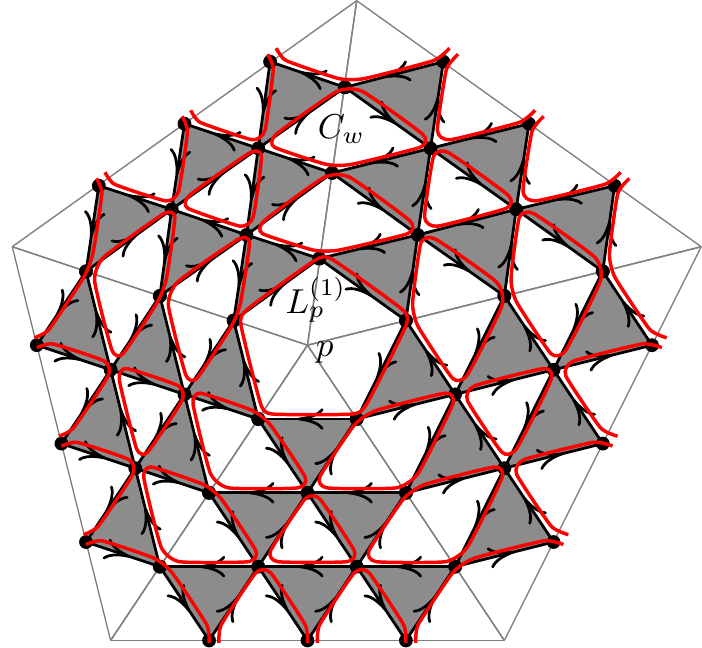}
	\end{minipage}%
	\begin{minipage}[b]{.33\textwidth}
		\centering
		\includegraphics[width=.9\textwidth]{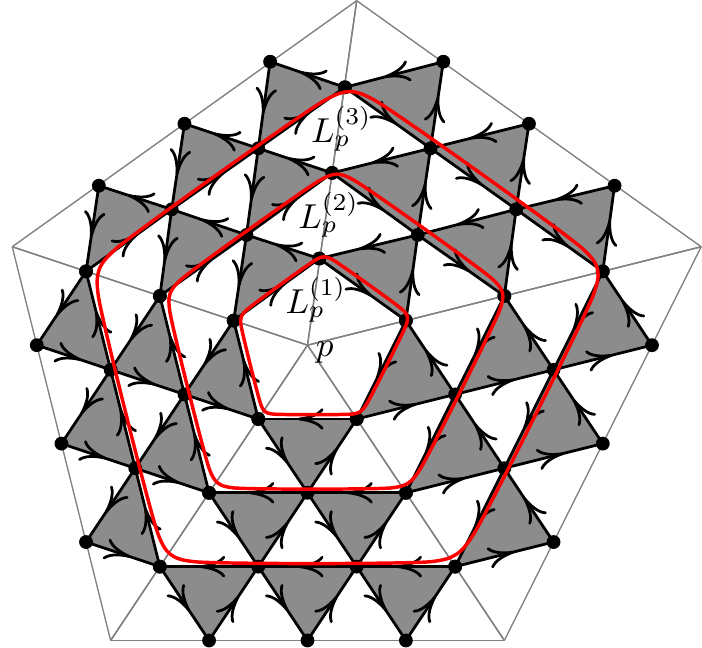}
	\end{minipage}
	\captionof{figure}{Three types of primitive cycles}
	\label{primtypes}
\end{figure}

\medskip
In our second result about spaces of potentials of the quivers $Q_{\TT, m}$ we specialize $m=2$. The notion of a \emph{strongly generic} potential used in the statement of the theorem is given in Definition~\ref{strgendef}. It can be expressed as non-vanishing of a system of linear equations on coefficients of cycles from $\mathfrak{C}$ (see (\ref{strgen})).

We expect that finite-dimensionality should hold for arbitrary $m\geq 1$.

\begin{theorem} The quotient of the space of strongly generic potentials on $Q_{\TT,2}$ modulo the action of the group of right-equivalences has dimension $d+2$.

Its natural projection to the space of equivalence classes of primitive potentials described in part (b) of Theorem \ref{mn1} is a fibration over the image with fibers isomorphic to an affine line $\mathbb{A}_{\emph{\kk}}^1$.
\end{theorem}

\subsection{Open 3d Calabi-Yau manifolds from points of Hitchin base} We introduce and study a class of symplectic Calabi-Yau manifolds that generalize the approach of Smith \cite{S13} to the case of local systems of arbitrary rank. These ideas are closely related to the construction from \cite{KS13}. We conjecture that Fukaya categories of these manifolds contain full subcategories that are $A_\infty$-equivalent to categories of finite-dimensional modules over Ginzburg algebra for quivers with potentials described in previous subsection.

Fix a complex curve $S$ with a divisor $D$ of marked points, and let $m\geq 1$ be the integer equal to the rank of local systems on $S$ minus one. The input of our construction is a generic point $\Phi = (\varphi_2, \varphi_3,\ldots \varphi_{m+1})$ of Hitchin base
$$
\mathcal{B}_{S,D,m} \equiv \bigoplus_{k = 2}^{m+1}H^0\left(S, K_S(D)^{\otimes k}\right),
$$
where $K_S(D)$ is a twisted canonical class of $S$. The quickest way to describe the open Calabi-Yau $3$-fold $Y_\Phi$ associated to the point of Hitchin base is by an equation in the total space of rank three vector bundle $\mathcal W\rightarrow S$. For clarity of exposition we fix two line bundles $\mathcal L_1, \mathcal L_2$ such that $\mathcal L_1\otimes \mathcal L_2 \simeq K_S(D)$. Then $\mathcal W$ is defined as the direct sum:
\begin{equation}\label{Wdef}
\mathcal{W} \overset{\operatorname{def}}=\mathcal L_1^{\otimes (m+1)}(-D)\oplus\mathcal L_1\mathcal L_2\oplus\mathcal L_2^{\otimes (m+1)}
\end{equation}
Given a point $\Phi$ of Hitchin base, $Y_\Phi$ is defined as a $3$-dimensional subvariety of $\Tot \mathcal {W}$ given by the equation:
\begin{equation}\label{Ydefintro}
(\delta a)c = b^{m+1} + \varphi_2b^{m-1} + \ldots + \varphi_{m+1},
\end{equation}
where $\delta\in H^0(S, \mathcal O(D))$ is the unique up to scalar multiplication section with zeroes at divisor $D$; and $a, b$ and $c$ represent coordinates on the respective components of decomposition (\ref{Wdef}). This equation is well-defined, since both sides belong to $H^0(S, K_S(D)^{\otimes(m+1)})$. It can be shown that $Y_\Phi$ is smooth and has a holomorphically trivial canonical class (Propositions~\ref{Ysmooth}, \ref{Ycan}).

In the case $m = 1$, $Y_\Phi$ coincides with Calabi-Yau manifolds described by Smith in \cite{S13}. We conjecture that many parts of his paper should also generalize to $Y_\Phi,\ m\geq 1$ which we outline in more detail below. We refer the reader to the original paper for more details.

\medskip
Let $\mathcal D(Q_{\mathcal T,m}, W)$ denote the $3d$ Calabi-Yau category defined via Ginzburg algebra $\Gamma(Q_{\mathcal T, m}, W)$ (see \cite{G06}). Up to $A_\infty$-equivalence it depends only on the right-equivalence class of potential $W$. Moreover, categories corresponding to different choices of triangulation $\mathcal T$ are equivalent by results of Fock and Goncharov \cite{FG03} together with Keller and Yang \cite{KY09}. In the former paper it was shown that quivers $Q_{\mathcal T, m}$ for different triangulations are related by sequences of mutations (see (\ref{seed_mut})); the latter paper proves a general result that categories from quivers with potentials related by mutations are equivalent. Based on this remark, we can fix without loss of generality a triangulation of $\SSS$ and study right-equivalence classes of potentials on a fixed quiver.

For a generic choice of parameters determining the cohomology class $[\omega]\in H^2(Y_\Phi, \RR)$ of a K\"ahler form on $Y_\Phi$ and for every choice of background class $b\in H^2(Y,\ZZ_2)$, there is a well-defined $A_\infty$-category $\mathcal{F}(Y_\Phi,b)$, the $b$-twisted strictly unobstructed Fukaya category.

\begin{conjecture}\label{conject}
For every admissible choice of the class of K\"ahler form $[\omega]\in H^2(Y_\Phi,\RR)$ there exists a potential $W$ on $Q_{\mathcal{T}, m}$ defined up to a right-equivalence, and a fully-faithful embedding
\begin{equation}\label{Smithequivalence}
\mathcal{D}(Q_{\mathcal{T},m}, W)\hookrightarrow \mathcal{DF}(Y_\Phi, b_0)
\end{equation}
for certain background class $b_0\in H^2(Y_\Phi, \ZZ)$.

\end{conjecture}

\medskip
Smith proves this result for $m=1$ by taking a full subcategory $\mathcal A(\mathcal T)\subset \mathcal{F}(Y_\Phi, b_0)$ generated by a collection of Lagrangian $3$-spheres, and then showing that$$\mathcal{DA}(T) \simeq \mathcal{D}(Q_{\mathcal T, 1}, W).$$
From the geometric point of view, the quiver describing the subcategory comes from the intersection form of a particular collection of Lagrangian $3$-spheres. Namely, every oriented Lagrangian sphere corresponds to a vertex of a quiver, and the number of arrows between two vertices stands for the intersection number of associated spheres.

A crucial argument in Smith's proof of the equivalence~\ref{Smithequivalence} is that by finite-dimensionality result it is enough to know only finitely many multiplicative constants to recover a category up to equivalence. Observe that for $m=2$ this part of the conjectural generalization is already covered by our result about right-equivalence classes of potentials on $Q_{\TT,2}$. It remains to identify corresponding geometric objects in $Y_\Phi$ (e.g. K\"ahler form, collection of Lagrangian spheres, pseudo-holomorphic disks).

We propose a mechanism to generate $3$-spheres representing non-trivial homology classes in $Y_\Phi$~(Section \ref{spheres}). Roughly speaking, such a sphere can be constructed from any $1$-cycle on the spectral curve $\Sigma\subset \Tot(K_S(D))$ that is contractible when projected to $S$. By definition, spectral curve is given by equation:
\begin{equation}
  b^{m+1} + \varphi_2b^{m-1} + \ldots + \varphi_{m+1} = 0.
\end{equation}

To define these spheres more explicitly we need a slightly different point of view on $Y_\Phi$. Following Smith, we have described Calabi-Yau manifold as a quadric fibration with isolated singular fibers over complex curve $S$. Alternatively, it is useful to view $Y_\Phi$ as a conic fibration over a complex surface. To motivate this construction we observe that there is a natural factorization
\begin{equation}
  \begin{tikzcd}
    Y_\Phi \arrow{r}{\be'} \arrow[swap]{dr}{\pi} & \Tot{K_S(D)} \arrow{d}{\kappa'} \\
    & S
  \end{tikzcd}
\end{equation}
induced by the projection of $\mathcal W$ onto the middle direct summand in \ref{Wdef}. Generic fiber of $\be'$ is a conic. Along the spectral curve $\Sigma$ and away from $\kappa'^{-1}(D)$ it degenerates into a simple crossing $\{ac = 0\}\subset \CC^2$. Moreover, for $p\in D$ the $\be'$-preimage of $\kappa'^{-1}(p)$ consists of $m+1$ distinct planes each of which collapses to a point under $\be'$ ($m+1$ planes for each point in $D$).

We observe that in the diagram above $\Tot K_S(D)$ can be replaced by its blow-up $T_\Phi$ at points of intersection $\Sigma\cap \kappa'^{-1}(D)$. Then the projection of $3d$ Calabi-Yau manifold becomes a conic fibration degenerating to simple crossing along the spectral curve. The image of the projection in the blow-up is the complement to the vertical fibers over $D$.

Thus, any embedding of a closed topological $2$-disk
$$f: B^2 \longrightarrow T_\Phi$$
such that $f(\partial B^2)\subset \Sigma$ gives rise to a topological $3$-sphere in $Y_\Phi$. Indeed, over interior points of the disk the fiber of $Y_\Phi\rightarrow T_\Phi$ is isomorphic to a non-degenerate conic isomorphic to $\CC^\times$ and over the boundary it degenerates to the intersection of two complex lines $\{ac = 0\}$. In particular, the circle generating $H_1(\CC^\times)$ for a generic fiber shrinks to a point. Choosing a continuous family of such degenerating circles over $f(B^2\setminus \partial B^2)$ gives the desired $3$-sphere.

In particular, a more detailed analysis of the topology of $T_\Phi$ shows that there are such spheres corresponding to loops on $\Sigma$ with contractible image in $C$. Under this correspondence the intersection form for $3$-spheres in $Y_\Phi$ translates to intersection form for loops on $\Sigma$.

\smallskip
Next step is to find a collection of $3$-spheres such that the intersection form corresponds to a quiver $Q_{\TT,m}$. This step would rely on the geometry associated to $\Phi$ which may be a complicated topic in general. To that end, we check the consistency of the conjecture on the level of ranks of homology groups. We use Decomposition theorem to compute (co)homology groups of $Y_\Phi$ (Proposition~\ref{topology}). In particular, we see that the rank of the most complicated term $H_3(Y_\Phi, \QQ)$ equals the number of vertices of $Q_{\TT,m}$:
\begin{proposition}
$$
\operatorname{rk}H_3(Y_\Phi, \QQ) = (m^2-1)(2g-2+d)
$$
\end{proposition}
\noindent This agrees with Smith's conjecture that collections of Lagrangian spheres must generate Fukaya categories in appropriate sense.

\smallskip
One can strengthen this claim slightly using Goncharov's notion of ``topological spectral cover'' \cite{G16}. Under certain assumptions about the relationship between topological spectral cover and holomorphic spectral curve $\Sigma$, the construction of Goncharov would imply that there is a collection of loops in $\Sigma$ such that pairwise intersections are indeed described by the quiver $Q_{\mathcal{T},m}$. Thus, in this context it would be natural to expect that the candidate for Smith's subcategory $\mathcal{DA}(T)$ is generated by Lagrangian spheres associated to collection of loops provided by Goncharov's topological spectral cover.

\subsection{acknowledgements} I am very grateful to Alexander Goncharov for introducing me to the subject and sharing innumerous ideas and insights that made this work possible. I have benefitted a lot from detailed explanations of Zhiwei Yun.


\section{Cluster structure on the moduli spaces of decorated/framed local systems.}

\subsection{Decorated surfaces.} Throughout the paper $\mathbb{S}$ denotes a compact oriented surface of genus $g$ with a set $\{x_1, \ldots x_d\}$ of marked points (also called punctures), such that $d > 0$ and $2 - 2g - d < 0$ (negative Euler characteristic requirement). We refer to this this object as \emph{a marked surface}.

To define various structures related to a marked surface we need to fix its triangulation, this requires following definitions. An ideal arc on $\SSS$ is a non-selfintersecting curve with endpoints at punctures $\{x_1, \ldots x_d\}$, considered up to isotopy relative to these markings. We demand that ideal arcs are not contractible loops on $\SSS-\{x_1, \ldots x_d\}$. A triangulation of $\SSS$ formed by ideal arcs is called an \emph{ideal triangulation}. More precisely, it is a maximal collection of pairwise different non-intersecting ideal arcs on $\SSS$. It is easy to see that the complement to all arcs of an ideal triangulation consists of triangles.

\begin{figure}
	\centering
	\includegraphics{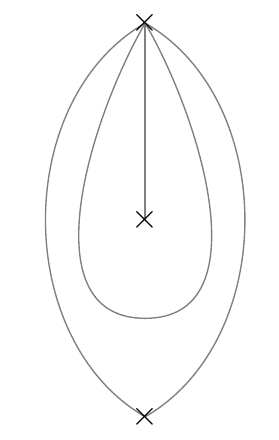}
	\captionof{figure}{Self-folded triangle}
	\label{self-folded}
\end{figure}

For our main results it is important that ideal triangulation is sufficiently generic. For example, in some degenerate cases it may occur that an ideal arc forms two sides of a triangle. Then the triangle is \emph{self-folded} (shown in Figure \ref{self-folded}). This situation is eliminated by the requirement that ideal triangulations in this paper satisfy following additional conditions:

\noindent\begin{minipage}{.01\textwidth}
\begin{equation}\label{triangulationcond}
\end{equation}
\end{minipage}%
\begin{minipage}{.95\textwidth}
\begin{itemize}
\item there are no arcs with coinciding end-points;
\item every vertex has valence at least $3$.
\end{itemize}
\end{minipage}

\noindent It can be shown that if $g > 0$ and $d \geq3$, or $g = 0$ and $d\geq 4$, then marked surface possesses an ideal triangulation satisfying these conditions.


It is well-known that the number of ideal arcs in any ideal triangulation of a marked surface of genus $g$ with $d$ punctures equals $6g-6+3d$. Furthermore, there is a simple way to pass between different ideal triangulations. Namely, take an edge~$e$ which is not a folded edge of a self-folded triangle. Then~$e$ is a diagonal of the quadrilateral formed by two triangles sharing~$e$, and one can replace this edge by another diagonal, thus passing to a new triangulation. We say that this move is a \emph{flip} of the initial triangulation at edge~$e$. Any two ideal triangulations are related by a sequence of flips, however such sequence is not unique.

\subsection{Cluster structure on algebraic varieties.}\label{clustervar} The main object of our study is moduli space of framed/decorated local systems for marked surfaces as defined in \cite{FG03}. For the purpose of this paper we don't need most of geometric definitions, since we are mainly interested in certain properties of explicit quivers arising from cluster structures on moduli spaces. For more material on cluster structure in this context we refer reader to the original paper of Fock and Goncharov.

Of particular relevance for us is the discovery that for semisimple Lie groups of type $A_m$, moduli spaces of decorated/framed local systems on $\SSS$ form a \emph{cluster ensemble} (\emph{cf.} \cite{FG03p2}). The reader not interested in cluster varieties perspective on the subject can freely pass to Section~\ref{QPsec}, where we discuss purely quiver-based part of the story.

A cluster ensemble is a pair of spaces $(A, X)$ with a canonical map $p:A\longrightarrow X$, where $A$ (resp. $X$) carries cluster $\mathcal{A}$- (resp. $\mathcal{X}$-) structure. This means that these varieties are covered by collections of Zariski open split algebraic tori $\mathbb{G}_m^n$, with transition maps prescribed by certain algebraic rules.

To make it precise we need a series of definitions from \cite{FG03}, that are simplified slightly for our discussion:
\begin{definition}
A \textbf{seed} is a pair $\ii = (I, \varepsilon_{ij})$, where $I$ is a finite set of vertices and $\varepsilon_{ij}$ is a skew-symmetric integer-valued function on $I\times I$.
\end{definition}

\noindent \emph{Remark.} In many papers it can be found that a seed is defined as a quadruple $(I, J, \varepsilon, d)$. However, in our case $J$ and $d$ components are trivial: the set of frozen vertices $J\subset I$ is empty, and the symmetrizer function~$d:I\rightarrow \ZZ$ is identically $1$. Thus, we usually omit this additional data to simplify our notation.

\medskip
For every seed $\textbf{i} = (I, \varepsilon_{ij})$ and $k \in I$ one can produce  a new seed called a \emph{mutation} of $\textbf{i}$ in the direction $k$, denoted by $\mu_k (\ii) = (I', \varepsilon_{ij}')$. By definition, finite sets $I$ and $I'$ coincide, and $\varepsilon_{ij}'$ is defined by the following formulas:

\begin{equation}\label{seed_mut}
\varepsilon_{ij}'=
\begin{cases}
-\varepsilon_{ij} & \text{if}  \ \ \ k \in \{i,\ j\} \\
\varepsilon_{ij} + \frac{|\varepsilon_{ik} |\varepsilon_{kj} + \varepsilon_{ik} |\varepsilon_{kj}|}{2} & \text{if}  \ \ \ k \notin \{i,\ j\}
\end{cases}
\end{equation}

A seed $\ii$ gives rise to two split algebraic tori $A_\ii$ and $X_\ii$ (sometimes called cluster tori), both isomorphic to $\mathbb{G}_m^{I}$. We denote coordinates on these tori by $(A_i)_{i\in I}$ and $(X_i)_{i\in I}$ respectively. To any mutation $\mu_k$ there corresponds a rational isomorphisms of split algebraic tori: $\mu_k: A_{\ii} \longrightarrow A_{\mu_k(\ii)}$ and $\mu_k: X_{\ii} \longrightarrow X_{\mu_k(\ii)}$ given by the following formulas:

\begin{equation}\label{a_mut}
\begin{cases}
\mu_k^* A_i= \frac{\prod\limits_{j: \varepsilon_{ij}>0}A_j^{\varepsilon_{ij}} + \prod\limits_{j: \varepsilon_{ij}<0}A_j^{-\varepsilon_{ij}}}{A_k} & \text{if}  \ \ \ i = k\\
\mu_k^* A_i  = A_i& \text{if}  \ \ \ i \neq k
\end{cases}
\end{equation}

\begin{equation}\label{x_mut}
\begin{cases}
\mu_k^* X_i= X_k^{-1} & \text{if}  \ \ \ i = k\\
\mu_k^* X_i  = X_i(1+X_k^{-\operatorname{sgn}(\varepsilon_{ik})})^{-\varepsilon_{ik}} & \text{if}  \ \ \ i \neq k
\end{cases}
\end{equation}

We say that an algebraic variety $A$ (resp. $X$) has a cluster $\mathcal{A}$- (resp. $\mathcal{X}$-) {structure} if there is a collection of seeds and a collection of open embeddings of corresponding cluster tori $A_\ii \longrightarrow A$ (resp. $X_\ii \longrightarrow X$), such that whenever two seeds are related by mutation, the corresponding transition map is given by formulas \ref {a_mut} (resp. \ref{x_mut}).

In other words, exhibiting a cluster structure on a variety amounts to providing a collection of open tori, such that gluing maps between them are controlled by seeds attached to every torus. For every mutation one has to apply transition formulas given above to pass to another chart, and change the seed according to the rule \ref{seed_mut}.

\noindent\emph{Remark.} Note that in principle, one can mutate a seed in any direction $k\in I$, this in turn will give rise to a rational map of cluster tori. However, this leads to an enormous family of seeds. We give a more flexible definition, which allows to chose only specific mutations of seeds (e.~g. we can take a cluster structure with only one seed where no mutations allowed).
\medskip

To define a cluster structure on a variety it is usually convenient to use an alternative language describing seeds via quivers.

\begin{definition}
A \textbf{quiver} is consists of data $(Q_0, Q_1,s,t)$, where first two entries are sets of vertices and arrows respectively; and $s,t: Q_1\longrightarrow Q_0$ are ``source'' and ``target'' maps that send an arrow $a\in Q_1$ to its initial and terminal points in $Q_0$.
\end{definition}

It is clear that there is a one-to-one correspondence between seeds $\ii$ with vertex set $I$ and quivers $Q$  without loops and oriented $2$-cycles with $Q_0 = I$. Concretely, matrix $(\varepsilon)_{i,j\in I}$ describes the number of arrows between $i$ and $j$ taken with sign depending on their direction (our convention is that if $\varepsilon_{ij}>0$ then arrows are directed from $i$ to $j$).

\begin{figure}
	\centering
	\begin{minipage}[b]{.33\textwidth}
		\centering
		\includegraphics{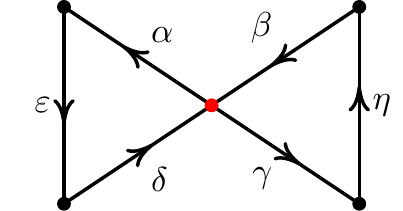}

		\vspace{15pt}
	\end{minipage}%
	\begin{minipage}[b]{.33\textwidth}
		\centering
		\includegraphics{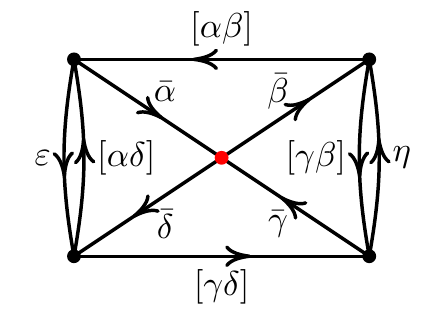}
	\end{minipage}%
	\begin{minipage}[b]{.33\textwidth}
		\centering
		\includegraphics{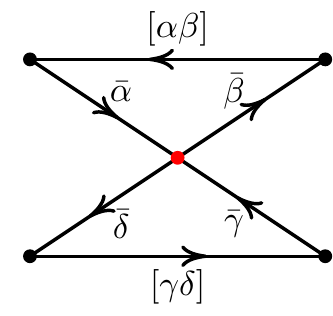}
\end{minipage}
	\captionof{figure}{Example of quiver mutation.}
	\label{quivermut}
\end{figure}

In these terms quiver mutations have nice combinatorial meaning. Let $Q = (Q_0, Q_1,s,t)$ and $\mu_k(Q) = (Q'_0, Q'_1,s',t')$ be two quivers corresponding to seeds $\textbf{i} = (I,\varepsilon_{ij})$ and $\mu_k(\textbf{i})$ related by mutation at vertex $k\in I$. Then vertex sets of $Q$ and $\mu_k(Q)$ coincide and arrow sets of the mutated quiver can be described in terms of $Q$ by the following combinatorial rules (\emph{cf.} \ref{seed_mut}):

\begin{enumerate}
\item Every $\al\in Q_1$ with $s(\al), t(\al)\neq k$ belongs to a mutated quiver $\mu_k(Q)$;
\item If $\al\in Q_1$ adjacent to $k$ (i.e. $s(\al) = k$ or $t(\al) = k$), there is an arrow $\bar{\al}$ of $\mu_k(Q)$ going in reversed direction;
\item For every pair of arrows $\al,\be\in Q_1$ with $s(\al) = k$ and $t(\be) = k$, one arrow $[\al\be]$ from $t(\al)$ to $s(\be)$ is added; and if after this operation there appear arrows between $t(\al)$ and $s(\be)$ in going in opposite directions, they must be cancelled out by deleting corresponding $2$-cycles.
\end{enumerate}
\noindent  Note that the resulting quiver has no loops or oriented $2$-cycles. An example of mutation is given in Figure \ref{quivermut}.
\medskip

Using quiver description, defining a cluster structure on a variety is equivalent to giving a collection of quivers with variables attached to their vertices, and specifying what mutations are allowed between these quivers.

\subsection {Cluster structure on moduli spaces of local systems.}\label{clusterstr} Here we define quivers discovered by Fock and Goncharov in \cite{FG03} arising from the cluster structure on moduli spaces of framed/decorated local systems. Fix $m\geq 1$ (it corresponds to local systems of type $A_m$), in what follows, we construct a quiver $Q_{\TT,m}$ associated to every ideal triangulation $\TT$ of $\SSS$.

In fact, it is convenient to imagine that quivers in this discussion are embedded in $\SSS$. The vertex set of $Q_{\mathcal T, m}$ is described as follows. Consider the standard 2-dimensional simplex $\Delta = \{(x,y,z)\in \RR_{\geq0}^3\mid x+y+z = m+1\}$ and subdivide it into smaller simplexes by $3m$ planes $x = k, y =k, z = k,\  0<k<m+1$ (an example for $m=4$ is shown in Fig. \ref{mtriangulation}). Thus, vertices of these subdivision coincide with integral points $\ZZ^3\cap\Delta$, and they are arranged as vertices of triangles of two types:

\begin{enumerate}\label{triangles}
\item  white (``upward'') triangles with vertices $(a+1,b,c),(a, b+1,c),(a,b,c+1)$ for all triples of integers $(a,b,c)$ with $a+b+c = m$;

\item black (``downward'') triangles $(\bar{a},\bar{b}+1,\bar{c}+1),(\bar{a}+1,\bar{b},\bar{c}+1),(\bar{a}+1,\bar{b}+1,\bar{c})$ for all triples of integers $(\bar{a},\bar{b},\bar{c})$ with $\bar{a}+\bar{b}+\bar{c} = m-1$.
\end{enumerate}

\begin{figure}
	\centering
	\begin{minipage}{.5\textwidth}
		\centering
		\captionsetup{width=1.2\textwidth}
		\includegraphics{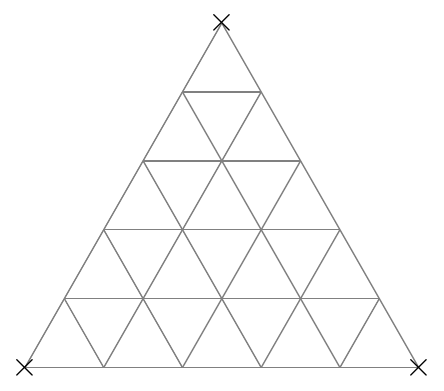}
		\captionof{figure}{$4$-triangulation of $\Delta$}
		\label{mtriangulation}
	\end{minipage}%
	\begin{minipage}{.5\textwidth}
		\centering
		\captionsetup{width=1.2\textwidth}
		\includegraphics{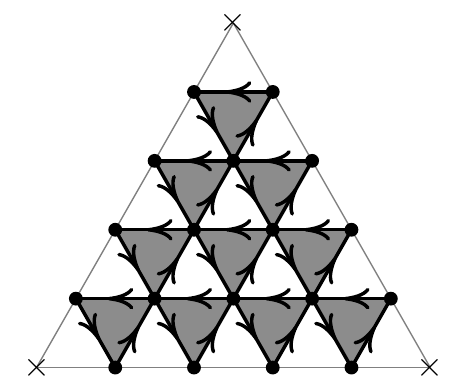}
		\captionof{figure}{Corresponding subquiver of $Q_{\TT,4}$}
		\label{triangle_quiver}
	\end{minipage}
\end{figure}

\noindent We orient edges of every downward triangle so that they go in counterclockwise order and view resulting union of arrows as the standard quiver embedded in $\Delta$ (see Fig \ref{triangle_quiver}). In order to obtain quiver $Q_{\TT,m}$ embedded in $\SSS$, we need to identify every triangle of $\TT$ with $\Delta$ and transport the standard quiver to the surface. Note that all identifications must agree on the edges and induce same orientation on the surface. Thus, there are $m$ vertices of $Q_{\TT,m}$ sitting on every edge of ideal triangulation. Informally one can think of this operation as gluing quiver $Q_{\TT,m}$ from standard quivers sitting on copies of $\Delta$, one copy for each triangle of $\TT$.

\smallskip
\noindent\emph{Remarks:} 1. It can be shown explicitly, that quivers $Q_{\TT,m}$ corresponding to different ideal triangulations are related by sequences of mutations. In fact, since any two such triangulations are related by a sequence of flips, it is enough to find a sequence of mutations that connects quivers for two flipped triangulations. For example, in the case $m = 1$ quivers $Q_{\TT,1}$ have a single vertex siting on every edge, and the sequence of mutations realizing a flip at an edge is just a single mutation in the corresponding vertex. However, for $m > 1$ the situation is more complicated and one needs $\frac{m(m+1)(m+2)}{6}$ mutations to pass between quivers assigned to triangulations related by a single flip.

\noindent 2. Strictly speaking, moduli spaces in this paper are not algebraic varieties but algebraic stacks. We can ignore this difference in this paper by passing to an open part of the moduli spaces.

\noindent 3. Variables attached to vertices of quivers $Q_{\TT,m}$ have a very concrete geometric meaning. They arise from certain coordinates on spaces of configurations of linear subspaces in $(m+1)$-dimensional vector space. Their precise description can be found in \cite{FG03} (or \cite{G16} for a more general class of quivers associated to ideal webs).

\section{Quivers with potential}\label{QPsec}
\subsection{Definition of quivers with potential.} The exposition of this subsection mainly follows \cite{DWZ07}.

Fix a field of characteristic zero $\mathbf{k}$ and a quiver $Q = (Q_0,Q_1, s,t)$. Let $R=\bigoplus_{i\in Q_0}\kk \cdot e_i$ be a commutative algebra formed by $|Q_0|$ idempotents associated to vertices of the quiver (i.e. $e_i\cdot e_j = \delta_{ij}$). Linear space spanned by quiver's arrows $\displaystyle{A = \bigoplus_{\al\in Q_1}\kk \cdot \al}$ has a natural $R$-bimodule structure:
\begin{equation}\label{comprule}
e_i\cdot\al = \delta_{i,t(\al)}\al; \ \ \al\cdot e_j = \delta_{s(a),j}\al
\end{equation}
Associated to any quiver $Q$ is its path algebra $R\langle Q\rangle$, defined as a graded tensor algebra of $A$ over $R$:
\begin{definition}
\textbf{Path algebra} of quiver $Q$ is the graded algebra:
$$R\langle Q\rangle=\bigoplus_{n=0}^\infty \underbrace{A\otimes_R A\ldots\otimes_R A}_{n},$$
with natural multiplication given by tensor product.
\end{definition}
There is a homogeneous basis of path algebra that consists of tensor products of the form $\al_1\otimes\al_2\otimes\ldots\otimes\al_n$ with $\al_j\in Q_1$ (hereafter tensor product signs are omitted). It is natural to call such elements \emph{paths}, multiplication of two paths is simply given by concatenation, whenever paths are composable, and zero otherwise. Note that with our convention~(\ref{comprule}) paths are composed from left to right as usual set theoretic maps. It is convenient to extend source and target maps $s,t$ to all paths in an evident way.

Path algebra possesses natural grading by path length. In many cases we need to work with \emph{completed path algebra} $R\langle\langle Q\rangle\rangle$ where the completion is taken with respect to degree. Elements of this new algebra can be represented as possibly infinite linear combination of paths, where only finitely many paths of given degree occur.

\smallskip

Next we pass to the definition of potentials:
\begin{definition}\label{potdef}
Let $R\langle\langle Q\rangle\rangle_{\operatorname{cyc}}$ be the span of all cyclic paths (i.e. paths $P$ with $s(P) = t(P)$), considered up to cyclic shift:
$\al_1 \al_2\ldots \al_n \longleftrightarrow \al_n\al_1\ldots \al_{n-1}.$

Elements of this space are called \textbf{potentials}. A quiver with potential is a pair $(Q,W)$ with $W\in R\langle\langle Q\rangle\rangle_{\operatorname{cyc}} $.
\end{definition}
\noindent \emph{Remark:} A slightly more conceptual way to define potentials is by taking the space of functionals on the quotient $R\langle Q\rangle/ [R\langle Q\rangle,R\langle Q\rangle]$. This description clarifies the origin of cyclic shifts.

\smallskip

Endomorphisms of completed path algebra over $R$ induce action on the subspace of cyclic paths and give rise to a notion of \emph{right-equivalence} of quivers with potentials that we describe below. It is easy to see that any $R$-bimodule homomorphism $\ph:A\longrightarrow R\langle\langle Q \rangle\rangle$ can be extended to endomorphism of the whole completed path algebra. Conversely, any endomorphism of the algebra can be restricted to arrow space $A$. Action of endomorphism $\ph$ on $A$ can be decomposed as a sum $\ph = \ph^{\diag}+\ph^{\un}$, where $\ph^{\diag}: A \rightarrow A$ is the degree preserving part, and $\ph^{\un}$ maps the arrow space to the subspace of $R\langle\langle Q \rangle\rangle$ spanned by paths of length at least $2$.

We cite Proposition 2.4 from \cite{DWZ07}, slightly adapted to our notations:
\begin{proposition}\label{phi}
Any pair $(\ph^{\diag}, \ph^{\un})$ of $R$-bimodule homomorphisms $\ph^{\diag}:A \longrightarrow A$ and $\ph^{\un}: A \longrightarrow R\langle\langle Q \rangle\rangle_{\geq 2}$ gives rise to a unique endomorphism $\ph$ of the completed path algebra, such that $\ph |_R = \operatorname{id}$ and $\ph |_A = (\ph^{\diag}, \ph^{\un})$. Furthermore, $\ph$ is an isomorphism if and only of $\ph^{\diag}$ is an $R$-bimodule isomorphism.
\end{proposition}

\noindent Clearly, any endomorphism of the completed path algebra induces a map on its cyclic part $R\langle\langle Q \rangle\rangle_{\operatorname{cyc}}$.

\begin{definition}
Two potentials $W$ and $W'$ on quiver $Q$ are called \textbf{right-equivalent} if there is an automorphism of the completed path algebra $\ph$ such that $\ph(W) = W'$.
\end{definition}

Denote the group of automorphisms of $R\langle\langle Q \rangle\rangle$ by $\GG$. There are two subgroups, described in terms of the decomposition $\ph = (\ph^{\diag}, \ph^{\un})$ from Proposition \ref{phi}:
\begin{itemize}
\item $\GG^{\diag}$ is the subgroup of automorphisms with $\ph^{\un} = 0$;
\item $\GG^{\un}$ is the subgroup of automorphisms with $\ph^{\diag} = \operatorname{id}$.
\end{itemize}
One has a semi-direct product presentation: $\GG=\GG^{\un}\rtimes\GG^{\diag}$. We call elements of the first and second factor ``diagonal'' and ``unitriangular'' right-equivalences respectively.

\subsection{Cutting operations} In this subsection we assume that all arrow spaces of $Q$ are at most one dimensional. Definitions and constructions below are borrowed from from \cite{A17}:
\begin{definition}\label{chord}
A \textbf{chord} in a cycle $\al_1\al_2...\al_n$ of a quiver $Q$ is a triple $(\be,i,j)$, where $\be$ is an arrow in $Q$ such that $t(\be) = t(\al_i)$ and $s(\be) = s(\al_j)$ with $i\neq j$. The term ``chordless'' is used for cycles without chords.
\end{definition}

There are two important ``cutting operations'' defined for every chord of a cycle:
\begin{definition}\label{cut}
Let $(\be, i, j)$ be a chord of cycle $C =  \al_1\ldots \al_n$, then:
\begin{itemize}
\item $\cut_{\be}C = \al_1...\al_{i-1}\be\al_{j+1}...\al_n$  is the operation that produces a new cycle;
\item $b\Rightcircle C = \al_i\ldots \al_j$ produces a new path form $s(\be)$ to $t(\be)$.
\end{itemize}
\end{definition}
The first operation can be extended to multiple chords. If $(\be_k, i_k, j_k)_{k=1...N}$ is a collection of chords in cycle $C$ and the sequence of arrows $(\al_{i_1}, \al_{j_1},\al_{i_2},\al_{j_2},\ldots \al_{i_N}, \al_{j_N})$ respects cyclic order of arrows of $C$, we can perform cutting along the whole collection of chords at once, getting a cycle denoted $\cut_{\be_1...\be_N}C$. Collections of chords for which such cyclic order condition is satisfied will be called \emph{nonintersecting chords}. We also use a convention $\cut_\varnothing C = C$ for cutting along empty collection of chords (empty collection is considered nonintersecting). When cyclic indices in a chord are clear from the context, we usually omit them and denote a chord simply by a single arrow $\be$.

For any $U\in R\langle\langle Q\rangle\rangle$ and any path $P$ denote by $U[P]$ the coefficient of $P$ in the expression of $U$. The following Lemma is a direct consequence of constructions and explains the significance of definitions above:
\begin{lemma}\label{cutlemma}
Let $(Q,W)$ be a quiver with potential, and let $C$ be any cycle in $Q$. For any right-equivalence $\varphi = (\ph^{\diag}, \ph^{\un})$ with trivial diagonal part $\ph^{\diag} = \operatorname{id}$, one has:
\begin{equation}
\varphi(W)[C] = \sum_{\be_1,\be_2,\ldots \be_N} W[\cut_{\be_1\be_2\ldots \be_N}C]\prod_{k = 1}^N\varphi(\be_k)[\be_k\Rightcircle C],
\end{equation}
where the sum is taken over all collections $(\be_1,\be_2,\ldots \be_N)$of nonintersecting chords of $C$.
\end{lemma}

It is easy to generalize this Lemma to arbitrary right-equivalences with non-trivial diagonal part. In particular, it follows that the effect of any right-equivalence $\ph$ on a coefficient of a chordless cycle in a potential is determined by diagonal part of $\ph$. This observation is used to show equivariance of the diagram \ref{equivariant1}.


\section{Spaces of primitive potentials for quivers $Q_{\TT,m}$}\label{primpotsec} In Section~\ref{clusterstr} we recalled a class of quivers $Q_{\TT,m}$ that describes cluster structure on the moduli spaces of framed local systems on a surface discovered by Fock and Goncharov \cite{FG03}. In this section we introduce a notion of a ``primitive'' potential on a quiver and study their equivalence-classes. By results of \cite{DWZ07} (and for categorical counterpart by \cite{KY09}) it is enough to consider a single quiver in its mutation class, since right-equivalence is a relation respected by mutations. Hence without loss of generality we can asssume that ideal triangulation $\TT$ is fixed and satisfies assumptions \ref{triangulationcond} (see Remarks in the end of Section~\ref{clusterstr}). To shorten notations we sometimes  write $Q$ instead of $Q_{\TT,m}$ throughout this section.

\subsection{Three bijections on the arrow set of $Q_{\TT,m}$} It will be useful to introduce three bijections $l,r,f$ (from ``left'', ``right'' and ``forward'') to describe various manipulations with elements of path algebras $R\langle\langle Q_{\TT,m}\rangle\rangle$ in this paper.

Recall that $Q_{\TT,m}$ can be naturally embedded in $\SSS$ and its complement is a union of regions which are contractible disks. Boundary of any such region is an oriented cycle of $Q$, hence every region can be painted black (resp. white) if the boundary cycle is oriented counterclockwise (resp. clockwise). Any arrow $\al \in Q_1$ in the arrow set of the quiver belongs to a unique black and a unique white region. Two bijections $l,r: Q_1\rightarrow Q_1$ map $\al$ to its successor in the corresponding black and white regions. Note that every black region is, in fact, a triangle. In particular, $l^2(\al)l(\al)\al$ is a cycle representing the boundary of the black triangle through arrow $\al$.

\begin{figure}
	\centering
	\begin{minipage}[b]{.5\textwidth}
		\centering
		\captionsetup{width=1\textwidth}
		\includegraphics{Lcycles_rev.pdf}
		\captionof{figure}{Cycles $L^{(k)}_p,\ k =1,2,3$}
		\label{Lcycles}
	\end{minipage}%
	\begin{minipage}[b]{.5\textwidth}
		\centering
		\captionsetup{width=1\textwidth}
		\includegraphics[width=1\textwidth]{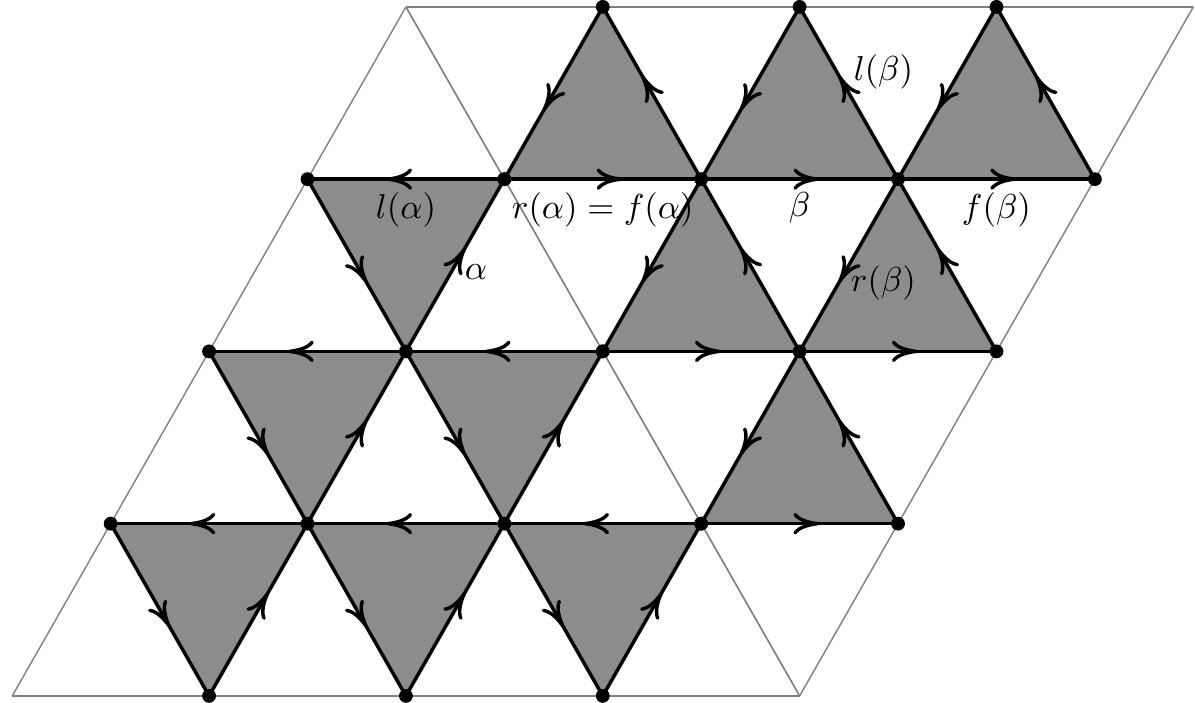}
		\captionof{figure}{Example of maps $l,r$ and $f$}
		\label{lrf_maps}
	\end{minipage}
\end{figure}

To define ``forward'' bijection $f$ we introduce an important class of cycles on quivers~$Q_{\TT,m}$. Choose one of the marked points $p\in \SSS$ and a number $1\leq k\leq m$. There are $m$ cycles $L_p^{(k)}$ in $Q$, geometrically they form $m$ disjoint clockwise paths around $p$. To describe them explicitly let $\tau_1,\ldots \tau_{\val(p)}$ denote the collection of all triangles of $\TT$ adjacent to $p$ written in their clockwise order ($\val(p)$ stands for valence of $p$ in the triangulation). As explained in Section~\ref{clusterstr} there is a standard subquiver embedded in every triangle of $\TT$. This subquiver consists of three families of arrows parallel to every side of a triangle. In particular, each of the triangles $\tau_1,\ldots \tau_{\val(p)}$ contains $m$ paths parallel to the opposite side of $p$. Enumerate these paths from $1$ to $m$ starting from the puncture. Then it is easy to see that $k$-th paths for triangles around the puncture are composable and form a clockwise cycle of length $k\val(p)$. This cycle is denoted $L_p^{(k)}$ (Fig. \ref{Lcycles}). Properties of these cycles are listed below.
\begin{lemma}\label{chordless_classification}
\begin{enumerate}[label=(\alph*)]
\item Cycles $L_p^{(k)}$ are chordless;
\item $L_p^{(1)}$  coincides with the boundary of white region containing $p$;
\item Any arrow $\al$ of $Q_{\TT,m}$ belongs to a unique cycle of the form $L_p^{(k)}$.\end{enumerate}
\end{lemma}
\begin{proof} Statement (a) follows from the fact, that triangulation $T$ has no ideal arcs with coinciding end-points (see~\ref{triangulationcond}). Other statements are immediate from the construction.
\end{proof}

The last part of this lemma allows us to define the third bijection $f: Q_1\rightarrow Q_1$. For any arrow $\al$, by definition $f(\al)$ is the successive arrow in the unique cycle $L_p^{(k)}$ containing $\al$. With this definition, for instance, we can write:
$$
L^{(1)}_p = f^{\val (p)}(\al)\ldots f^2(\al)f(\al)\al
$$
for some arrow $\al$.

\smallskip
Note that boundaries of black and white regions are chordless cycles in $Q$. In the next proposition we show that together with cycles $L_p^{(k)}$ this exhausts the set of all chordless cycles of quivers $Q_{\TT,m}$.
\begin{proposition}\label{chordless}
Let $\TT$ be an ideal triangulation of a marked surface $\SSS$ satisfying (\ref{triangulationcond}). Then the set of chordless cycles of $Q_{\TT,m}$ is the union of the following collections:
\begin{enumerate}[label=(\roman*)]
\item Boundary cycles of black regions (oriented counterclockwise);
\item Boundary cycles of white regions (oriented clockwise);
\item Cycles $L^{(k)}_p$, where $p$ is a puncture of $\SSS$ and $k = 1,2,\ldots m$ (oriented clockwise).
\end{enumerate}
The only intersection between these collections of cycles $L_p^{(1)}$.
\end{proposition}

\begin{figure}
	\centering
	\begin{minipage}[b]{.5\textwidth}
		\centering
		\captionsetup{width=1.2\textwidth}
		\includegraphics{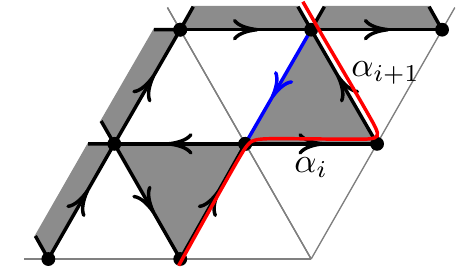}
		\captionof{figure}{Chord for $C = ...l(\al_i)\al_i...$}
		\label{fchord}
	\end{minipage}%
	\begin{minipage}[b]{.5\textwidth}
		\centering
		\captionsetup{width=1.2\textwidth}
		\includegraphics{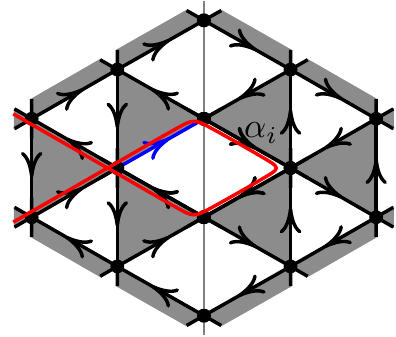}
		\captionof{figure}{Chord for $C = ...r^2(\al_i)r(\al_i)\al_i...$}
		\label{gchord}
	\end{minipage}
\end{figure}

\begin{proof} Take a chordless cycle $C = \al_1\al_2\ldots \al_n$ of $Q$ that is not a boundary of a black or white region. We need to show that $C$ has to be of form $L^{(k)}_p$ for some puncture $p$ and $1\leq k \leq m$. The key observation for the proof is that $\al_{i-1}\neq l(\al_{i}) \ \forall i \in \ZZ/ n\ZZ$. Indeed, otherwise the cycle would have a chord $(f^{-1}(\al_i), i-2, i+1)$ formed by two sides of a black triangle (see Fig. \ref{fchord}). Now consider any arrow $\al_i$, we claim that $\al_{i-1} = f(\al_i)$, this will be clearly enough to finish the proof of the proposition.

A vertex of $Q_{\TT,m}$ that belongs to two triangles of $\TT$ is called \emph{edge vertex}. We consider different cases for whether the source and the target of $\al_i$ are edge vertices:

\begin{enumerate}
\item If both $s(\al_i)$ and $t(\al_i)$ are not edge vertices, then $(r^{-1}(\al_i), i-2, i+1)$ would be a chord unless $\al_{i-1} = f(\al_i)$ (cycle $C$ includes two sides of a white triangle); hence $\al_{i-1} = f(\al_i)$ necessarily.

\item If $s(\al_i)$ is an edge vertex and $t(\al_i)$ is not, then $\al_{i-1} \neq r(\al_{i})$. Otherwise $C$ would contain a fragment (see Fig \ref{gchord}):
$$\ldots r^2(\al_i)r(\al_i)\al_i\ldots$$
If this was the case, $(r^{-1}(\al_i), i-3, i+1)$ is a chord or $d = 4$. For $d = 4$ we get $C = r^3(\al_i)r^2(\al_i)r(\al_i)\al_i$, which contradicts the initial assumption that $C$ is not a boundary of a white fragment;

\item If $t(\al_i)$ is an edge vertex then $\al_{i-1} =  f(\al_i) = r(\al_i)$ and there is nothing left to prove.
\end{enumerate}

We have shown that  $\al_{i-1} = f(\al_i) \forall i\in \ZZ/n\ZZ,$ it follows that $C$ has form $\left(L^{(k)}_p\right)^N$ for some power $N>0$. A chordless cycle cannot pass through the same vertex twice: if $s(\al_i) = s(\al_j), i\neq j$ then $(\al_j, i, j)$ is a chord of $C$. Hence $N = 1$ and the proposition is proven.
\end{proof}

\subsection{Right-equivalence classes of primitive potentials on $Q_{\TT,m}$}\label{primdefsec} The following definition is general and works for any quiver, not necessarily $Q_{\TT,m}$.

\begin{definition}\label{primdef}
\begin{enumerate}[label=(\roman*)]
  \item A potential $W$ on a quiver $W$ is \textbf{primitive} if it is a linear combination of chordless cycles and every chordless cycle appears in $W$ with nonzero coefficient.
  \item \textbf{Primitive part} of a potential $W$ on $Q$ is the projection of $W$ to the subspace R$\langle\langle Q\rangle\rangle_{\operatorname{prim}}\subset R\langle\langle Q\rangle\rangle_{\operatorname{cyc}}$ spanned by all chordless cycles.
  \item A potential on $Q$ is \textbf{generic} if its primitive part is a primitive potential.
\end{enumerate}
\end{definition}

\noindent Last part just says that in the expression of generic potential every chordless cycle appears with nonzero coefficient. For quivers $Q_{\TT,m}$ Proposition~\ref{chordless} makes this definition explicit.

Our goal in this subsection is to study the effect of right-equivalences on primitive potentials. It is easy to see, that the space spanned by chordless cycles is not preserved by arbitrary right-equivalences acting on $R\langle\langle Q\rangle\rangle$. However, if the space of primitive potentials is viewed as a quotient- rather than a sub-space, the equivariance can be restored. This is made explicit in Section~\ref{strgensec}.

Recall group $\GG$ of all right-equivalences of the quiver $Q_{\TT,m}$ and its semi-direct product presentation $\GG=\GG^{\un}\rtimes\GG^{\diag}$. Since all arrow spaces in $Q_{\TT,m}$ are one-dimensional $\GG^{\operatorname{diag}}$ is isomorphic to $(\kk^{\times})^{|Q_1|}$. Clearly, diagonal subgroup preserves the subspace of primitive potentials. Moreover, there is a group epimorphism:
\begin{equation}
\begin{split}
\GG \longrightarrow& \ \GG^{\operatorname{diag}}\\
\varphi = (\varphi^{\diag},\varphi^{\un}&) \mapsto \varphi^{\diag}
\end{split}
\end{equation}
\noindent where $(\varphi^{\diag},\varphi^{\un})$ is the decomposition from Proposition \ref{phi}. Using this and Lemma \ref{cutlemma} one sees that to determine the action of a right-equivalence $\varphi$ on a primitive part of a potential, it suffices to consider only its diagonal part $\varphi^{\diag}$.

This observation allows us to describe equivalence classes of primitive potentials modulo right-equivalences. In fact, to make this more precise we have to consider only right-equivalences that preserve the subspace $R\langle\langle Q\rangle\rangle_{\prim}$. As explained above to it is enough to consider part of the action through the quotient $\GG \twoheadrightarrow \GG^{\diag}\simeq(\kk^{\times})^{|Q_1|}$. Note however, that there may exist non-trivial right-equivalences inducing trivial action on the span of chordless cycles.
\medskip

It is convenient to state the answer in terms of a natural 2-dimensional CW complex associated to $Q_{\TT,m}$. We are going to recognize equivalence classes of primitive potentials as the second cohomology group of this complex.

Let $\mathcal{C}(\TT,m)$ be a CW complex whose 1-skeleton is identified with $Q_{\TT,m}$, and that has a $2$-cell attached along every chordless cycle. By definition, gluing maps for $2$-cells are consistent with cycle orientation given by quiver. That is, the chain differential for $\mathcal{C}(\TT,m)$ maps any $2$-cell to the sum of edges (=arrows) on the boundary.

Objects introduced above can be interpreted in terms of $1$- and $2$-cochains on $\mathcal{C}(\TT,m)$. Indeed, any primitive potential on $Q_{\TT,m}$ can be identified with a $2$-cocycle on $\mathcal{C}(\TT,m)$ with coefficients in $\kk^{\times}$. Elements in $(\kk^{\times})^{|Q_1|}$ naturally correspond to $1$-cochains. We have

\begin{proposition}\label{prim}
Space of primitive potentials on $Q_{\TT,m}$ modulo the action of the group of right-equivalences, preserving primitive potentials, is isomorphic to the second cohomology group of $\mathcal{C}(\TT,m)$.

The rank of this group is given by:
$$H^2(\mathcal{C}(\TT,m), \emph{\kk}^{\times})\simeq (\emph{\kk}^{\times})^{d(m-1) + 1}.$$
\end{proposition}
\begin{proof} As mentioned above, the action of a right-equivalence preserving primitive potentials depends only on its diagonal part (see Prop. \ref{phi}). So, the identification of the quotient with the second cohomology is immediate from construction and preceding discussion. To find rank of $H^2(\mathcal{C}(\TT,m), \kk^{\times})$ we apply simple topological considerations.

Namely, note that the subcomplex of $\mathcal{C}(\TT,m)$ obtained by removing $2$-cells associated to cycles $L^{(k)}_p$ with $2\leq k\leq m$ is isomorphic to $\SSS$. Gluing back each of these $2$-cells adds one to the rank of the whole CW complex. There are $d$ punctures and we need to add $(m-1)$ disks for each of them. The result follows.
\end{proof}

\begin{corollary}\label{primcor}
There are $d(m-1) +1$ functions on the space of generic potentials invariant under the action of group of right-equivalences.
\end{corollary}
The proof is an easy consequence of the previous proposition. For explicit construction of functions see definitions of $h_p(W)$ and $h(W)$ \ref{hinv},\ref{hpinv} and discussion of \ref{equivariant1}.

\section{Space of generic potentials on $Q_{\TT,2}$.}

This section is entirely dedicated to the study of potentials on quivers $Q_{\TT,2}$ associated to $PGL_3$ and $SL_3$ local systems on marked surface $\SSS$. Some of the constructions can be generalized to quivers $Q_{\TT,m}$ with arbitrary $m>1$ but usually we do not spell out explicit details.

We introduce a notion of a~\emph{strongly generic} potential defined as a non-vanishing condition for a system of polynomial equations for coefficients of cycles (Definition~\ref{strgendef}). This property is preserved by all right-equivalences and thus one can study the quotient of the space of strongly generic potentials on $Q_{\TT,2}$ modulo the action of the group of right-equivalences. We denote this quotient $\Pot_{Q_{\TT,2}}$.

The main result of this section is Theorem~\ref{mn3} that gives full description of right-equivalence classes of strongly generic potentials on $Q_{\TT,2}$ and in particular shows that $\Pot_{Q_{\TT,2}}$ is finite dimensional. We copy the statement given in the introduction for the convenience of the reader.

\begin{theorem}\label{mn3} The quotient of the space of strongly generic potentials on $Q_{\TT,2}$ modulo the action of the group of right-equivalences has dimension $d+2$.

Its natural projection to the space of equivalence classes of primitive potentials is a fibration over the image with fibers isomorphic to an affine line $\mathbb{A}_{\emph{\kk}}^1$.
\end{theorem}

In the previous section we have shown that primitive part of a potential is an invariant of the action of group of right-equivalences $\GG$. We recall this statement in the beginning of this section and provide a set of coordinates on the quotient space of primitive potentials modulo equivalence relations. As the second part of the statement shows, knowing primitive part of a potential is not enough to reconstruct it up to right-equivalence.

The proof is done in a quite ``uneconomic'' way and requires construction of infinite compositions of right-equivalences (that are well-defined in the completed path algebra). For the first part of the theorem, we show that there is a finite collection of cycles in $Q_{\TT,2}$, such that any strongly generic potential is right-equivalent to a linear combination of cycles from that collection (Definition \ref{reduced}, Proposition \ref{reducepot}). Second part of the theorem is the refinement of this presentation: coefficients of chordless cycles are controlled by Proposition \ref{prim}, however, there is also a finite number of more complicated cycles. Thus, we need to identify further equivalences between them. It turns out that there is a single extra parameter that describes right-equivalence class of a strongly generic potential with a given primitive part.

\smallskip
To simplify parts of our exposition we introduce additional notation. Let $\al$ be and arrow of $Q$, and let $j_1,\ldots j_d$ be a sequence of letters $f,g$ and $h$. Then we by denote $\overleftarrow{j_d...j_1}(\al)$ the path of degree $d+1$ obtained by composing successive compositions of maps$j_i$:
$$
\overleftarrow{j_d...j_1}(\al)=j_d(j_{d-1}...j_1(\al))\ldots j_2(j_1(\al))j_1(\al)\ldots j_2(j_1(\al))j_1(\al)\al
$$
For example, consider an arrow $\al$ that belongs to some cycle $L_p^{(2)}$ of $Q_{\TT,2}$ then $E = \overleftarrow{rrr}(\al)$ is a cycle of degree $4$ that can be naturally associated to an edge of the triangulation $\TT$ (e.g. Fig.~\ref{divEfig}). For that reason we will also refer to $E$ as ``edge cycle''.

\subsection{Definition of strongly generic potentials.}\label{strgensec} Recall by Proposition \ref{prim} that the space of primitive potentials $R\langle\langle Q\rangle\rangle_{\operatorname{prim}}$ modulo diagonal right-equivalences is identified with second cohomology of the complex $\mathcal{C}(\TT,2)$. In case $m = 2$ this group is $\left(\kk^{\times}\right)^{d+1}$. Moreover, there is a natural $\GG^{\operatorname{diag}}$-equivariant projection (\emph{cf.} \ref{equivariant2}):

\begin{center}
\begin{equation}\label{equivariant1}
\begin{tikzcd}
\GG\arrow{r}{\rho}&\GG^{\diag} \\
R\langle\langle Q\rangle\rangle_{\operatorname{cyc}}\arrow[loop, distance = 45pt]  \arrow{r} & \arrow[loop, distance = 40pt]R\langle\langle Q\rangle\rangle_{\operatorname{prim}}
\end{tikzcd}
\end{equation}
\end{center}

\noindent The bottom arrow is just a projection on the subspace of primitive potentials by taking primitive part (see Definition \ref{primdef}).

\medskip
To define strongly generic potentials, we describe particular coordinates on $H^2(\mathcal{C}(\TT,m), \kk^{\times})$ using ratios of coefficients of primitive cycles. More precisely, we construct functions on the space of cochains $C^2(\mathcal{C}(\TT,m), \kk^{\times})$ that descend to the cohomology group. Recall notation $W[C]$ for the coefficient of cycle $C$ in potential $W$.

The first coordinate $h$ on the space of $2$-cochains is given by:
\begin{equation}\label{hinv}
h(W) =\left( \prod_{C \in \{\text{white cycles}\}} W[C]\right)\cdot\left(\prod_{C \in \{\text{black cycles}\}} W[C]\right)^{-1}
\end{equation}
\noindent Each of the remaining $d$ coordinates is naturally associated to each puncture $x_i\in\SSS, i\in 1,...d$. For every puncture $p$, the cycle $L^{(2)}_p$ splits $\SSS$ into two parts: a disk around a $p$ and its complement in $\SSS$. Consider a subcomplex of $\mathcal{C}(\TT,2)$ formed by a cell associated to $L^{(2)}_p$, and the complement to the disk around $p$ (this subcomplex is homotopy equivalent to surface $\SSS$). Define coordinate $h_p$ by:
\begin{equation}\label{hpinv}
h_p(W) =W[L^{(2)}_p]\cdot\left( \prod_{\left\{\substack{\text{white cycles} \\ \text{in the complement}}\right\}} W[C]\right)\cdot\left(\prod_{\left\{\substack{\text{black cycles} \\ \text{in the complement}}\right\}} W[C]\right)^{-1}
\end{equation}

It is immediate from definitions that functions $\{h, h_{x_1},\ldots h_{x_d}$ descend to the second cohomology group (in other words, these functions are $\GG^{\diag}$-invariant) and form there a full set of coordinates.

\begin{definition}\label{strgendef}
Generic potential $W$ on $Q_{\TT,2}$ is called \textbf{strongly generic} if for every marked point $x_i\in \SSS,\ 1\leq i\leq d$
\begin{equation}\label{strgen}
h(W) + (-1)^{\val(x_i)}h_{x_i}(W)\neq 0
\end{equation}
where $\val(p)$ denotes the valence of vertex $p$ in the ideal triangulation $\TT$.
\end{definition}

\begin{lemma}
Strongly generic potentials form a set stable under the action of the group of right-equivalences.
\end{lemma}

\begin{proof}
Since the property of being generic is preserved by right-equivalences, and the condition (\ref{strgen}) is formulated in terms of primitive parts, it is enough to check the statement for the action of $\GG^{\diag}$. But then it is obvious, since coordinate functions $h$ and $h_p$ are invariant under its action.
\end{proof}

For the convenience of the exposition we fix a ``standard'' representative in $\GG^{\diag}$-orbit of a generic potential $W$. Let $W_{\operatorname{prim}}$  be the primitive part of $W$. By Proposition \ref{prim} after possibly rescaling arrow spaces we can assume that all chordless cycles, other than $L_p^{(k)}$, have coefficient one. Thus, $W_{\operatorname{prim}}$ is fully specified by nonzero numbers $v_p^{(k)}$ --- coefficients of cycles $L_p^{(k)}$ in $W$. In other words, the primitive part is given by:
\begin{equation}\label{prim_std}
W_{\operatorname{prim}}  = \sum_{\text{black cycles}} C_b + \sum_{\substack{\text{white cycles,}\\ \text{not } L^{(1)}_p}} C_w + \sum_{\substack{p\in{x_1,\ldots, x_d} \\ 1\leq k\leq m}}v_p^{(k)}L_p^{(k)}
\end{equation}

\noindent This presentation of the primitive part of a potential in its equivalence class is not unique, but it is convenient for our exposition. With this notations the condition~(\ref{strgen}) becomes:
\begin{equation}\label{strgen2}
v^{(1)}_{x_i} + (-1)^{\val(x_i)} v^{(2)}_{x_i} \neq 0
\end{equation}

\subsection {$\Pot_{Q_{\TT,2}}$ is finite dimensional.}
In this subsection we prove the first part of Theorem \ref{mn3}. This is done in many stages, where at every step we construct some right-equivalence transforming potential to a simpler form. We repeatedly use one fundamental idea, which was also used in \cite{GLS13} to prove analogous result in $A_1$-case. Say, we fix a set of ``bad'' cycles, and we want to prove that potential $W$ is right-equivalent to a potential that has no ``bad'' cycles in its expression. For this purpose we exhibit a sequence of right-equivalences $\Phi_n$, such that:

\begin{enumerate}[label=(\alph*)]
\item the smallest degree of a ``bad'' cycle entering $\Phi_n(W)$ with nonzero coefficient is at least $d = d(n)$, for some sequence of numbers $d(n)$ with $\lim_{n\rightarrow \infty}d(n) = \infty$.
\item $W \equiv   \Phi_n(W) \mod R\langle\langle Q\rangle\rangle_{\operatorname{cyc},{d - c}}$ where $c$ is some constant and subscript $d-c$ denotes corresponding graded component.
\end{enumerate}

\noindent If there exists such sequence, then by the second property,  $\lim_{n\rightarrow \infty}\Phi_n(W)$ is well defined in the completed path algebra, and by the first property the limit has no ``bad'' cycles.

\medskip
In what follows we will step by step choose an appropriate class of ``bad'' cycles and construct sequences of right-equivalences eliminating them. For this we need some definitions (see Figures \ref{Loc_cycle_ex} and \ref{Nonloc_cycle_ex}).

\begin{figure}
	\centering
	\begin{minipage}[b]{.5\textwidth}
		\centering
		\captionsetup{width=1.2\textwidth}
		\includegraphics{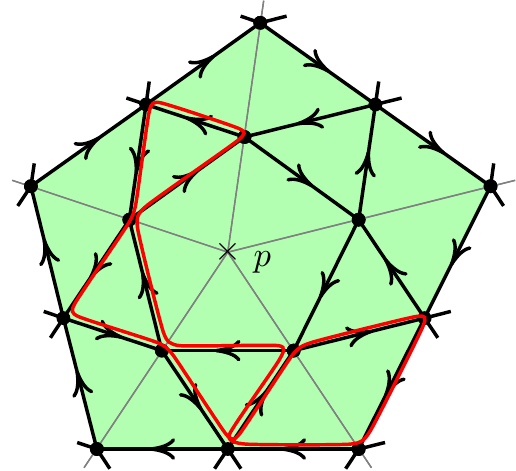}
		\captionof{figure}{$p$-patch $Q_p$ and a local cycle}
		\label{Loc_cycle_ex}
	\end{minipage}%
	\begin{minipage}[b]{.5\textwidth}
		\centering
		\captionsetup{width=1.2\textwidth}
		\includegraphics{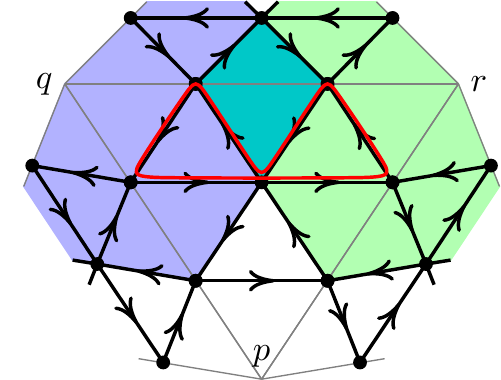}
		\captionof{figure}{Nonlocal cycle}
		\label{Nonloc_cycle_ex}
	\end{minipage}
\end{figure}

Recall that any arrow $\al$ of $Q_{\TT,m}$ belongs to a unique cycle of form $L_p^{(k)}$, in this case we say that $\al$ \emph{is associated to puncture} $p$. We also introduce a class of subquivers of $Q_{\TT,m}$ located around marked points of $\SSS$.

\begin{definition}\label{patch_def}
Let $p$ be a puncture of a marked surface $\SSS$. A $p$-\textbf{patch} is a subquiver $Q_p\subset Q_{\TT,m}$ spanned by all vertices that belong to cycles $L_p^{(k)}, \ 1\leq k\leq m$.
\end{definition}
\noindent Note that a patch contains not only arrows on cycles $L_p^{(k)}$ but also all arrows between them.

\begin{definition}
A cycle in $Q_{\TT,m}$ is called \textbf{local}, if it is contained in a single patch; otherwise, it is called \textbf{non-local}.
\end{definition}

\noindent Recall bijection $l$ on $Q_1$ that sends an arrow to the next one along a black triangle.
\begin{definition}
A cycle $\al_1\al_2... \al_d$ in  $Q_{\TT,m}$ is called \textbf{straight} if it is chordless or: $$l(\al_i)\neq\al_{i-1}\forall i\in \ZZ/ d \ZZ$$
\end{definition}

\medskip
Equipped with these definitions, we can give a roadmap (\ref{roadmap}) of the proof the finite-dimensionality result

Any potential can be split into three parts: $W = W_{\operatorname{prim}}+W_{\operatorname{loc}}+W_{\operatorname{nonloc}}$, where the first summand is the primitive part, the second summand is a linear combination of remaining local cycles, and the last summand includes all non-local terms in the expression for $W$. We use terms ``primitive part'', ``local part'' to refer to a corresponding summand in such decomposition.

The road map for the proof of the finite-dimensionality result of Theorem \ref{mn3} is as follows:

\begin{equation}\label{roadmap}
\begin{split}
W_{\operatorname{prim}}&+W_{\operatorname{loc}}+W_{\operatorname{nonloc}}
\xrightarrow{\text{Prop.\ref{prim}}}
W^{\circ}_{\operatorname{prim}}+W_{\operatorname{loc}}+W_{\operatorname{nonloc}}
\xrightarrow{\text{Prop.\ref{straight}}}
W^{\circ}_{\operatorname{prim}}+W_{\operatorname{loc, str}}+W_{\operatorname{nonloc}}\rightarrow \\
\xrightarrow{\text{Lemm. \ref{onlyL}}}&
W^{\circ}_{\operatorname{prim}}+\sum_{n\geq2,k,i} v_{n,k,i}\left(L_{x_i}^{(k)}\right)^n+W_{\operatorname{nonloc}}
\xrightarrow[\text{strongly generic!}]{\text{Lemm.\ref{removeL}}}
W^{\circ}_{\operatorname{prim}}+W_{\operatorname{nonloc}}
\xrightarrow{\substack{\text{Cor.\ref{nonloc}}\\ \text{Prop.\ref{nononloc}}}}
W^{\circ}_{\operatorname{prim}}+W^{\circ}_{\operatorname{nonloc}}
\end{split}
\end{equation}

Let us elaborate on this scheme here. The first step transforms the primitive part of $W$ to some simpler form (this is done only for the convenience of the exposition). Next two steps are used to simplify the local part of the potential; and then Lemma \ref{removeL} is used to get rid of the local part completely (see also the remark after the formulation of this lemma). These three steps are aggregated in the statement of Corollary \ref{Wloc}. The last step reduces nonlocal part of the potential to the form, where only cycles of degree at most $7$ are allowed. It is clear that after applying all steps, we get a potential that has nonzero coefficients only for cycles in a finite fixed collection; that concludes the proof of finite-dimensionality of right-equivalence classes. We marked the step where \emph{strongly generic} condition (\ref{strgen}) is crucially used in the scheme.

\medskip
\begin{proposition}\label{straight}
Any generic potential $W$ for $Q_{\TT,m}$ is right-equivalent to a linear combination of straight cycles.
\end{proposition}
\begin{proof}
We use the idea described above. Let $N$ be the smallest degree of a non-straight cycle appearing in $W$. The set of all non-straight cycles of degree $N$ is finite. Let $C$ be one of them and $u = W[C]$. We construct right-equivalence $\Phi$ such that:
\begin{enumerate}[label=(\alph*)]\label{straight0}
\item The number of non-straight cycles of degree $N$ in $\Phi(W)$ is strictly less than in $W$;
\item $W \equiv   \Phi(W) \mod R\langle\langle Q\rangle\rangle_{\operatorname{cyc},{N+1}}$
\end{enumerate}
\noindent In fact, one can see from the proof that statement (b) can be strengthened to:
\begin{equation*}
W - \Phi(W)  \equiv   uC \mod R\langle\langle Q\rangle\rangle_{\operatorname{cyc},{N+1}}
\end{equation*}
Write $C = \al_1\ldots \al_{i-2}l(\al_i)\al_i\ldots \al_N$ and note that $(l^2(\al_i), i+1, i-2)$ is a chord of the cycle. There is right-equivalence $\varphi_{1}$ given by:
\begin{equation}\label{straight1}
\begin{cases}
\varphi_{1}(\be) = \be - u\cdot \al_{i+1}\al_{i+2}\ldots \al_N\al_1\ldots \al_{i-2},\ \ &\text{if } \be = l^2(\al_i);\\
 \varphi_{1}(\be) = \be, &\text{otherwise.}
\end{cases}
\end{equation}

Suppose $l^2(\al_i)$ belongs to a $L^{(2)}_p$ cycle, in this case the only cycle of degree $\leq 3$ passing through $f^2(\al_i)$ is a triangle $l^2(\al_i)f(\al_i)\al_i$, and we get:
\begin{equation}\label{straight2}
W - \varphi_{1}(W) \equiv uC\mod R\langle\langle Q\rangle\rangle_{\operatorname{cyc},{N+1}}
\end{equation}
\noindent In simple words, it means that we got rid of cycle $C$ after applying $\varphi_{1}$.

\begin{figure}
	\centering
	\begin{minipage}[b]{.33\textwidth}
		\centering
		\includegraphics{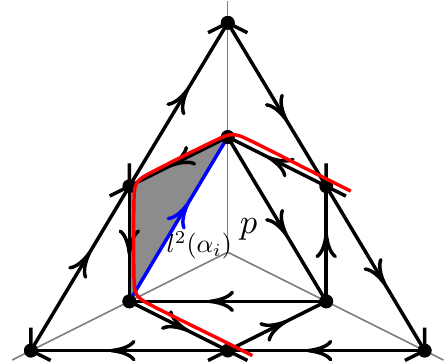}
	\end{minipage}%
	\begin{minipage}[b]{.33\textwidth}
	\centering
		\includegraphics{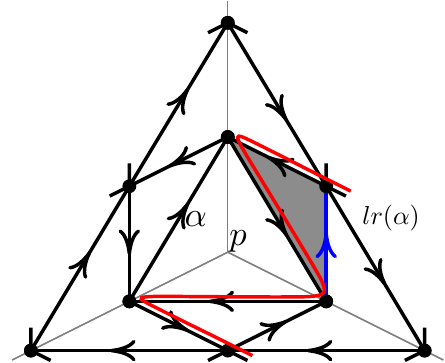}
	\end{minipage}%
	\begin{minipage}[b]{.33\textwidth}
		\centering
		\includegraphics{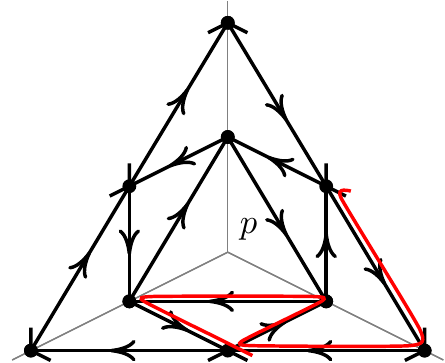}
	\end{minipage}
	\captionof{figure}{Proof of Proposition \ref{straight} when $l^2(\al_i)$ belongs to $L^{(1)}_p$.}
	\label{straightfig}
\end{figure}
Suppose now that for any $\al_i$ with $l(\al_i) = \al_{i-1}$, arrow $\al = l^2(\al_{i})$ belongs to a cycle of form $L_p^{(1)}$ (Fig. \ref{straightfig}). Fix one such $i$, and let $L_p^{(1)}$ be the cycle containing $\al$. Then necessarily $\al_{i+1} = r^{-1}(\al_{i})$ and applying $\varphi_{1}$ gives:
\begin{equation}\label{straighteq}
W - \varphi_{1}(W) \equiv uC - u v^{(1)}_p\cdot \al_{1}\ldots \al_{i-2}\left(\dd_{\al} L_p^{(1)} \right)\al_{i+1}\ldots \al_N\mod R\langle\langle Q\rangle\rangle_{\operatorname{cyc},{N+1}}
\end{equation}
We have used notation $\dd_{\al} L_p^{(1)}$ for the so-called ``cyclic derivative'' that is defined as follows. For a cycle $\gamma_1\gamma_2\ldots \gamma_n$ and an arrow $\gamma_k$:
$$
\dd_{\gamma_k} \left(\gamma_1\gamma_2\ldots \gamma_n\right) = \gamma_{k+1}\ldots\gamma_n\gamma_1\ldots\gamma_{k-1}.
$$
Thus, the last term in~(\ref{straighteq}) comes from interaction with cycle $L_p^{(1)} = l^{\operatorname{val}(p)-1}(\al)\ldots l(\al)\al$ and has degree $N$ iff $\operatorname{val}(p) = 3$. Right-equivalence $\varphi_{1}$ removes cycle $C$ but possibly introduces another cycle $C'$ of the same degree. We claim that it can be removed by applying the previous argument.

Observe that $\dd_{\al} L_p^{(1)}$ ends with $r(\al)$, and $\al_{i+1} = r^{-1}(\al_i) = l^{-1}r(\al)$ (see Fig. \ref{straightfig}). Thus, $C'$ is not straight because it has a fragment $r(\al)l^{-1}r(\al)$. These two arrows form two sides of a black triangle, and the third arrow $lr(\al)$ belongs to $L^{(2)}_q$ for some puncture $q$. This is the setting of the previous argument and $C'$ can be removed from $\varphi_1(W)$ as in the previous case after applying some $\varphi_2$ that acts non-trivially on $lr(\al)$.

It follows that after taking the composition $\varphi_2\circ\varphi_1$ if necessary we get right-equivalence $\Phi$ satisfying desired properties~\ref{straight0}.
\end{proof}

In the course of the proof we constructed necessary right-equivalence as a product of unitriangular right-equivalences changing a single arrow $\al$ to $\al - P$, where $P$ is a path from $s(\al)$ to $t(\al)$ of degree at least $2$:

\begin{equation}\label{req_notation}
\begin{cases}
\varphi(\be) = \be - P,\ \ &\text{if } \be = \al;\\
\varphi(\be) = \be, &\text{otherwise.}
\end{cases}
\end{equation}

\noindent Such equivalences will be frequently used in subsequent proofs and they we will call them \emph{elementary}. To make our notation more concise we will omit the second line from (\ref{req_notation}), and simply write
$$
\varphi(\be) = \be - P,\ \ \text{if } \be = \al
$$

\medskip
\noindent We need the following strengthening of this Proposition \ref{straight}:
\begin{corollary}\label{nonloc}
Any generic potential $W$ for $Q_{\TT,2}$ is right-equivalent to $W'$, such that
\begin{enumerate}[label=(\roman*)]
\item primitive and local parts of $W$ and $W'$ are the same;
\item all non-local cycles of degree at least $8$ in $W'$ are straight.
\end{enumerate}
\end{corollary}

\begin{figure}
	\centering
	\begin{minipage}[b]{.33\textwidth}
		\centering
		\includegraphics{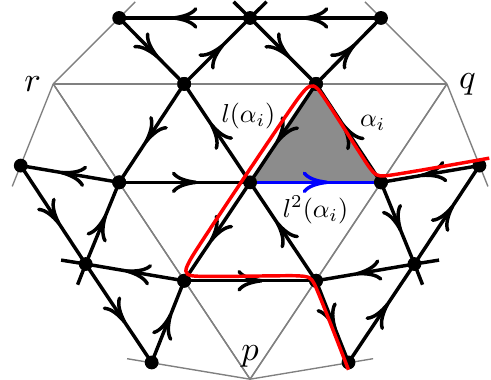}
	\end{minipage}%
	\begin{minipage}[b]{.33\textwidth}
		\centering
		\includegraphics{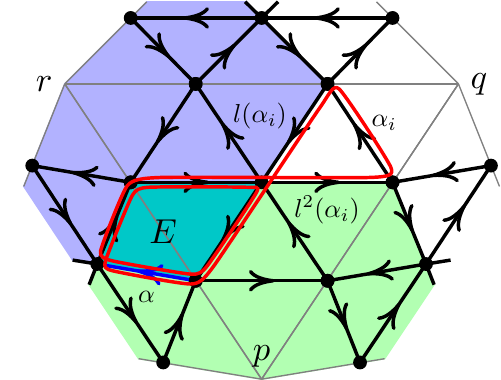}
	\end{minipage}%
	\begin{minipage}[b]{.33\textwidth}
		\centering
		\includegraphics{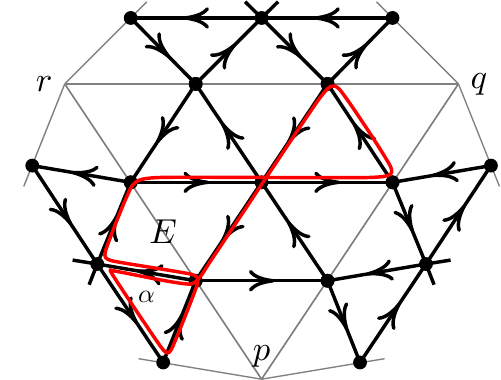}
	\end{minipage}
	\captionof{figure}{Proof of Corollary \ref{nonloc}}
	\label{nonlocfig}
\end{figure}

\noindent\emph{Remark:} The statement does not regard non-local cycles of degree less than $8$. The latter are either a product of two black triangles sharing a vertex, or a product of a black triangle and an edge cycle that share exactly one vertex. More precisely, these two classes are given by formulas $\overleftarrow{rllfl}(\al)$ and $\overleftarrow{ffrffl} (\al)$, where $\al$ is any arrow that belongs to a cycle of the form $L_p^{(1)}$. These types are shown as Triangle Terms in Table \ref{table1} and they are important for the second part of Theorem~\ref{mn3}.

\begin{proof}
We need to show that in the proof of Proposition \ref{straight}, for every non-local and non-straight cycle $C=\al_1\ldots \al_n$ in $W$, there is a right-equivalence that removes $C$ without changing local part of the potential.

As before, consider case when $C$ has a fragment $...l(\al_{i})\al_i... $, where $l^2(\al_i)$ belongs to a cycle of the form $L^{(2)}_p$. Without loss of generality we can assume that $\al_i$ belongs to some $L^{(1)}_q$ and $l(\al_1)$ belongs to some $L^{(2)}_r$ as shown on the left of Figure \ref{nonlocfig} (evidently, $p,q$ and $r$ form a triangle of $T$). Then we can rewrite $C = Pl(\al_i)\al_i$ where by non-locality assumption $P$ is a path that does not belong to patch $Q_q$ (indeed, otherwise $C$ is contained in $Q_q$).

If in addition $P$ does not belong to patch $Q_p$, one can apply an elementary right-equivalence:
\begin{equation}
\varphi(\be) = \be - P, \ \ \text{if}\  \be = l^2(\al_i)
\end{equation}

\noindent Note that all cycles in the expression for $W - \varphi(W)$ contain path $P$, hence are not local. Moreover, $W - \varphi(W) = C \mod R\langle\langle Q\rangle\rangle_{\operatorname{cyc},{n+1}}$, so the non-local non-straight cycle is removed, while the local part is unchanged.

Assume that $P$ is contained in $Q_p$, then necessarily $\al_{i+1} = l^{-1}(\al_i) = l^2(\al_i)$ (see the middle of Fig \ref{nonlocfig}). Therefore, we can rewrite our expression for $C$ as  $P'f(\al_{i})\al_{i}f^2(\al_{i})$, and without loss of generality can assume that $P'$ now belongs to \emph{both} patches $Q_p$ and $Q_r$ (otherwise, we can repeat the previous argument with $r$ in place of $p$). That implies that $P'$ is supported on a single edge cycle $E$ lying between these patches. So $P'$ has form $E^k$ and $k>1$, since  $\deg C > 7$.

Let $\al$ be an arrow of $E$ that belongs to $L^{(2)}_r$, then we have elementary right-equivalence:
$$
\varphi(\al) = \be - \al r^{-1}(\al)\al_{i-1}\al_i\al_{i+1}E^{k-2}r^2(\al)r(\al)\al , \ \text{if}\  \be = \al\\
$$

\noindent Again $W - \varphi(W)$ contains only non-local cycles. The only new term of degree at most $d$ is shown on the right of Figure \ref{nonlocfig} (for $k=2$):
$$
W - \varphi(W) = C - l^2(\al)l(\al)\al r^{-1}(\al)\al_{i-1}\al_i\al_{i+1}E^{k-2}r^2(\al)r(\al)\al\mod R\langle\langle Q\rangle\rangle_{\operatorname{cyc},{n+1}},
$$
\noindent Even though the second term on the right hand side has degree $n-1$, it can be seen that running essentially same argument for the fragment $\al l^{2}(\al)$ allows to remove it by a suitable right-equivalence. This concludes considerations of the case when there exists fragment $l(a_{i})a_i$ in $C$ such that $l^2(a_i)$ belongs to cycle $L^{(2)}_p$.
\smallskip

If, on the other hand, for any $i \in \ZZ/n\ZZ$ with $\al_{i-1} = l(\al_i)$ the arrow $l^2(\al_i)$ belongs to cycle of form $L^{(1)}_p$, the argument of Proposition \ref{straight} goes through without any changes, and with an additional remark that right-equivalences constructed there do not change local part of the potential in this case.
\end{proof}

\medskip
In Proposition \ref{straight} we proved that any potential is right-equivalent to a linear combination of straight cycles. Now we show that local part of a potential can be reduced to an expression involving only cycles $\left(L^{(k)}_p\right)^n$ (Lemma \ref{onlyL}).


\begin{figure}
	\centering
	\includegraphics{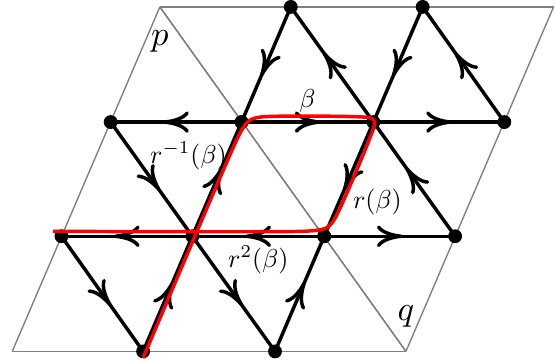}
  \captionof{figure}{Proof of Lemma~\ref{divE}}
	\label{divEfig}
\end{figure}

\begin{lemma}\label{divE}
Any non-chordless straight cycle either has form $\left(L^{(k)}_p\right)^n, n\geq 2$, or contains a fragment $... r^3(\al)r^2(\al)r(\al)\al...$, corresponding to some edge cycle.
\end{lemma}

\begin{proof}

Let $C$ be a non-chordless straight cycle different from $\left(L^{(k)}_p\right)^n$, then according to our notation we can write it in the form $\overleftarrow{h_1\ldots h_n}(\al)$, where each $h_j$ is either $r$ or $f$. Let us require that in this expression we use bijection $l$ if there is ambiguity: $r(\overleftarrow{h_j\ldots h_n}(\al)) = f(\overleftarrow{h_j\ldots h_n}(\al))$. Thus, $h_1,\ldots h_n$ is a sequence of letters $r$ and $f$ where at least one $r$ occurs.

Consider this fragment $\ldots r(\be)\be\ldots$ in $C$. By construction $r(\be) \neq f(\be)$ and so $s(\be)$ belongs to some $L^{(1)}_p$. Since $C$ is straight, $\be$ is necessarily preceded by $f^{-1}(\be)= r^{-1}(\be)$. Similarly, $t(r(\be))$ belongs to some $L^{(1)}_{q}$, hence $r(\be)$ is necessarily followed by $r^2(\be)$ (see Fig \ref{divEfig}).

Combining these observations, we see that $C$ contains a fragment $\ldots r^2(\be)r(\be)\be r^{-1}(\be)\ldots$.
\end{proof}

\begin{lemma}\label {onlyL}
Any generic potential $W$ on $Q_{\TT,2}$ is right-equivalent to a sum of the form $W_{\operatorname{prim}} + W_{\operatorname{loc}} + W_{\operatorname{nonloc}}$, satisfying conditions:
\begin{enumerate}[label=(\roman*)]
\item $W_{\operatorname{prim}}$ is the primitive part of $W$;
\item $W_{\operatorname{loc}}$ is a linear combination of cycles of $\left(L^{(k)}_p\right)^n, \ n>1$;
\item $W_{\operatorname{nonloc}}$ is a linear combination of nonlocal cycles.
\end{enumerate}
\end{lemma}

\noindent \emph{Remark:} By contrast with Lemma \ref{removeL}, here we require only that $W$ is generic.

\begin{proof}
By Proposition \ref{straight} we can assume that all local cycles in $W$ are already straight. Let $C = \al_1\ldots \al_n$ be a local cycle different from edge cycle $E$ and from a power of $L^{(k)}_p$; and suppose that $C$ has minimal degree among cycles of $W$ having these properties.

To prove the lemma it is enough to construct right-equivalence $\varphi$ such that:

\begin{enumerate}[label=(\alph*)]
\item $W - \varphi(W) = 0 \mod R\langle\langle Q\rangle\rangle_{\operatorname{cyc},{n-1}}$;
\item The only local cycle of degree at most $n$ in $W - \varphi(W)$ is $C$.
\end{enumerate}

\noindent If such right-equivalence is found, we can remove one-by-one local cycles of unsuitable form. Note that during the process non-local cycles may be created, but condition (a) guarantees that for every degree only finitely many new non-local cycles will be introduced.

Let $Q_p$ be a patch containing $C$. By Lemma \ref{divE} any local cycle different from $\left(L^{(k)}_p\right)^n$ is ``divisible'' by some edge cycle. That is, we can write $C = EC' = r^3(\al)r^2(\al)r(\al)\al C'$, where $\al$ is one of the arrows in edge cycle $E$ (see Fig \ref{onlyLfig1}). This cycle lies in two patches, say, $Q_p$ and $Q_q$, and hence $\al$ is associated to one of the punctures $p$ or $q$.
\smallskip

\begin{figure}
	\centering
	\begin{minipage}[b]{.5\textwidth}
		\centering
		\includegraphics{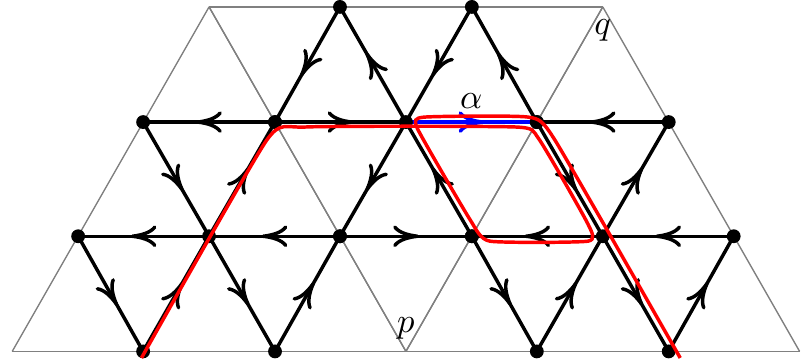}
	\end{minipage}%
	\begin{minipage}[b]{.5\textwidth}
		\centering
		\includegraphics{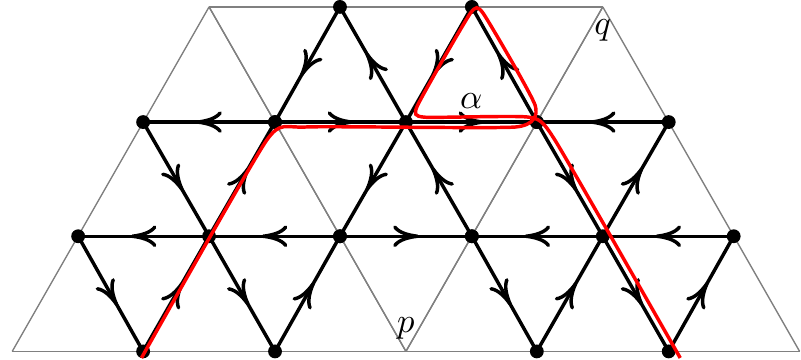}
	\end{minipage}
	\captionof{figure}{Lemma \ref{onlyL}, \emph{Case 1}}
	\label{onlyLfig1}
\end{figure}

\begin{figure}
	\centering
	\begin{minipage}[b]{.5\textwidth}
		\centering
		\includegraphics{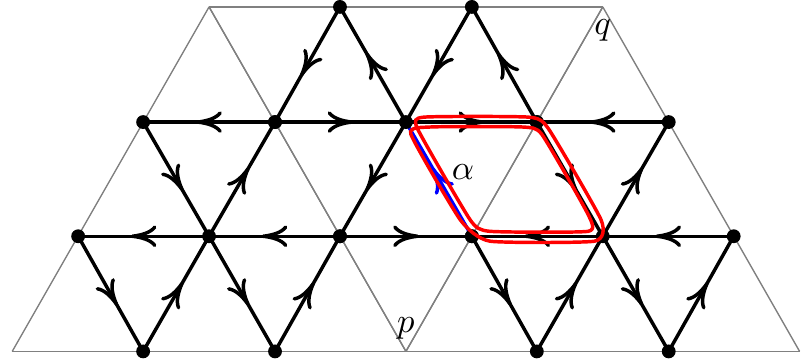}
	\end{minipage}%
	\begin{minipage}[b]{.5\textwidth}
		\centering
		\includegraphics{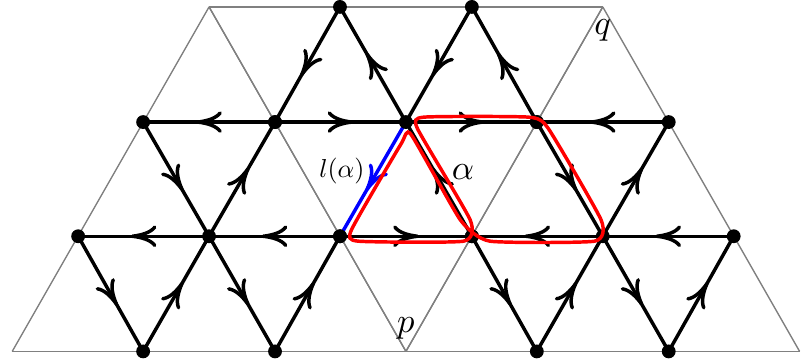}
	\end{minipage}
	\captionof{figure}{Lemma \ref{onlyL}, \emph{Case 2}}
	\label{onlyLfig2}
\end{figure}

 \noindent \emph{Case 1:} Assume first that $C$ is not proportional to a power of $E$, and that $\al$ is associated to $p$. Then we can use elementary right-equivalence $\varphi$:
\begin{equation}
\varphi(\be) = \be - \be C', \ \ \text{if } \be = \al
\end{equation}

\noindent Then $W - \varphi(W) = C + l^2(\al)l(\al)\al C' \mod R\langle\langle Q\rangle\rangle_{\operatorname{cyc},{n+1}} $, and the second term is not local (because $C'$ is contained in $Q_p$ and is not a power of $E$).

\smallskip
\noindent \emph{Case 2:} If $C$ is not proportional to a power of $E$, but $\al$ is associated to puncture $q$ different form $p$, then similarly to the previous case, the job is done by $\varphi$:
\begin{equation}
\varphi(\be) = \be - C'\be, \ \ \text{if } \be = r^3(\al)
\end{equation}

\begin{figure}
	\centering
	\includegraphics{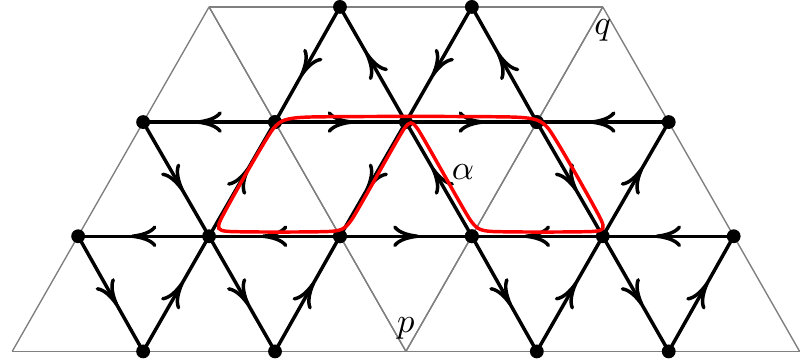}
	\captionof{figure}{Lemma \ref{onlyL}, \emph{Case 3}}
	\label{onlyLfig3}
\end{figure}

\noindent\emph{Case 3:}  Finally, consider the case $C = E^k = v\left(r^3(\al)r^2(\al)r(\al)\al \right)^k,\ k\geq 2,\ v\in \kk^{\times}$. Without loss of generality we can assume that $f^{-1}(\al)$ and $\al$ are associated to $q$ (see Fig. \ref{onlyLfig3}). We can apply a composition of two right-equivalences  $\varphi''_2 \circ \varphi''_1$, where:

\begin{minipage}{.5\textwidth}
\begin{equation*}
	\varphi''_1(\be) = \be - v\be E^{k-1}, \ \ \text{if } \be  = \al;
\end{equation*}
\end{minipage}%
\begin{minipage}{.5\textwidth}
\begin{equation*}
	\varphi''_2(\be) = \be + vE^{k-1}\be, \ \ \text{if } \be  = l(\al)
\end{equation*}
\end{minipage}

\smallskip
\noindent The first map $\varphi''_1$ creates a new local cycle of degree $n-1$:
$$
W -\varphi''_1(W) = C + l^2(\al) l(\al)\al E^{k-1} \mod R\langle\langle Q\rangle\rangle_{\operatorname{cyc},{n+1}}
$$
\noindent This new local cycle is replaced by yet another local cycle of degree $n$ after application of the second right-equivalence:
$$
\varphi''_1 (W) - \varphi''_2\circ \varphi''_1(W) =  - l^2(\al) l(\al)\al E^{k-1} - \overleftarrow{rrr}(l(\al))E^{k-1}\mod R\langle\langle Q\rangle\rangle_{\operatorname{cyc},{n+1}}
$$
\noindent The new term here is local of degree $n$, and it can be dealt with by arguments from Case 1 or 2. The proof of Lemma \ref{onlyL} is complete.
\end{proof}

Using Lemmas \ref{straight} and \ref{onlyL}, we can reduce any generic potential to the form, where local part consists only from chordless cycles and cycles of the form $\left(L^{(k)}_p\right)^n, \ n>1$. We now prove that if potential is \underline{strongly} generic, then these local terms can be pushed to non-local part of the potential.

\begin{lemma}\label{removeL}
Let $W$ be a strongly generic potential on $Q_{\TT,m}$ such that all non-chordless local cycles have form $\left(L^{(2)}_p\right)^n,\ n\geq2$. Then $W$ is right-equivalent to the sum of its primitive part and a linear combination of non-local cycles.
\end{lemma}

\begin{figure}
	\centering
	\begin{minipage}[b]{.5\textwidth}
		\centering
		\includegraphics{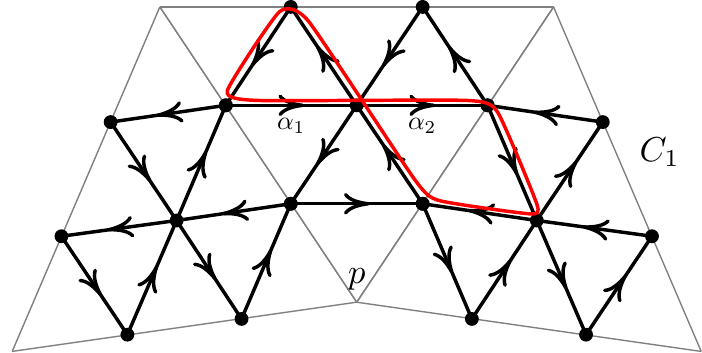}
		\includegraphics{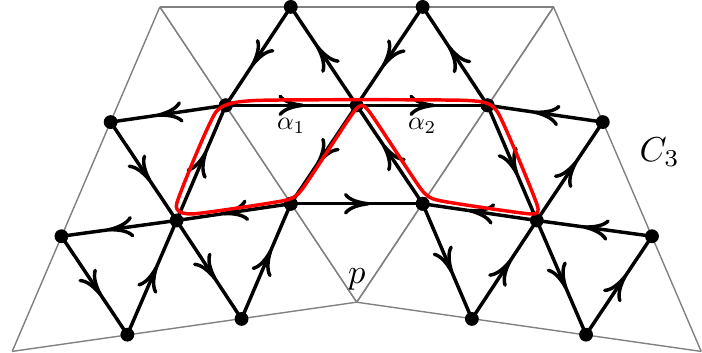}
	\end{minipage}%
	\begin{minipage}[b]{.5\textwidth}
		\centering
		\includegraphics{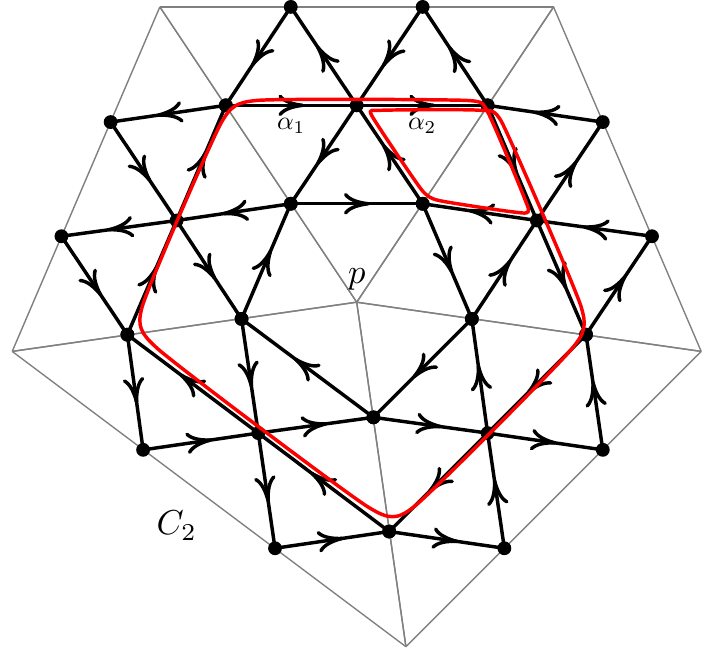}
	\end{minipage}
	\captionof{figure}{Lemma \ref{removeL}, cycles $C_1,C_2,C_3$.}
	\label{removeLfig1}
\end{figure}

\noindent \emph{Remark:} To see the relevance of this lemma, note that arguments of Propositions \ref{straight} and Lemmas \ref{onlyL} allow one not only to reduce generic potential to some specific form, but also to do that without creating new cycles of the form $\left(L^{(1)}_p\right)^n,\ n > 1$. That is, coefficients of these cycles remain unchanged during various manipulations in the proofs. This can be verified directly by inspecting our formulas for right-equivalences.

Thus, starting from any generic potential, one can first find a right-equivalent potential that does-not involve cycles $\left(L^{(1)}_p\right)^n,\ n > 1$. And then apply Proposition \ref{straight} and Lemma \ref{onlyL} to assume that the only local cycles, that are not chordless, have form $\left(L^{(2)}_p\right)^n,\ n > 1$. Hence, every strongly generic potential has a representative in its right-equivalence class for which assumptions of this lemma are satisfied.

\begin{proof}
Fix $n\geq 2$ and puncture $p$, such that $\left(L^{(2)}_p\right)^n$ is a non-chordless local cycle in the expression for $W$ of lowest degree among such cycles; let $N = 2n\val (p)$ be its degree. It is enough to construct a right-equivalence $\Phi$ such that:

\noindent\begin{minipage}{.01\textwidth}
\begin{equation}\label{Phi_prop}
\end{equation}
\end{minipage}%
\begin{minipage}{.95\textwidth}
\begin{enumerate}[label=(\alph*)]
\item The only local cycle in $W - \Phi(W)\mod R\langle\langle Q\rangle\rangle_{\operatorname{cyc},{N+1}}$ is $\left(L^{(2)}_p\right)^n$ with arbitrary nonzero coefficient;
\item $W - \Phi(W)\equiv 0\mod R\langle\langle Q\rangle\rangle_{\operatorname{cyc},{N-c}}$ for some constant $c$.
\end{enumerate}
\end{minipage}

Recall the coefficient $v^{(k)}_q = W[L^{(k)}_q]$. Let $\al_1$ and $\al_2$ be two arrows on $L^{(2)}_p$ such that $\al_2 = f(\al_1)$ and $r(\al_1)\neq \al_2$. Then define elementary right-equivalence~$\varphi_{-2}$:
\begin{equation}\label{phi-2}
\varphi_{-2}(\be) = \be + r^3(\al_2)r^2(\al_2)r(\al_2)\al_2\left(L^{(2)}_p\right)^{n-2}\al_1, \ \ \text{if } \be = \al_1
\end{equation}
There are three new terms of degree less than $N$ created by this operation:
\begin{equation}
\begin{split}
W - \varphi_{-2}(W) = -\ r^3(\al_2)r^2(\al_2)r(\al_2)\al_2&\left(L^{(2)}_p\right)^{n-2}\al_1l^2(\al_1)l(\al_1) -\\
-v^{(2)}_pr^3(\al_2)r^2(\al_2)r(\al_2)\al_2&\left(L^{(2)}_p\right)^{n-1}- \\
-\ r^3(\al_2)r^2(\al_2)r(\al_2)\al_2\left(L^{(2)}_p\right)^{n-2}\al_1r^3(\al_1)&r^2(\al_1)r(\al_1)\ \mod R\langle\langle Q\rangle\rangle_{\operatorname{cyc},{N+1}}
\end{split}
\end{equation}

\noindent Denote these terms by $C_1, C_2$ and $C_3$ respectively (see Fig. \ref{removeLfig1} for $n=2$ case). We deal with them one-by-one to arrive to the potential with desired properties.

\smallskip

Cycles from $C_1$ are not local for any $n\geq1$, and so we only need to check that their degree is bounded from below by $N-c$ for some universal constant $c$. In this case the existence of such constant is evident; for example, we can choose $c$ to be twice maximal valence of vertices of triangulation $\TT$.
\smallskip

For the cycle $C_2$ we take right-equivalence $\varphi_{-1}$:
\begin{equation}
\varphi_{-1}(\be) = \be - v^{(2)}_p\al_2\left(L^{(2)}_p\right)^{n-1}, \ \ \text{if } \be = \al_2
\end{equation}
\begin{figure}
	\centering
	\begin{minipage}[b]{.5\textwidth}
		\centering
		\includegraphics{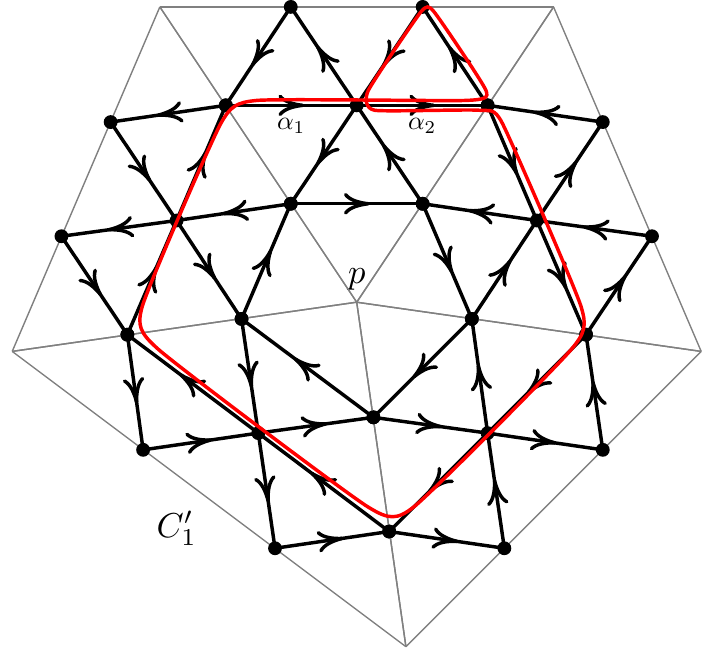}
	\end{minipage}%
	\begin{minipage}[b]{.5\textwidth}
		\centering
		\includegraphics{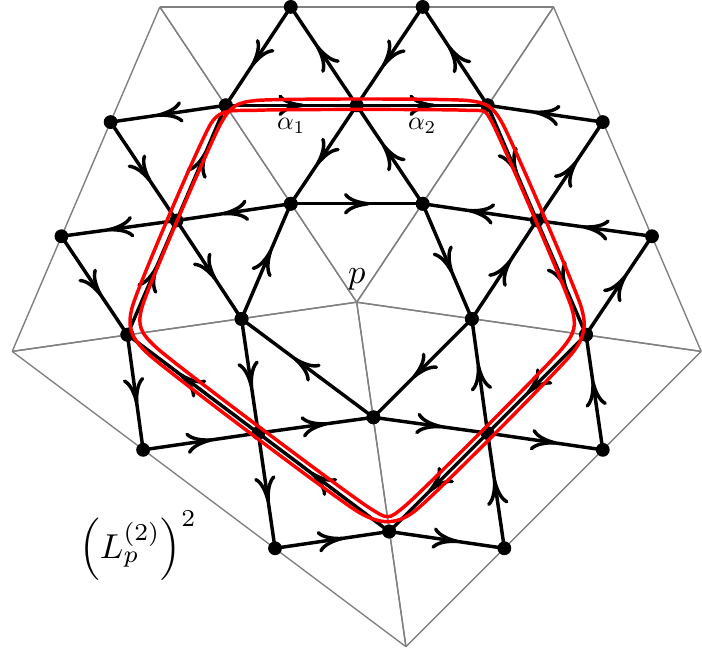}
	\end{minipage}
	\captionof{figure}{Lemma \ref{removeL}, cycles $C'_1,\ \left(L^{(2)}_p\right)^2$.}
	\label{removeLfig2}
\end{figure}
\noindent Then it is easy to see that (Fig. \ref{removeLfig2}):
$$
W - \varphi_{-1}\circ\varphi_{-2}(W) = -C_1+C'_1+ \left(v^{(2)}_p \right)^2\left(L^{(2)}_p\right)^{n}- C_3\
\mod R\langle\langle Q\rangle\rangle_{\operatorname{cyc},{N+1}}
$$
\noindent Here $C'_1$ is another linear combination of non-local cycles $l^2(\al_2)l(\al_2)\al_2\left(L^{(2)}_p\right)^{n-1}$. After two right-equivalences $\varphi_{-2}\circ\varphi_{-1}$, cycle $C_2$ has been replaced by a constant times $\left(L^{(2)}_p\right)^{n}$, and we leave it that way for now.

\smallskip
Now we deal with the cycle $C_3 = r^3(\al_2)r^2(\al_2)r(\al_2)\al_2\left(L^{(2)}_p\right)^{n-2}\al_1r^3(\al_1)r^2(\al_1)r(\al_1)$. This will require $\val(p)+1$ steps $\varphi_0,\ldots \varphi_{\val(p)}$, such that after applying their composition we arrive to:
\begin{equation}\label{finale}
W - \varphi_{\val(p)}\ldots\varphi_0\varphi_{-1}\varphi_{-2}(W) = C_{\operatorname{nonloc}}+ \left(v^{(2)}_p \right)^2\left(L^{(2)}_p\right)^{n} +(-1)^{\val(p)}v^{(2)}_p v^{(1)}_p\left(L^{(2)}_p\right)^{n} \
\mod R\langle\langle Q\rangle\rangle_{\operatorname{cyc},{N+1}}
\end{equation}

\noindent Where $C_{\operatorname{nonloc}}$ stands for some linear combination of non-local cycles.

To define $\varphi_j$ for $j = 0,\dots \val(p)-1$, we fix following notations (shown in Figure \ref{removeLfig3} for $\val(p)=5$):

\begin{enumerate}
\item Denote $\val(p)$ arrows that go from $L^{(1)}_p$ to $L^{(2)}_p$ by $\gamma_0,\ldots_{\val(p)-1}$ in the clockwise order starting from $\gamma_0 = r^3(\al_2)$.
\item Denote $\val(p)$ arrows that form $L^{(1)}_p$ by $\delta_0,\ldots\delta_{\val_p-1}$ in clockwise order starting from $\delta_0 = lr(\al_1)$.
\item Denote $2\val(p)$ arrows that form $L^{(2)}_p$ by $\delta_0',\delta_0'',\delta_1',\ldots\delta_{\val(p)-1}',\delta_{\val(p)-1}''$ in clockwise order starting from $\delta_0' = \al_1, \delta_0'' = \al_2$.
\end{enumerate}

\begin{figure}
	\centering
	\begin{minipage}[b]{.5\textwidth}
		\centering
		\includegraphics{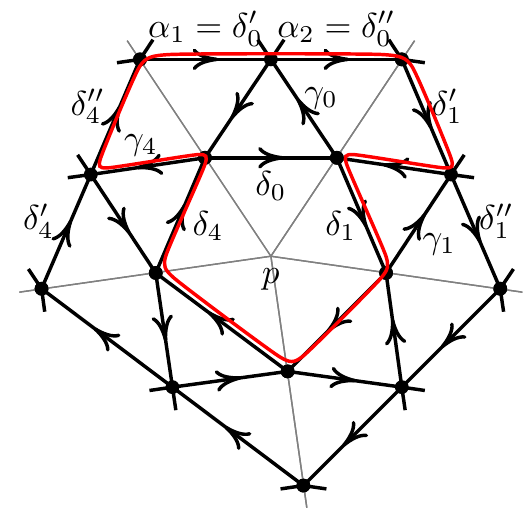}

		\includegraphics{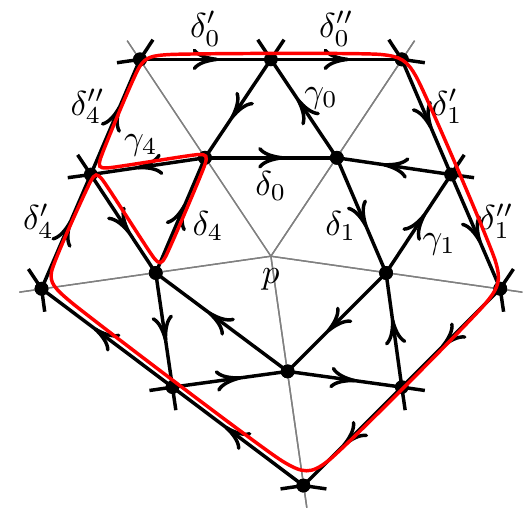}
	\end{minipage}%
	\begin{minipage}[b]{.5\textwidth}
		\centering
		\includegraphics{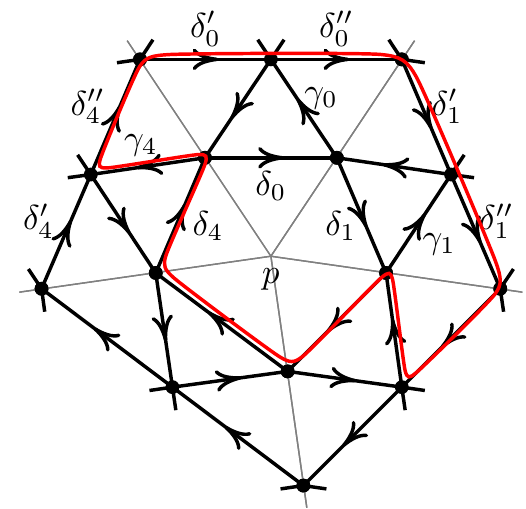}

		\includegraphics{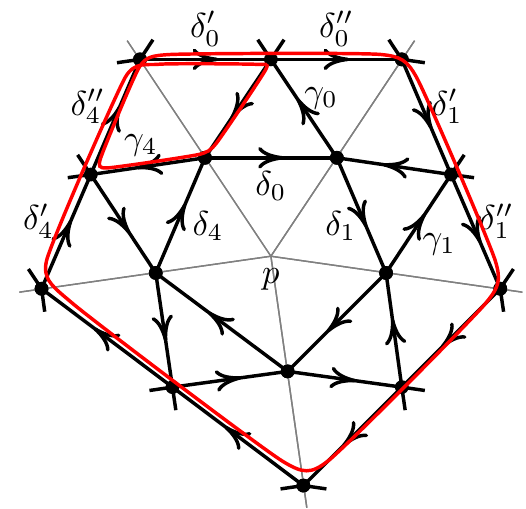}
	\end{minipage}
	\captionof{figure}{Lemma \ref{removeL}, sequence of cycle replacements for $\varphi_0,...\varphi_5$.}
	\label{removeLfig3}
\end{figure}

The zeroth right-equivalence is the following:
\begin{equation}
\varphi_{0}(\be) = \be - r^2(\al_2)r(\al_2)\al_2\left(L^{(2)}_p\right)^{n-2}\al_1r^3(\al_1)r^2(\al_1), \ \ \text{if } \be = rl(\al_1)
\end{equation}

\noindent Then modulo  $R\langle\langle Q\rangle\rangle_{\operatorname{cyc},{N+1}}$ we get (new term is depicted on top-left of Fig. \ref{removeLfig3}):
$$
W - \varphi_0\varphi_{-1}\varphi_{-2}(W) \equiv -C_1 +C'_1+ \left(v^{(2)}_p \right)^2\left(L^{(2)}_p\right)^{n} + v^{(1)}_p l(\gamma_1)\delta'_1\delta''_0\left(L^{(2)}_p\right)^{n-2}\al_1r^3(\al_1)r^2(\al_1)\delta_{\val(p)-1}...\delta_1
$$

\noindent Next for $j = 1,\ldots \val(p)-2$ take:

\begin{equation}
\varphi_{j}(\be) = \be + (-1)^{j+1} v^{(1)}_p \delta_j'...\delta_1' \delta_0''\left(L^{(2)}_p\right)^{n-2}\al_1r^3(\al_1)r^2(\al_1)\delta_{\val(p)-1}...\delta_{j+2}\delta_{j+1}, \ \ \text{if } \be = \gamma_j
\end{equation}

Denote obtained successive compositions by $\Phi_j = \varphi_j\ldots\varphi_0\varphi_{-1}\varphi_{-2}$. It is verified by induction directly that for $j = 1,\ldots \val(p)-2$  (two cycles on the top of Fig. \ref{removeLfig3} are the last two summands in \ref{removeLeq} for $n=2, \ j=1$):
\begin{equation}\label{removeLeq}
\begin{split}
\Phi_{j-1}(W) - \Phi_{j}(W) = C_{\operatorname{nonloc}}+(-1)^{j} v^{(1)}_p l(\gamma_{j})\delta_{j}'...\delta_1'\delta_0''\left(L^{(2)}_p\right)^{n-2}\al_1r^3(\al_1)r^2(\al_1)\delta_{\val(p)-1}...\delta_j + \\
+ (-1)^{j} v^{(1)}_p l(\gamma_{j+1})\delta_{j+1}'...\delta_1'\delta_0''\left(L^{(2)}_p\right)^{n-2}\al_1r^3(\al_1)r^2(\al_1) \delta_{\val(p)-1}...\delta_{j+1}
\mod R\langle\langle Q\rangle\rangle_{\operatorname{cyc},{N+1}}
\end{split}
\end{equation}

\noindent And this leads to (last term is on the bottom-left of Fig. \ref{removeLfig3}):
\begin{equation}
\begin{split}
W - \Phi_{\val(p)-2}(W) = C_{\operatorname{nonloc}} + \left(v^{(2)}_p \right)^2&\left(L^{(2)}_p\right)^{n} +\\ +(-1)^{\val(p)} v^{(1)}_pl(\gamma_{\val(p)-1})\delta_{\val(p)-1}'...\delta'_1\delta''_0\left(L^{(2)}_p\right)^{n-2}\al_1r^3(\al_1)&r^2(\al_1)\delta_{\val(p)-1}\mod R\langle\langle Q\rangle\rangle_{\operatorname{cyc},{N+1}}
\end{split}
\end{equation}
Define $\varphi_{\val(p)-1}$:
\begin{equation}
\varphi_{\val(p)-1}(\be) = \be + (-1)^{\val(p)}v^{(1)}_p \delta'_{\val(p)-1}...\delta'_1\delta''_0\left(L^{(2)}_p\right)^{n-2}\al_1r^3(\al_1)\be, \ \ \text{if}\ \  \be = \gamma_{\val(p)-1}
\end{equation}
After that we have (last term is on the bottom-right of Fig. \ref{removeLfig3}):
\begin{equation}
\begin{split}
W - \Phi_{\val(p)-1}(W) = C_{\operatorname{nonloc}} + \left(v^{(2)}_p \right)^2&\left(L^{(2)}_p\right)^{n} +\\ -(-1)^{\val(p)} v^{(1)}_p\left(L^{(2)}_p\right)^{n-1}&\al_1r^3(\al_1)r^2(\al_1)r(\al_1)\mod R\langle\langle Q\rangle\rangle_{\operatorname{cyc},{N+1}}
\end{split}
\end{equation}
Finally define $\varphi_{\val(p)}$ by:
\begin{equation}
\varphi_{\val(p)}(\be) = \be - (-1)^{\val(p)} \left(L^{(2)}_p\right)^{n-1}\be, \ \ \text{if}\ \ \be = \al_1\\
\end{equation}

\noindent It follows that:
\begin{equation}\label{removeLfin}
W - \Phi_{\val(p)}(W) = C_{\operatorname{nonloc}} + \left(v^{(2)}_p \right)^2\left(L^{(2)}_p\right)^{n}  +(-1)^{\val(p)} v^{(1)}_pv^{(2)}_p\left(L^{(2)}_p\right)^{n}\mod R\langle\langle Q\rangle\rangle_{\operatorname{cyc},{N+1}}
\end{equation}

This is precisely the form anticipated in \ref{finale}. Note that the degree of cycles in $C_{\operatorname{nonloc}}$ can be trivially bounded from below by
$$\deg\left(L^{(2)}_p\right)^{n-2} = 2(n-2)\cdot \val(p)$$

\noindent So the second condition from \ref{Phi_prop} is also fulfilled.

\smallskip
From this argument we observe that if $W$ is a strongly generic potential, then for any $c \in \kk$ one can find a right-equivalence that adds term $c\left(L^{(2)}_p\right)^{n}$ to the local part of the potential, and does not change local terms of smaller degree. The condition \ref{strgen2} was used to ensure that the aggregate coefficient of the power of $L^{(2)}_p$ in \ref{finale} is not zero. The lemma \ref{removeL} is proved.
\end{proof}

\begin{corollary}\label{Wloc}
Any strongly generic potential on $Q_{\TT,2}$ is right-equivalent to a sum of the form $W_{\operatorname{prim}} + W_{\operatorname{nonloc}}$, satisfying conditions:
\begin{enumerate}[label=(\roman*)]
\item $W_{\operatorname{prim}}$ is the primitive part of $W$;
\item $W_{\operatorname{nonloc}}$ consists only of nonlocal terms.
\end{enumerate}
\end{corollary}
The corollary follows immediately from Lemmas \ref{onlyL} and \ref{removeL}. To finish the proof of finite-dimensionality of $\Pot_{Q_{\TT,2}}$ it remains to deal with non-local cycles. Moreover, another Corollary~\ref{nonloc} allows one to assume that all non-local cycles of degree at least $8$ are straight. In fact, this remark is not so essential to the logic of our argument. Instead we will apply Corollary \ref{nonloc} whenever non-straight and non-local cycles appear during our transformations.

\medskip
Recall that for any potential on $Q_{\TT,m}$ there is a decomposition $W = W_{\operatorname{prim}} + W_{\operatorname{loc}} + W_{\operatorname{nonloc}}$, where $W_{\operatorname{prim}}$ is the primitive part, and local part $W_{\operatorname{loc}}$ consists of all remaining local cycles (that is, local cycles that are not chordless).

\begin{proposition}\label{nononloc}
Let $W$ be a generic potential with zero local part: $W_{\operatorname{loc}} = 0$. Then $W$ is right equivalent to $W_{\operatorname{prim}} + W^{\circ}_{\operatorname{nonloc}}$, where the nonlocal term consists of cycles of degree at most $7$.
\end{proposition}

\noindent \emph{Remark:} This proposition together with previous facts implies that $\Pot_{Q_{\TT,2}}$ is finite dimensional. Indeed, we can remove all non-primitive local cycles by Corollary \ref{Wloc}, and then use Proposition \ref{nononloc}. Then all cycles in the expression are either chordless or have degree at most $7$. There are evidently only finitely many such cycles in $Q_{\TT,2}$; Theorem \ref{mn3} follows.

\begin{proof}
Let $C$ be a nonlocal cycle of degree at least $8$ appearing in $W$ with nonzero coefficient, and suppose that $C$ has minimal degree among cycles having these properties. Denote by $u\in \kk^{\times}$ the coefficient of $C$ in $W$. To prove the proposition, it is enough to find a right-equivalence $\Phi$, satisfying properties:
\begin{enumerate}[label=(\alph*)]
\item $W - \Phi(W) = uC \mod R\langle\langle Q\rangle\rangle_{\operatorname{cyc},{N+1}}
$, where $N = \deg (C)$;
\item $W - \Phi(W)$ has only non-local terms.
\end{enumerate}

\begin{figure}
	\centering
	\begin{minipage}[b]{.5\textwidth}
		\centering
		\includegraphics{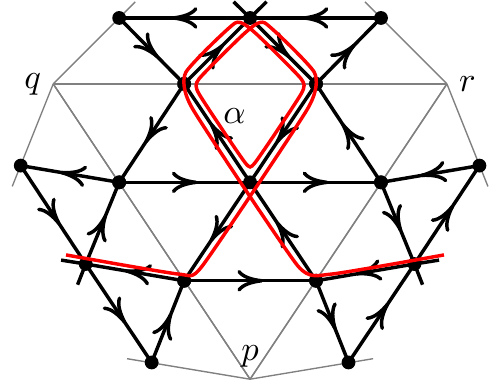}

		\vspace{10pt}
		\includegraphics{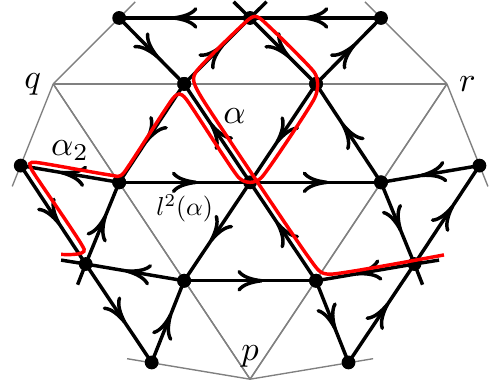}
	\end{minipage}%
	\begin{minipage}[b]{.5\textwidth}
		\centering
		\includegraphics{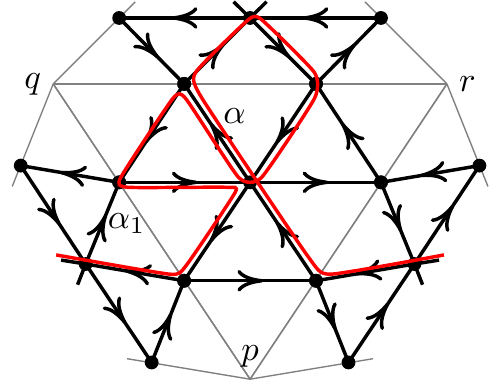}

		\vspace{10pt}
		\includegraphics{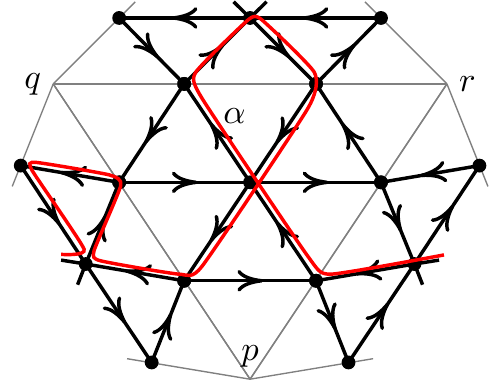}
	\end{minipage}
	\captionof{figure}{Lemma \ref{nononloc}, sequence of cycle replacements for $\varphi_1,\varphi_2,\varphi_3$.}
	\label{nononlocfig1}
\end{figure}

\noindent Similarly to the argument in the proof of Corollary \ref{Wloc}, we define a sequence of elementary right-equivalences, such that their composition satisfies the properties listed above.

If $C$ is not straight, one can apply Corollary \ref{nonloc} to remove it by applying a suitable right-equivalence, that does not change the local part of $W$. Thus, we assume that $C$ is straight. Then by extending slightly the argument from Lemma \ref{divE}, it is not difficult to see that $C$ has form $E^n C'$, where $E = r^3(\al)r^2(\al)r(\al)\al$, and $P $ is some path with $s(P) = t(P) = s(\al)$, satisfying $P =u\cdot f^{-1}(\al) \ldots f(r^3(\al))$. Example with $n = 2$ is given on top-left of Figure \ref{nononlocfig1}.

The first right-equivalence $\varphi_1$ is defined by:
\begin{equation}
\varphi_1(\be) = \be - \be E^{n-1}P, \ \ \text{if}\ \ \be = \al \\
\end{equation}

\noindent Since in the first case all new terms in $\varphi_1(W)$ have subfragment $\al E^{n-1}P$, which is not local, then the right-equivalence does not change local part of $W$. In fact, this argument for the local terms can be adapted for every step in this prove, so we won't mention it again. We have:
$$
W - \varphi_1(W) = C + l^2(\al)l(\al)\al E^{n-1}P \mod R\langle\langle Q\rangle\rangle_{\operatorname{cyc},{N+1}}
$$
\noindent The second summand is shown on top-right of Figure \ref{nononlocfig1}.

Now note that $P$ is straight and then necessarily $P = u\cdot P' f^2(r^3(\al))f(r^3(\al))$ (where $P'$ denotes the path that forms corresponding part of $P$). Since $f^2(r^3(\al)) = r^2(l^2(\al))$ we conclude, that the cycle $l^2(\al)l(\al)\al E^{n-1}P$ has a chord $\al_1$ from $r^2(l^2(\al))$ to $s(l^2(\al))$. Then we define $\varphi_2$:
\begin{equation}
\varphi_2(\be) = \be + u\cdot l(\al)\al E^{n-1} P', \ \ \text{if}\ \ \be = \al_1 \\
\end{equation}

\noindent Denote successive compositions by $\Phi_j = \varphi_j\ldots\varphi_1$, for this notation we get:
$$
W - \Phi_2(W) = C - u\cdot  E^{n-1} P' \overleftarrow{lrl}(\al) \mod R\langle\langle Q\rangle\rangle_{\operatorname{cyc},{N+1}}
$$
\noindent New cycle is on bottom-left of Figure \ref{nononlocfig1}.
Set $\al_2 = r(l(\al))$ and take $\varphi_3$:
\begin{equation}
\varphi_3(\be) = \be - u\cdot E^{n-1}P' l(\al_2)\al_2, \ \ \text{if}\ \ \be = l^2(\al) \\
\end{equation}
\noindent This results in (see bottom-right of Figure \ref{nononlocfig1}):
$$
W - \Phi_3(W) = C + u\cdot \al_1r^{-1}(\al_1)r^{-2}(\al_1)E^{n-1}  P'l(\al_2)\al_2 \mod R\langle\langle Q\rangle\rangle_{\operatorname{cyc},{N+1}}
$$

Define $\varphi_4$:
\begin{equation}
\varphi_4(\be) = \be + u \cdot r^2(l^2(\al))r(l^2(\al)) E^{n-1}  P' \be, \ \ \text{if}\ \ \be = l(\al_2) \\
\end{equation}
\noindent We get:
$$
W - \Phi_4(W) = C - u\cdot r^2(l^2(\al))r(l^2(\al))E^{n-1} P' E_1 \mod R\langle\langle Q\rangle\rangle_{\operatorname{cyc},{N+1}}
$$
\begin{figure}
	\centering
	\includegraphics{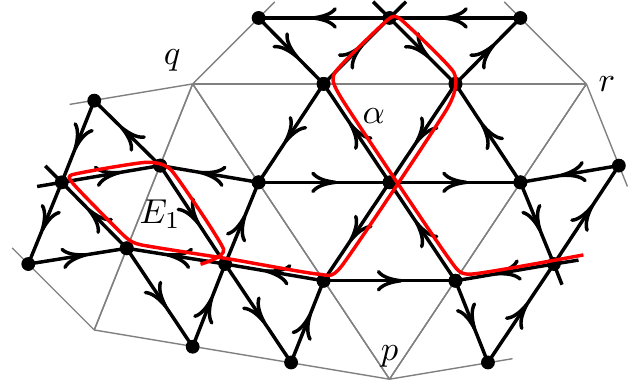}
	\captionof{figure}{Lemma \ref{nononloc}, $C_1$ -- new cycle in $\Phi_4(W)$.}
	\label{nononlocfig2}
\end{figure}
\noindent With $E_1=\overleftarrow{rrr}(l(\al_2))$, another edge cycle (see Fig. \ref{nononlocfig2}).

Note again that during these procedures no new local terms appeared in the expression of the transformed potential. Furthermore, $W -  \Phi_4(W) \mod R\langle\langle Q\rangle\rangle_{\operatorname{cyc},{N+1}}$ consists of two cycles of degree $N$, which means that we replaced $C$ by another cycle, say, $C_1$ of the same degree. Moreover, the power of $E$ appearing in $C_1$ is one less.

It follows that we can repeat the procedure with $C_1$ in place of $C$, until the power of $E$ becomes zero. Then according to the construction we will get:
$$
W - \Phi_{4n}(W) = C - u\cdot r^2(l^2(\al))r(l^2(\al))E^{0} P' E_1 \mod R\langle\langle Q\rangle\rangle_{\operatorname{cyc},{N+1}}
$$
\noindent Denote the second summand by $C_n$, this cycle is not straight. Indeed, recall that by definition of path $P'$, it ends with $f^{-1}(\al)$, so $C_n$ has a fragment $r(l^2(\al))f^{-1}(\al) = l(f^{-1}(\al))f^{-1}(\al)$.

Finally we are in position to use Corollary \ref{nonloc} to get rid of $C_n$ (again, without affecting the local part). And this concludes the proof of Proposition \ref{nononloc}
\end{proof}

\subsection{Invariants of the action of the group of right-equivalences on $R\langle\langle Q\rangle\rangle_{\operatorname{cyc}}$.}

In this section the second part of Theorem \ref{mn3} is proved. As the first step, we show that any right-equivalence class of potentials has a representative that is a sum of a primitive part and at most one extra term (Proposition \ref{onecoeff}). However, the choice of this term is not canonical and such representative is by no means unique. Furthermore, from that construction it is not immediately obvious whether there are any further relations. For example, by analogy with $SL_2$ case one could suspect that any potential is determined up to a right-equivalence by its primitive part \cite{GLS13}.

To address these questions we write down an explicit function that is an invariant of the action of right-equivalences on $R\langle\langle Q\rangle\rangle_{\operatorname{cyc}}$. Its values distinguish $\kk$ different strongly generic potentials with a given class of primitive part (\emph{cf.} \ref{equivariant1}). Moreover, it is evident from the construction that any two strongly generic potentials, with the same value of the invariant, are right-equivalent.
\medskip

\begin{proposition}\label{onecoeff}
Any strongly generic potential $W$ on $Q_{\TT,2}$ is right-equivalent to a potential of the form $W_{\operatorname{prim}} + vC$, where $W_{\operatorname{prim}}$ is the primitive part of $W$, and $C$ is some cycle from Tables \ref{table1}-\ref{table3} taken with coefficient $v\in\emph{\kk}$.
\end{proposition}
\begin{proof}
As in the previous subsection, Proposition \ref{prim} allows us to assume that the primitive part of $W$ has a form as in~(\ref{prim_std}). Further, Corollary Proposition \ref{nononloc} says that $W$ is right-equivalent to potential $W' = W_{\operatorname{prim}} + W^{\circ}_{\operatorname{nonloc}}$, where the last term consists of nonlocal terms of degree at most $7$.

As noted in the remark after Corollary \ref{nonloc}, we know that cycles in $W^{\circ}_{\operatorname{nonloc}}$ are $\overleftarrow{rllfl}(\al)$ and $\overleftarrow{ffrffl} (\al)$ (here $\al$ is any arrow that belongs to some cycle $L_q^{(1)}$). In fact, after applying suitable right-equivalences one can assume that nonlocal part of $W$ consists of cycles of degree $7$ only. Indeed, if $W$ contains degree $6$ term of form $u\cdot\overleftarrow{rllfl}(\al),\ u\in \kk^{\times}$, then one can apply right-equivalence $\varphi$:
\begin{equation}
\varphi(\beta) = \be - u \cdot \overleftarrow{llf}(\be), \ \ \text{if}\ \ \be = l(\al)
\end{equation}

\begin{figure}
	\centering
	\begin{minipage}[b]{.33\textwidth}
		\centering
		\includegraphics{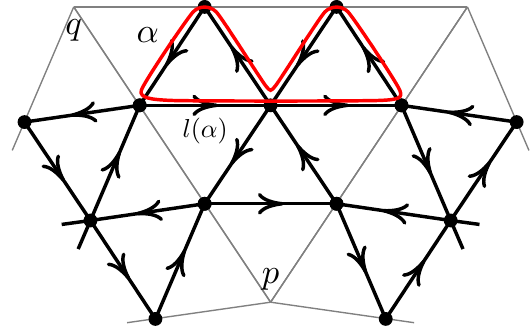}
	\end{minipage}%
	\begin{minipage}[b]{.33\textwidth}
		\centering
		\includegraphics{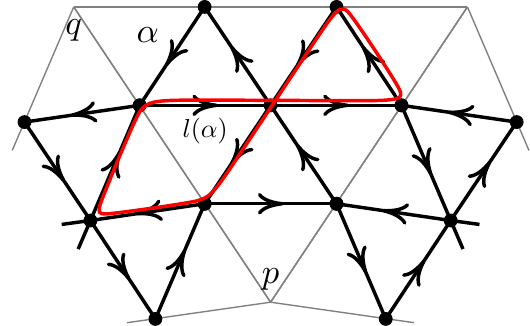}
	\end{minipage}%
	\begin{minipage}[b]{.33\textwidth}
		\centering
		\includegraphics{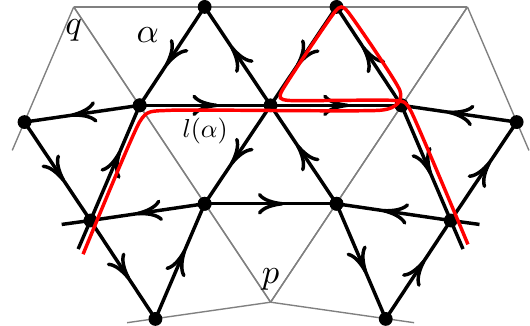}
	\end{minipage}
	\captionof{figure}{Proposition \ref{onecoeff}, eliminating non-local cycles of degree $6$.}
	\label{onecoefffig1}
\end{figure}

\noindent  The difference $W' - \varphi(W')$ is (Figure \ref{onecoefffig1}):
$$
u\cdot\overleftarrow{rllfl}(\al) + u\cdot\overleftarrow{rffllf}(l(\al)) + uv^{(2)}_p\cdot \overleftarrow{llf}(l(\al))L^{(2)}_p + \left(W^{\circ}_{\operatorname{nonloc}} - \varphi(W^{\circ}_{\operatorname{nonloc}})\right)
$$
\noindent Last term consists of nonlocal cycles of degree at least $8$ and can be disregarded by Proposition \ref{nononloc}. Same argument applies to the third summand having degree at leat $9$ (by our assumptions any puncture $p$ in $T$ satisfies $\val(p) \geq 3$). In the resulting potential the term $u\cdot\overleftarrow{rllfl}(\al)$ of initial $W'$ is replaced with $-u\cdot\overleftarrow{rffllf}(l(\al))$. The latter is a nonlocal term of degree $7$. Repeating this procedure for every degree $6$ term, we arrive to a potential with homogeneous nonlocal part of degree $7$. It consists of cycles of the form shown in the middle of Figure \ref{onecoefffig1}.

\smallskip
It suffices to show that by applying suitable right-equivalences, we can ``move'' coefficient of one such cycle to any other. Then, in particular, we can move \emph{all} coefficients to one cycle of our choice. In fact, right-equivalences that we need, were already implicitly present in the proof of Lemma \ref{removeL}. Indeed, transformation $\varphi_{-2}$ with $n = 2$ (see \ref{phi-2}) produces three terms, one of which has form $\overleftarrow{rffllf}(\al)$, and the other two after applying the composition of $\varphi_{-1},\ldots\varphi_{\val(p)}$ give cycle $\left(L^{(2)}_p\right)^{2}$ with nonzero coefficient. More precisely, (see \ref{removeLfin}), we have:
$$W' - \Phi(W') = W_{\operatorname{nonloc}} + \left(v^{(2)}_p \right)^2\left(L^{(2)}_p\right)^{2}  +(-1)^{\val(p)} v^{(1)}_pv^{(2)}_p\left(L^{(2)}_p\right)^{2}$$

Following the proof of Lemma, it is not difficult to check that
$$W_{\operatorname{nonloc}} = -\overleftarrow{rffllf}(\al) + \text{higher degree terms}$$

\noindent And then Proposition \ref{nononloc} allows one to remove everything, except something of the form $\overleftarrow{rffllf}(\al) + v\left(L^{(2)}_p\right)^{2}$ (here $p$ is the puncture to which arrow $r(\al)$ is associated.

But now we can apply this procedure with opposite signs to any other arrow $\al'$ for which $r(\al')$ belongs to $\left(L^{(2)}_p\right)$. Consequently, we get new potential $W''$, such that:
$$W' - W'' = -\overleftarrow{rffllf}(\al) + \overleftarrow{rffllf}(\al')$$

\noindent That enables us to replace coefficient of one nonlocal degree $7$ cycle by a coefficient of another cycle of that form. Since $\SSS$ is connected, we can ``accumulate'' all coefficients in one cycle; the Proposition \ref{onecoeff} is proved.
\end{proof}

We have shown that the space of strongly generic potentials with fixed primitive part modulo right-equivalences (preserving the primitive part) is at most one-dimensional. Now we are going to prove that, in fact, this space is isomorphic to $\mathbb{A}_1(\kk)$.

For that we define a finite-dimensional subspace of \emph{reduced} potentials $\overline{R\langle\langle Q\rangle\rangle}_{\operatorname{cyc}}$, such that any strongly generic potential is right-equivalent to a reduced potential. More specifically, this subspace consists of potentials that are linear combinations of cycles listed in Tables \ref{table1}-\ref{table3}. To describe the residual action of the group of right-equivalences on $\overline{R\langle\langle Q\rangle\rangle}_{\operatorname{cyc}}$, we further define a quotient map $\pi:\GG\twoheadrightarrow\bar{\GG}$ that makes the natural projection $R\langle\langle Q\rangle\rangle_{\operatorname{cyc}}\rightarrow\overline{R\langle\langle Q\rangle\rangle}_{\operatorname{cyc}}$ equivariant with respect to $\pi$. In particular, $\bar{\GG}$ acts on reduced potentials, and two reduced potentials are right-equivalent iff they lie in one orbit of $\bar{\GG}$ action (\emph{cf.} \ref{equivariant1}):

\begin{center}
\begin{equation}\label{equivariant2}
\begin{tikzcd}
\GG\arrow{r}{\pi}&\bar{\GG} \\
R\langle\langle Q\rangle\rangle_{\operatorname{cyc}}\arrow[loop, distance = 45pt]  \arrow{r} & \arrow[loop, distance = 40pt]\overline{R\langle\langle Q\rangle\rangle}_{\operatorname{cyc}}
\end{tikzcd}
\end{equation}
\end{center}

Informally, these statements mean that when studying potentials on $Q_{\TT,2}$ up to right-equivalences, one can consider only an essential ``reduced'' part, because all other terms can be eliminated by the action of $\GG$. Moreover, for understanding equivalence relations between reduced potentials it is enough to consider $\bar{\GG}$ action.
\medskip

As was already mentioned, full list of higher terms allowed in reduced potentials (see Definition \ref{reduced}) is given in tables in the end of the paper. Denote this set by $\mathfrak{C}$. It is divided into three groups: Vertex, Edge and Triangle Terms according to their relative position to the ideal triangulation of $\SSS$.  Note that Item IX.* in the table of Vertex Terms represents a class of cycles of the form:
$$
\overleftarrow{\underbrace{r...r}_{\val (p) -k -1}lr\underbrace{f...f}_{2k+1}r}(\al),\ 2\leq k\leq (\val (p)-2),
$$
\noindent for any puncture $p$ and any $\al$, such that $r(\al)$ belongs to $L^{(2)}_p$.

\begin{definition}\label{reduced}
A potential on $Q_{\TT,2}$ is called \textbf{reduced}, if it has form $W_{\operatorname{prim}}+ W'$, where the first part is the primitive part and the second term $W'$ is a linear combination of cycles from $\mathfrak{C}$. A \textbf{reduced part} of any potential is its projection onto the subspace of $R\langle\langle Q\rangle\rangle_{\operatorname{cyc}}$, spanned by chordless cycles and cycles from from $\mathfrak{C}$.
\end{definition}

Note that all non-local terms allowed in reduced potentials are the two types of ``triangle terms'', and they are precisely non-local cycles, described in the remark after Corollary \ref{nonloc}. The followig proposition follows from the standard logic of our arguments.

\begin{proposition}\label{reducepot}
Any strongly generic potential on $Q_{\TT,2}$ is right-equivalent to its reduced part.
\end{proposition}

\smallskip

To define group $\bar{\GG}$ controlling the effect of a right-equivalence on the reduced part of a potential, we study properties of cycles in $\mathfrak C$. Recall two cutting operations for a cycle $C$ along its chord $(\al,i,j)$: $\cut_\al C$ and $\al\Rightcircle C$ (Definition \ref{cut}). Direct inspection of Tables \ref{table1}-\ref{table3} gives:

\begin{lemma}\label{cutC}
For any chord $\al$ of $C\in\mathfrak{C}$, the cycle $\cut_{\al}C$ is chordless.
\end{lemma}

Together with Lemma \ref{cutlemma} this immediately shows that unitriangular right-equivalences affect the reduced part of potential only through coefficients of chordless cycles. We want to make this statement more precise. Let $\mathfrak{P}$ be a collection of all paths obtained from $\mathfrak{C}$ by $\al\Rightcircle C$ operation.

\begin{lemma}\label{cutP}
Paths from $\mathfrak{P}$ have no \emph{proper} chords. That is, for $\al_1\al_2...\al_k\in \mathfrak{P}$ a triple $(\al, i, j)$ is a chord of $P$ if and only if $i = 1, j = k$ and $\al$ is the unique arrow from $t(\al_1)$ to $s(\al_k)$.
\end{lemma}
\begin{proof}
This follows from the previous lemma. Assume for contradiction that a path \\$\al_1\al_2...\al_k=P\in\mathfrak{P}$ contains a proper chord $(\be, i, j)\neq(\al,1,k)$; where $\al$ is the unique arrow from $t(\al_1)$ to $s(\al_k)$. Choose any chordless cycle $C$ through $\al$, such that $C' =(\dd_{\al}C)P$ belongs to~$\mathfrak{C}$. Then by assumption $\cut_{\be}C$ has a chord $\al$, that contradicts the statement of Lemma \ref{cutC}.
\end{proof}

Since arrow spaces of $Q_{\TT,m}$ are at most one-dimensional, we know that any right-equivalence $\ph$ acts as
\begin{equation}\label{req}
\ph(\al) = \lambda_{\al} \al+\sum_{\substack{s(Q)=s(\al)\\ t(Q)=t(\al)\\ \deg Q>1}} c_Q Q,\ \ \ \lambda_{\al}\in \kk^{\times},\ c_Q\in \kk
\end{equation}

We claim that the set of unitriangular right-equivalences for which $c_Q = 0\ \forall Q\in \mathfrak{P}$ forms a normal subgroup $\GG_{\mathfrak{P}^c}\subset \GG$. Indeed, take any right-equivalences $\ph_1$ and $\ph_2$ and write
\begin{equation}
\ph_i(\al) = \lambda_{\al,i}\al+\sum_{\substack{s(P)=s(\al)\\ t(P)=t(\al)\\P\in\mathfrak{P}}} c_{P,i}P+\sum_{\substack{s(Q)=s(\al)\\ t(Q)=t(\al)\\Q\notin\mathfrak{P}}} c_{Q,i}Q, \ \ \ \lambda_{\al,i}\in \kk^{\times};\  c_{P,i},\ c_{Q,i}\in \kk, \ i = 1,2
\end{equation}

\noindent Then the we can compute the composition $\ph_2\circ\ph_1(\al)$:
\begin{equation}\label{reqprod}
\begin{split}
\ph_2\circ\ph_1(\al)  = \lambda_{\al,1}\lambda_{\al,2} \al\ +& \sum_{\substack{s(P)=s(\al)\\ t(P)=t(\al)\\P\in\mathfrak{P}}} P\left(c_{P,2}\lambda_{\al,1} + c_{P,1}\prod_{\al_j\in P} \lambda_{\al_j,2}\right) +\\ &+\sum_{\substack{s(Q)=s(\al)\\ t(Q)=t(\al)\\Q\notin\mathfrak{P}}} Q\left(c_{Q,2}\lambda_{\al,1} + c_{Q,1}\prod_{\al_j\in Q} \lambda_{\al_j,2}\right) + \sum_{\substack{s(R)=s(\al)\\ t(R)=t(\al)}} c_{R,i}'R
\end{split}
\end{equation}

In this expression the second and the third terms are formed by rescaling arrows by $\lambda_{\al,1}$ and then adding paths $c_{Q,2} Q$, or by first adding $c_{Q,1} Q$ and then rescaling all arrows according to $\ph_2$. Hence the fourth term consists of more complicated ``composed'' paths, that were obtained from some path $\prod_j\al_j$ appearing in $\ph_1(\al)$, by applying higher order terms of $\ph_2$ to arrows~$\al_j$. These terms can not belong to $\mathfrak{P}$ by Lemma \ref{cutP}. It follows that the set $\GG_{\mathfrak{P}^c}$ is a subgroup. To see that it is normal, note that in \ref{reqprod} the second summand does not depend on coefficients of paths $Q\notin \mathfrak{P}$.
\medskip

It follows from the discussion above that there is a well-defined group $\bar{\GG} = \GG/ \GG_{\mathfrak{P}^c}$. Moreover, the diagram \ref{equivariant2} is equivariant, since for any right-equivalence $\varphi$ coefficients of cycles from $\mathfrak{C}$ in $\ph(W)$ are affected only by numbers $\lambda_{\al}$ and $c_{P},\ P\in\mathfrak{P}$ from \ref{req}.

\smallskip
One can extract even finer information from \ref{reqprod} and the fact, that the last term in it has no terms from $\mathfrak{P}$. Recall group $\GG^{\operatorname{diag}} \equiv \left(\kk^{\times}\right)^{|Q_1|}$ acting by rescaling arrow spaces (i.e. all $c_Q$ in \ref{req} are zero), and subgroup $\GG^{\operatorname{un}}$ of unitriangular right-equivalences. Denote its image in $\bar{\GG}$ by $\bar{\GG}^{\operatorname{un}}$.

\begin{lemma}\label{semidir}
The quotient group $\bar{\GG} = \GG/ \GG_{\mathfrak{P}^c}$ is a semi-direct product:
$$
\bar{\GG} = \bar{\GG}^{\operatorname{un}} \rtimes \GG^{\operatorname{diag}}
$$
\noindent Moreover, the subgroup $\bar{\GG}^{\operatorname{un}}$ is abelian.
\end{lemma}

\medskip
Proposition \ref{onecoeff} implies that the space of reduced potentials $\overline{R\langle\langle Q\rangle\rangle}_{\operatorname{cyc}}$ modulo the action of $\bar{\GG}$ is at most one-dimensional. To prove last statement of Theorem \ref{mn3} it is sufficient to construct a non-trivial $\bar\GG$-invariant function on $\overline{R\langle\langle Q\rangle\rangle}_{\operatorname{cyc}}$.

For that fix strongly generic primitive part $W_{\operatorname{prim}}$ as in \ref{prim_std} and for every puncture $p$ of $\SSS$ define number $$k_p = v^{(1)}_p + (-1)^{\val p}v^{(2)}_p.$$
Observe that these numbers are non-zero by definition of strongly generic potentials. Next, let $e$ be an edge of the ideal triangulation $T$ connecting punctures $p$ and $q$, define numbers:
\begin{equation}\label{thetae}
\theta_e = \frac{v^{(1)}_p}{k_p} - \frac{(-1)^{\val q} v^{(2)}_q}{k_q}=\frac{v^{(1)}_q}{k_q} - \frac{(-1)^{\val p} v^{(2)}_p}{k_p} = \frac{v^{(1)}_pv^{(1)}_q-(-1)^{\val p + \val q}v^{(2)}_pv^{(2)}_q}{k_pk_q}
\end{equation}
This number is the coefficient indicated in the top-right corner of the only Edge Term field in Table \ref{table1}. Similarly, for Vertex and Triangle terms $C\in \mathfrak{C}$ in Tables \ref{table1}-\ref{table3} let $\theta_C$ be the number indicated in top-right corner of the corresponding field in these tables.

\begin{proposition}\label{Theta}
Linear functional $\Theta$ on the space of reduced potentials $\overline{R\langle\langle Q\rangle\rangle}_{\operatorname{cyc}}$ whose value on potential $W = W_{\operatorname{prim}} + \sum_{C\in \mathfrak{C}} u_C C$ is given by:
\begin{equation}
\Theta(W) = \sum_{C\in \mathfrak{C}} \theta_C u_C
\end{equation}
is invariant under the action of $\bar{\GG}^{\operatorname{un}}$.
\end{proposition}

\noindent\emph{Remark:} Modulo this proposition to finish the proof of Theorem \ref{mn3}, it remains to show that right-equivalences from $\GG^{\operatorname{diag}}$ that preserve the primitive part \ref{prim_std} act trivially on cycles from $\mathfrak{C}$. This is done in Lemma \ref{diageq}.

\begin{proof}
By Lemma \ref{semidir} group $\bar{\GG}^{\operatorname{un}}$ is an additive abelian group. It is generated by right-equivalences $\ph_{P,v},\ P\in \mathfrak{P}, v\in \kk^{\times}$ defined by:
\begin{equation}
\begin{cases}
\ph_{P,v}(\al) = \al + vP, \ \ \ &\text{if } \al \text{ is the unique arrow from } s(P) \text{ to } t(P)\\
\ph_{P,v}(\al) = \al,  \ \ \ &\text{otherwise.}
\end{cases}
\end{equation}

Thus, it is sufficient to check that $\Theta$ is preserved by right-equivalences of the form $\ph_{P,v}$ with. Without loss of generality we assume$v = 1$ (and to simplify notation we will omit the second subscript). Verifying this can be done by hand via inspection of tables in the end of the paper. We provide the computation in some illustrative cases here. Throughout these computations, we write a potential as a sum of its primitive part and remaining terms:
$$W = W_{\operatorname{prim}} + \sum_{C\in \mathfrak{C}} u_C C$$

\smallskip
\noindent \emph{Case 1:} $P = \overleftarrow{rrrf}(\al)$, where $\al$ belongs to some cycle of form $L^{(2)}_p$, and satisfies $r(\al) \neq f(\al)$. Note that in this case $P$ is a subpath only for the following three cycles from $\mathfrak{C}$:
\begin{enumerate}
\item $C_1 =\overleftarrow{lfrrrf}(\al)$ (type TI);
\item $C_2 =\overleftarrow{rrlrrrf}(\al)$ (type VII);
\item $C_3 = L^{(2)}_p\overleftarrow{rrr}(f(\al))$ (type XI).
\end{enumerate}

\noindent Then we have: $\pi(\varphi_P(W) - W) = C_1 + C_2 + v^{(2)}_p C_3$ (here and further $\pi$ denotes projection on the reduced part). It follows:
\begin{equation}\label{comp1}
\Theta(\pi(\varphi_P(W))) = \sum_{C\in \mathfrak{C}} \theta_C \cdot u_C + \theta_{C_1} \cdot1 + \theta_{C_2} \cdot1 + \theta_{C_3} \cdot v^{(2)}_p = \Theta(W) + \theta_{C_1} + \theta_{C_2} + v^{(2)}_p \theta_{C_3}
\end{equation}

\noindent Hence, we need to check that last three terms add up to zero, which is indeed true:
\begin{equation}\label{comp2}
\theta_{C_1} + \theta_{C_2} + v^{(2)}_p \theta_{C_3} = -1 + \frac{v^{(1)}_p}{k_p}+v^{(2)}_p\frac{(-1)^{\val p}}{k_p} = 0
\end{equation}

\smallskip
\noindent \emph{Case 2:} $P = \overleftarrow{llr}(\al)$ for some $\al$ with $r(\al)\neq f(\al)$. Computations are similar to the previous case. Path $P$ appears in the following cycles from $\mathfrak{C}$:
\begin{enumerate}
\item $C'_1 = \overleftarrow{lfllr}(\al)$ (Type TII);
\item $C'_2 = \overleftarrow{rrlllr}(\al)$ (Type V);
\item $C'_3 = L^{(2)}_p\overleftarrow{ll}(r(\al))$ (Type X).
\end{enumerate}

\noindent Then $\pi (\ph_p(W) - W) = C'_1 + C'_2 + v^{(2)}_p C'_3$, and:
\begin{equation}\label{comp3}
\Theta(\pi(\varphi_P(W))) = \sum_{C\in \mathfrak{C}} \theta_C \cdot u_C + \theta_{C'_1} \cdot1 + \theta_{C'_2} \cdot1 + \theta_{C'_3} \cdot v^{(2)}_p = \Theta(W) + \theta_{C'_1} + \theta_{C'_2} + v^{(2)}_p \theta_{C'_3}
\end{equation}

\noindent So \ref{comp1} and \ref{comp2} are almost identical, with the only difference in cycles taken. Invariance of~$\Theta$ in this case amounts to:
\begin{equation}
\theta_{C'_1} + \theta_{C'_2} + v^{(2)}_p \theta_{C'_3} = 1 - \frac{v^{(1)}_p}{k_p} - v^{(2)}_p\frac{(-1)^{\val p}}{k_p} = 0
\end{equation}

\medskip
\noindent \emph{Case 3:} $P = \overleftarrow{rrrr}(\al)$ for an arrow $\al$, such that $\al$ belongs to $L^{(2)}_p$, and $r(\al)$ belongs to $L^{(2)}_q$ with $p\neq q$. There are three cycles from $\mathfrak{C}$ containing $P$:
\begin{enumerate}
\item $C''_1 = \overleftarrow{rrrrrrr}(\al)$ (Type E);
\item $C''_2 \overleftarrow{llrrrr}(\al)$ (Type V);
\item $C''_3 = L^{(2)}_p\overleftarrow{rrr}(\al)$ (Type XI).
\end{enumerate}

\noindent Note that in this case cycles $C''_2$ and $C''_3$ belong to different patches, we call them $q$- and $p$-patches respectively. Then we get: $\pi (\ph_P(W) - W) = C''_1 + C''_2 + v^{(2)}_p C''_3$.

Let $e$ be the edge of triangulation between punctures $p$ and $q$, such that the edge cycle $\overleftarrow{rrrr}(\al)$ is associated to $e$. For the invariance we need to check that new contributions in $\Theta(\ph_P(W))$ add up to zero:
\begin{equation}
\theta_{C''_1} + \theta_{C''_2} + v^{(2)}_p \theta_{C''_3} = \theta_e - \frac{v^{(1)}_q}{k_q}
+v^{(2)}_p\frac{(-1)^{\val p}}{k_p} = 0,
\end{equation}
\noindent where the last equality follows directly from the definition of $\theta_e$ (see \ref{thetae}).

\medskip
We have checked invariance of $\Theta$ under elementary right-equivalences of the form $\ph_P,\ P\in\mathfrak{P}$ in three cases. It is not difficult to see directly from Tables \ref{table1}-\ref{table3} that in all remaining cases the following two simplifying properties hold:
\begin{enumerate}[label=(\roman*)]
\item $\pi (\ph_P(W) - W)$ consists of exactly two cycles $C_1$ and $C_2$;
\item cycles $C_1$ and $C_2$ belong to the same patch.
\end{enumerate}

\noindent Using these two properties, proving invariance of $\Theta$ in remaining cases becomes an elementary check since there are always two terms in $\Theta(\pi (\ph_P(W) - W))$ that cancel out. We do not write down all computations and rather give one sample instance of this. Remaining checks are completely analogous.

\smallskip
Let $P = \overleftarrow{rfffr}(\al)$ for some $\al$ with $r(\al)\neq f(\al)$. Then cycles from $\mathfrak{C}$ containing subpath $P$ are:
\begin{enumerate}
\item $C_1 = \overleftarrow{lrrfffr}(\al)$ (Type VII);
\item $C_2 = \overleftarrow{\underbrace{r...r}_{\val p - 2}lrfffr}(\al)$ (Type VIII).
\end{enumerate}

\noindent Here $p$ is the puncture, such that $C_1$ and $C_2$ lie in $p$-patch. Then
$\pi (\ph_P(W) - W) = C_1 + v^{(1)}_pC_2$ and:
\begin{equation}
\Theta(\pi (\ph_P(W) - W)) = \theta_{C_1} + v^{(1)}_p\theta_{C_2} = \frac{v^{(1)}_p}{k_p} + v^{(1)}_p\frac {-1}{k_p} = 0
\end{equation}
\noindent The invariance follows.
\smallskip

\noindent The rest of the details in the proof of Proposition \ref{Theta} is left for the reader.
\end{proof}

\medskip
Recall the decomposition $\bar{\GG} = \bar{\GG}^{\operatorname{un}} \rtimes \GG^{\operatorname{diag}}$  from Lemma \ref{semidir}. Previous proposition shows invariance of $\Theta$ under the unitriangular part of the group of right-equivalences. To conclude the proof of Theorem \ref{mn3}, it remains to deal with right-equivalences acting by scalars on arrow spaces. Proposition \ref{prim} allows to reduce the primitive part of any potential to the standard form \ref{prim_std} using the action of $\GG^{\operatorname{diag}}$. Hence we can consider only the action of this subgroup preserving the standard form of the primitive part.

\begin{lemma}\label{diageq}
Let $\bar{\GG}^{\circ}\subset \GG^{\operatorname{diag}}$ be the subgroup, that acts trivially on primitive parts of potentials. Then it preserves coefficients of cycles from $\mathfrak{C}$.
\end{lemma}

\begin{proof}
From the description of the space of primitive potentials, we know that the $\GG^{\operatorname{diag}}$ can be identified with $1$-cochains in $\mathcal{C}(\TT,m)$, the complex which has $1$-skeleton $Q_{\TT,m}$ constructed in section \ref{primpotsec}. It is immediate from the same description that $\bar{\GG}^{\circ}$ is identified with $1$-cocycles in this description. Moreover, cocycles representing zero cohomology class in $H^1(\mathcal{C}(\TT,m), \kk^{\times})$ act trivially on any potential.

It remains to check the statement for a set of cocycles generating $H^1(\mathcal{C}(\TT,m), \kk^{\times}) \equiv H^1(\SSS, \kk^{\times})$. This can be done using intersection pairing on $\SSS$. Namely,  all nontrivial cohomology classes are represented by $2g$ loops on $\SSS$ transverse to $Q_{\TT,m}$ via counting intersection index with a given arrow of the quiver. Since all cycles from $\mathfrak{C}$ are contractible as 1-chains on $\SSS\subset\mathcal{C}(\TT,m)$, their intersection index with any loop on the surface is zero.
\end{proof}
The combination of \ref{onecoeff}, \ref{Theta} and \ref{diageq} concludes the proof of Theorem \ref{mn3}.

\section{Open 3d Calabi-Yau manifolds from points of Hitchin base.} In this section we describe a class of open $3d$ manifolds $Y_\Phi$ that generalize the construction of Smith \cite{S13} to higher rank cases. We prove that these manifolds have holomorphically trivial canonical class and compute their homology and cohomology groups. Finally, we propose a construction of topological $3$-spheres that conjecturally must represent objects of Fukaya categories associated to $Y_\Phi$. Part of these ideas also appears in~\cite{KS13}

\smallskip
In what follows we fix a complex curve $S$ of genus $g$ with $d$ distinct marked points. Degree~$d$ divisor of marked points is denoted $D$.  We show that the rank of $H^3(Y_\Phi, \QQ)$ equals the number of vertices of quiver $Q_{\TT,m}$ defined in Section~\ref{clusterstr}. Furthermore, the second cohomology group $H^2(Y_\Phi, \QQ)$ has rank $md+1$.

Informally, one can think that for the choice of K\"ahler form on $Y_\Phi$ there are $m$ parameters for every point of $D$ plus a parameter for the area form on $S$. We conjecture that for every choice of symplectic structure there is a full subcategory of the corresponding Fukaya category that is equivalent to the category of finite-dimensional modules over Ginzburg algebra $\Gamma(Q_{\TT, m}, W)$ for the specific choice of parameters for the equivalence class of the primitive part of the potential $W$. The latter is governed by Theorem~\ref{mn1}(b) but a more geometric argument is necessary to make the matching of parameters precise. 

\smallskip
By contrast with Smith's construction, where $Y_\Phi$ is realized as a quadric fibration over~$S$, in our approach~$Y_\Phi$ also appears as a conic fibration over $2d$ surface $T_\Phi$. This surface is Zariski open in a blow-up of $\Tot{K_S(D)}$, the total space of the twisted canonical line bundle of $S$.

\subsection{Case $m=1$: Smith's construction of $3d$ Calabi-Yau manifolds}\label{CYfoldsm1} We recall Smith's construction of $Y_\varphi$. Let $\varphi_2$ be a meromorphic quadratic differential on $S$ with poles of order two at points of $D$ and simple zeroes. In other words, $\varphi_2$ is a generic section in $H^0\left(S, \left(K_S(D)\right)^{\otimes 2}\right)$. Further, fix rank two vector bundle $\mathcal{V}$ over $S$ satisfying $\det\mathcal{V} \simeq K_S(D)$. The manifold $Y_{\varphi_2}$ is realized inside of the total space of rank three vector bundle $\mathcal W$ over $S$ that fits in a short exact sequence:
\begin{equation}
  0\longrightarrow \mathcal{W} \overset{\al}\longrightarrow \Sym^2\mathcal V\longrightarrow  i_*(\underline{\CC}_D) \longrightarrow 0,
\end{equation}
where the last term is just a direct sum of one-dimensional skyscraper sheaves at points of $D$. The first arrow $\al$ must satisfy an additional condition that will be explained momentarily. Following Smith if that is the case we call such $\al$ ``an elementary modification'' of $\Sym^2\mathcal{V}$. Note that the map $\alpha$ is an isomorphism away from divisor $D$, and for every $p\in D$ it is an embedding of subspace $\mathcal{W}_p\subset\left(\Sym^2\mathcal{V}\right)_p$ of codimension one.

Consider the composition of $\alpha$ with the determinant map:
\begin{equation}\label{detmap}
  \mathcal{W}\overset{\al}\longrightarrow\Sym^2\mathcal{V}\overset{\det}\longrightarrow \left(\Lambda^2 \mathcal V \right)^{\otimes 2} \simeq \left( K_S(D)\right)^{\otimes 2}
\end{equation}
Aforementioned condition on the elementary modification $\al$ requires that at every point $p\in D$ the preimage of any nonzero vector from $\left(K_S(D)\right)^{\otimes 2}$ is a union of two parallel planes in the fiber $\mathcal{W}_p$.

\noindent\emph{Remark:} It is easy to show that this condition is the same as the requirement that $\forall p \in D$ the $2$-plane $\al(\mathcal{W}_p)$ is tangent to the cone of decomposable tensors in the  fiber $\left(\Sym^2\mathcal V\right)_p$. In turn, this is equivalent to a choice of a point in $\PP(\mathcal V_p)$.

\medskip
In the case when there is a splitting $\mathcal V \simeq \mathcal L_1 \oplus \mathcal L_2$ one can choose elementary modification~$\al$ in a particularly simple way. Observe that there is an isomorphism:
\begin{equation}\label{decomposition}
\Sym^2\mathcal V \simeq \left(\mathcal L_1\right)^{\otimes 2}\oplus \mathcal L_1\mathcal L_2 \oplus\left(\mathcal L_2\right)^{\otimes 2},
\end{equation}
and we can choose local trivializations $s_{1,2}$ of $\mathcal L_{1,2}$. Then the determinant map can be written as $(a,b,c)\mapsto ac - b^2$, where $a,b,c$ correspond to the coefficients of $s_1\otimes s_1$, $s_1 \otimes s_2$ and $s_2 \otimes s_2$. Thus, there is an elementary modification $\mathcal W \overset{\al}\longrightarrow\mathcal \Sym^2\mathcal V$ formed by twisting the first term in (\ref{decomposition}) by $\mathcal O(-D)$. Corresponding $\al$ is the multiplication of the first summand by a section $\delta \in H^0(S,\mathcal O(D))$ vanishing precisely at points of divisor $D$. According to this construction, at every point of $D$ the subspace $\al(\mathcal W_p)$ is given by $b = 0$.

\medskip
Smith defines $Y_{\varphi_2}$ as the preimage of the section $\varphi_2 \in \Tot\left(K_S(D)\right)^{\otimes 2}$ under the fiberwise quadratic map $\det\circ\alpha$ as in (\ref{detmap}). In other words, it is defined by the equation:
\begin{equation}\label{Smitheq}
  (\delta a)c = b^2 + \varphi_2,
\end{equation}
where $a\in H^0(S, \mathcal L_1^{\otimes 2}(-D))^\vee,\ c\in H^0(S, \mathcal L_2^{\otimes 2})^\vee,\ b \in H^0(S,\mathcal L_1 \mathcal L_2)^\vee$.

The following proposition follows immediately from the construction.
\begin{proposition}
  The fiber of the projection map $\pi: Y_{\varphi_2} \rightarrow S$ over $x\in S$ inside of $\mathcal W_x$ is:
  \begin{itemize}
    \item union of two parallel planes if $x \in D$;
    \item cone $\delta(x)ac = b^2$ if $x$ is a zero of $\varphi_2$;
    \item smooth quadric, otherwise.
  \end{itemize}
\end{proposition}

\subsection{Construction of manifolds $Y_\Phi$ for $m\geq 1$} Now we pass to the discussion of higher rank case. The input to the construction is a generic point $\Phi$ of Hitchin base:
\begin{equation}\label{HitchinPhi}
  \Phi = (\varphi_2, \varphi_3, \ldots, \varphi_{m+1}) \in \mathcal B_{S, D, m}\equiv \bigoplus_{k = 2}^{m+1}H^0\left(S, \left(K_S(D)\right)^{\otimes k}\right)
\end{equation}
Given such $\Phi$, we define a polynomial map from $H^0(S, K_S(D))$ to $H^0\left(S, \left(K_S(D)\right)^{\otimes (m+1)}\right)$ by the standard formula:
\begin{equation}\label{Phimap}
  \Phi(b) = b^{m+1} + \varphi_2b^{m-1} + \ldots + \varphi_{m+1} = \sum_{k = 0}^{m+1}\varphi_k b^{m-k+1}
\end{equation}
Here we abused notation by using same letter $\Phi$ for the polynomial map and assuming $\varphi_0 = 1$, $\varphi_1 = 0$. In particular, the preimage of the zero section defines a spectral curve:
\begin{equation}\label{spectraldef}
\Sigma \equiv \{ b\in K_S(D) \mid \Phi(b) = 0\} \subset \Tot{K_S(D)}
\end{equation}

We will need the following
\begin{lemma}\label{spectralgenus}
  Let $\Phi$ be a generic point of $B_{S, D, m}$, then the spectral curve $\Sigma$ is smooth with genus given by the formula:
  \begin{equation}
    g(\Sigma) = (m+1)^2(g-1) + \frac{m(m+1)}{2}d + 1
  \end{equation}
\end{lemma}
\begin{proof} Observe that $\Sigma \rightarrow S$ is a degree $m+1$ covering with simple branch points. The branch locus is controlled by the determinant of $\Phi(b)$, which is a section of $\left(K_S(D)\right)^{\otimes m(m+1)}$. For generic $\Phi$ the discriminant has distinct zeroes and their number is equal to the degree of the line bundle:
$$
\deg\left(\left(K_S(D)\right)^{\otimes m(m+1)}\right) = m(m+1)(d + 2g - 2)
$$
And so by Riemann-Hurwitz type argument:
$$
g(\Sigma) = (m+1)(g-1) + 1 - \frac 12\#\{\text{branch points}\} = (m+1)(g-1)+ 1 + \frac{m(m+1)(d+2g-2)} 2.
$$
Rearranging terms implies the statement of the lemma.
\end{proof}

\medskip
We define $Y_\Phi$ by modifying the right-hand side of equation \ref{Smitheq}. Let $\mathcal V$ be rank two vector bundle on $S$ as before. For clarity of exposition we also assume that there is a splitting $\mathcal {V} \simeq \mathcal L_1 \oplus \mathcal L_2$. Fix $\delta\in H^0(S, K_S(D))$ with zeroes precisely at points of $D$. Then we can write:
\begin{equation}\label{Goncharoveq}
  (\delta a)c = \Phi(b),
\end{equation}
where $a\in H^0(S, \mathcal L_1^{\otimes (m+1)}(-D))^\vee,\ c\in H^0(S, \mathcal L_2^{\otimes (m+1)})^\vee,\ b \in H^0(S,\mathcal L_1 \mathcal L_2)^\vee$. Both sides of this equation belong to $\left(K_S(D)\right)^{\otimes(m+1)}$ so it is well-defined.

\smallskip
By contrast with the case $m = 1$ (\ref{Smitheq}), the role of variable $b$ is now slightly different from $a$ and $c$. By definition $\det\mathcal V\simeq K_S(D)$ and $Y_\Phi$ defined by (\ref{Goncharoveq}) admits a natural map to $\Tot(K_S(D))$. Note that the image of this projection is not a subvariety. Moreover, its generic fiber is a conic, but over the roots of $\Phi(b)$ over points of $D$ it degenerates to an affine plane. Both of this unpleasant circumstances can be fixed by considering affine blow-ups of $\Tot(K_S(D))$ as we show in the next subsection.

We conclude this subsection showing that $Y_\Phi$ is smooth and the canonical class $K(Y_\Phi)$ is trivial.
\begin{proposition}\label{Ysmooth}
  $Y_\Phi$ is smooth.
\end{proposition}
\begin{proof} This follows from equation (\ref{Goncharoveq}) and genericity assumption for $\Phi$. There are several cases to consider. If $\delta \neq 0$ and $ac\neq 0$ then the differential of the equation contains a term $d a$ or $d c$ with nonzero coefficient. The case when $\delta \neq 0$ and $ac = 0$ corresponds to zeroes of $\Phi(b)$, the smoothness follows because their multiplicity is at most two at isolated points corresponding to roots of the discriminant. Finally, if $\delta = 0$, then $db$ has nonzero coefficient because $\Phi(b)$ has simple zeroes over points of $D$ by genericity assumption again.
\end{proof}
\begin{proposition}\label{Ycan}
  $Y_\Phi$ has holomorphically trivial canonical line bundle.
\end{proposition}
\begin{proof}

Recall the following basic fact:
\begin{lemma}
  Let $Y$ be a smooth algebraic variety and $\mathcal W \overset{\pi}\rightarrow Y$ be a vector bundle. Then the canonical class of the total space of $\mathcal W$ can be expressed as:
  \begin{equation}
    K_{\Tot\mathcal W} = \pi^*(\det \mathcal W^\vee) \otimes \pi^*(K_Y)
  \end{equation}
\end{lemma}
In our construction $Y_\Phi$ appears as the zero locus of a fiberwise polynomial map of vector bundles. More specifically, let
\begin{equation}
\mathcal W\overset{\text{def}}{=} \left(\mathcal L_1\right)^{\otimes (m+1)}(-D)\oplus \mathcal L_1\mathcal L_2 \oplus\left(\mathcal L_2\right)^{\otimes (m+1)},
\end{equation}
Then, by transporting all terms in (\ref{Goncharoveq}) to left-hand side, we can view $Y_\Phi$ as the zero locus of a section $s\in H^0\left(\Sym^\bullet(\mathcal W^\vee) \otimes \pi^*\left((K_S(D))^{\otimes(m+1)}\right)\right)$. Note that the sheaf $\Sym^\bullet(\mathcal W^\vee)$ is by definition a trivial sheaf on $\Tot\mathcal W$, and hence adjunction formula gives:
\begin{equation}\label{cantrivial}
\begin{split}
  K_{Y_\Phi} =&\left( K_{\Tot \mathcal W} \otimes \pi^*\left((K_S(D))^{\otimes(m+1)}\right)\right)\big|_{Z(s)} =\\=
  &\left(\pi^*(\det\mathcal W^\vee)\otimes \pi^*(K_S)\otimes \pi^*\left((K_S(D))^{\otimes(m+1)}\right)\right)|_{Z(s)} = \\ =
  &\left(\pi^*\left(\left(K_S(D)\right)^{\otimes (-m-2)}\otimes \mathcal O(D)\right)\otimes \pi^*(K_S)\otimes \pi^*\left((K_S(D))^{\otimes(m+1)}\right)\right)|_{Z(s)} =\\=
  &\ \left(\mathcal O_{\Tot \mathcal W}\right)|_{Z(s)}
\end{split}
\end{equation}
\end{proof}

\subsection{Manifolds $Y_\Phi$ as conic fibrations.}\label{conicfibr} Consider a factorization of the projection of $Y_\Phi$ onto $S$:
\begin{equation}\label{prefactor}
  \begin{tikzcd}
    Y_\Phi \arrow{r}{\be'} \arrow[swap]{dr}{\pi} & \Tot{K_S(D)} \arrow{d}{\kappa'} \\
    & S
  \end{tikzcd}
\end{equation}
We want to analyze fibers of the map $\be'$ via the projection $\pi$. For a given point $x\in S$ we write $\Phi_x(\be)$ for the restriction of $\Phi(b)$ to the fiber $\left(K_S(D)\right)_x$. Further, we denote by $\be_x$ the fiber of $\be$ over $x$:
$$\be_x:(Y_\Phi)_x \longrightarrow \left(K_S(D)\right)_x.$$
Since outside of $D$ the section $\delta\in H^0(S, \mathcal O(D))$ is nonzero, the fibers of $\be_x$ for $x\in S\setminus D$ are conics that degenerate to a union of two lines $a=0$ and $c = 0$ over zeroes of $\Phi_x(b)$. On the other hand, if $p \in D$ then the equation (\ref{Goncharoveq}) reads $\Phi(b) = 0$ and becomes vacuous for $a$ and $c$. Hence the image of $\be_p,\ p\in D$ consists of $(m+1)$ distinct points, and the preimage of each of them is an affine plane in $\mathcal W_p$

\medskip
To realize $Y_\Phi$ as a conic fibration over a surface we replace $\Tot K_S(D)$ by its affine blow up as follows. Let $\overline{T}_\Phi$ be the surface defined by the equation:
\begin{equation}\label{blow}
\delta \lambda = \Phi(b) \mu, \ [\lambda:\mu] \in \PP\left(
\pi'^*\left(\left(K_S(D)\right)^{\otimes(m+1)}(-D)\right) \oplus \mathcal O
\right).
\end{equation}
Thus, $\overline{T}_\Phi$ is the blow-up of $\Tot K_S(D)$ at points where $\delta = \Phi(b) = 0$. These points are precisely roots of $\Phi(b)$ over divisor $D$, so there are $d(m+1)$ of them. Our desired modification of $\Tot K_S(D)$ is defined as an open subvariety $\mu \neq 0$ of the blow-up $\overline{T}_\Phi$, we denote it $T_\Phi$.

Now it remains to note that the projection $\pi: Y_\Phi \rightarrow S$ factors through $T_\Phi$ by setting: $\lambda = ac\mu$. Hence we obtain the factorization:
\begin{equation}\label{factor}
  \begin{tikzcd}
    Y_\Phi \arrow{r}{\be} \arrow[swap]{dr}{\pi} & T_\Phi \arrow{d}{\kappa} \\
    & S
  \end{tikzcd}
\end{equation}
Outside of divisor $D$ it coincides with (\ref{prefactor}). If $x\in D$ then $\delta = 0$ and equation \ref{Goncharoveq} implies $\Phi_x(b) = 0$. Hence for every choice of $[\lambda:\mu]\neq [1:0]$ we have a conic $ac\mu = \Phi(b)\lambda$ which is exactly what we needed.

\smallskip
\noindent\emph{Remark:} Condition $\Phi(b) = 0$ is defined for points $T_\Phi$ and defines a curve in $Y_\Phi$ which is isomorphic to the spectral curve \ref{spectraldef}. Abusing notation we will usually treat spectral curve~$\Sigma$ as sitting inside of $T_\Phi$. In particular, it is immediate from the construction that the locus where conics of the singular fibration $Y_\Phi\overset{\be}\longrightarrow T_\Phi$ degenerate to the intersection of lines $ac = 0$ is given precisely by the spectral curve.

\subsection{Topology of Calabi-Yau manifolds $Y_\Phi$} In this subsection we compute cohomology and homology groups of $Y_\Phi$. Of particular interest to us are $H^2(Y_\Phi, \QQ)$ and $H_3(Y_\Phi, \QQ)$. The former is related to the choice of symplectic form in the construction of the Fukaya categories, and the latter is important as the first step towards understanding configurations of Lagrangian $3$-spheres in $Y_\Phi$.

The computation proceeds in three steps:
\begin{enumerate}
  \item Direct computation of cohomology groups of $T_\Phi$ and $\overline{T}_\Phi$
  \item Application of Decomposition Theorem to compute cohomology groups of a fiberwise compactification $X_\Phi$
  \item Derivation of $H^\bullet(Y_\Phi, \QQ)$ from an exact sequence that relates these groups to cohomology of the compactification.
\end{enumerate}
An excellent exposition of Decomposition Theorem due to Beilinson, Bernstein, Deligne and Gabber is given in~\cite{CM07}.

\medskip
Computation of the topology of $\overline{T}_\Phi$ follows from its description as a blow-up of $\Tot K_S(D)$ at $(m+1)d$ points. It follows that:
\begin{equation}
  H^\bullet(\overline{T}_\Phi, \QQ) = \QQ[0]\oplus\QQ[-1]^{\oplus 2g} \oplus \QQ[-2]^{\oplus(1 + (m+1)d)}
\end{equation}
We have complementary embedding of open and closed subsets:
\begin{equation}
  T_\Phi\lhook\joinrel\overset{j}\longrightarrow \overline{T}_\Phi\overset{i}\longleftarrow \joinrel\rhook Z\equiv \underbrace{\AAA^1\sqcup \ldots \sqcup \AAA^1}_{d\text{ copies}}
\end{equation}
And the corresponding exact sequence of sheaves:
\begin{equation}
  0\longrightarrow j_!j^!\left(\underline\QQ_{\overline T_\Phi}\right)
  \longrightarrow \underline\QQ_{\overline{T}_\Phi}
  \longrightarrow i_*i^*\left(\underline\QQ_{\overline T_\Phi}\right)
  \longrightarrow 0
\end{equation}
Applying $R\Gamma_c(-)$ to this sequence results in a long exact sequence relating cohomology with compact support $H_c^\bullet\left(T_\Phi,\QQ\right)$, $H_c^\bullet\left(\overline{T}_\Phi,\QQ\right)$ and $H_c^\bullet\left(Z,\QQ\right)$. The middle term is known by Poincar\'e duality and we have:
\begin{alignat*}{2}
    & H_c^0\left(T_\Phi,\QQ\right)\to\hspace{30pt}0 & \to &\ 0  \to \\
\to & H_c^1\left(T_\Phi,\QQ\right)\to\hspace{30pt}0 & \to &\ 0  \to \\
\to & H_c^2\left(T_\Phi,\QQ\right)\to \QQ^{\oplus(1+(m+1)d)} & \to &\QQ^{\oplus d}  \to \\
\to & H_c^3\left(T_\Phi,\QQ\right)\to\hspace{28pt}\QQ^{\oplus 2g} &\to &\ 0  \to \\
\to & H_c^4\left(T_\Phi,\QQ\right)\to\hspace{28pt}\QQ & \to &\ 0 \to
\end{alignat*}
The third arrow in the second cohomology row comes from ordinary restriction of differential forms with compact support and is clearly a surjection. Applying Poincar\'e duality again we get:
\begin{equation}
  H^\bullet(T_\Phi, \QQ) = \QQ[0]\oplus\QQ[-1]^{\oplus 2g}\oplus\QQ[-2]^{\oplus(1 + md)}
\end{equation}

\medskip
The fiberwise compactification $Y_\Phi\subset X_\Phi$ can be obtained by homogenizing equation for affine conics. More specifically, recall equation \ref{blow}. For fixed $x\in S$ and $[\lambda:\mu]$ the condition $\lambda = ac\mu$ is represented by an element of $\Sym^\bullet\left(\left( \mathcal L_1\right)^{\otimes(m+1)}(-D)^\vee \oplus \left( \mathcal L_2\right)^{\otimes(m+1)\vee}\right)$; it is homogenized by adding a trivial direct summand and writing: $E^2\lambda = AC\mu$, where $a = A/E,\ c = C/e$.

Note that $\mu\neq 0$ by definition of $T_\Phi$ and so for $E = 0$ there are always two solutions $[A:C:E] = [1:0:0]$ and $[0:0:1]$. Consequently, the complement $X_\Phi\setminus Y_\Phi$ consists of two copies of $T_\Phi$.

\smallskip
Since the map $\bar\beta: X_\Phi\longrightarrow T_\Phi$ is proper the Decomposition Theorem applies:
\begin{equation}\label{decompositionthm}
  R\bar\beta_*(\underline{\QQ}_{X_\Phi}) \simeq \bigoplus_{q \geq 0}\left(R^q \bar\beta_*(\underline{\QQ}_{X_\Phi})[-q]\right)
\end{equation}
Recall spectral curve $\Sigma \subset T_\Phi$ given by $\{\Phi(b) = 0\}$. We have:
\begin{lemma}\label{directimage} Derived direct images $R^q\bar\beta_*(\underline{\QQ}_{X_\Phi})$ are:
  \begin{itemize}
    \item $R^0 \bar\beta_*(\underline{\QQ}_{X_\Phi}) \simeq \underline{\QQ}_{T_\Phi};$
    \item $R^2 \bar\beta_*(\underline{\QQ}_{X_\Phi}) \simeq \underline{\QQ}_{T_\Phi} \oplus (i_\Sigma)_*(\underline{\QQ}_\Sigma);$
    \item $R^q \bar\beta_*(\underline{\QQ}_{X_\Phi}) = 0,\ q\neq 0,2$
  \end{itemize}
\end{lemma}
\begin{proof} Generic fiber of $\bar\beta$ is $\PP_1$ and as noted in the end of Section~\ref{conicfibr} it degenerates to a union of two spheres meeting at a point $\{AC =0 \}\subset\PP_2$ along $\Sigma$. If $\al$ and $\gamma$ represent fundamental classes of spheres $A = 0$ and $C = 0$ respectively, then restriction to a generic fiber is given by $(\al, \gamma)\mapsto \al+\gamma$. Hence we can split off the term $(i_\Sigma)_*(\underline{\QQ}_\Sigma),$ the kernel of the restriction.
\end{proof}

Finally we compute homology and cohomology groups for $Y_\Phi$. The complement $X_\Phi \setminus Y_\Phi$ is given by the disjoint union of varieties $\{A =0\}$ and $\{C = 0\}$, each isomorphic to $T_\Phi$. Hence there is a pair of complementary embeddings:
\begin{equation}
  Y_\Phi\lhook\joinrel\overset{j}\longrightarrow X_\Phi\overset{i}\longleftarrow \joinrel\rhook T_\Phi^{\{A=0\}}\sqcup T_\Phi^{\{C=0\}}
\end{equation}
We apply the functor of sections with compact support to the exact sequence:
\begin{equation}
  0\longrightarrow j_!j^!\left(\underline\QQ_{X_\Phi}\right)
  \longrightarrow \underline\QQ_{X_\Phi}
  \longrightarrow i_*i^*\left(\underline\QQ_{X_\Phi}\right)
  \longrightarrow 0
\end{equation}

\begin{equation*}
\begin{tikzcd}[column sep=0.3cm]
\text{Degree} & & H_c^\bullet(Y_\Phi, \QQ) & & H_c^\bullet(X_\Phi, \QQ) & & H_c^\bullet(T_\Phi\sqcup T_\Phi, \QQ) & &\\
0 & & 0 & \longrightarrow & 0 & \longrightarrow& 0 & \longrightarrow &\\
1 & \longrightarrow & 0 & \longrightarrow & 0 & \longrightarrow & 0 & \longrightarrow &\\
2 & \longrightarrow & H_c^2(Y_\Phi, \QQ) & \longrightarrow & \QQ\oplus\QQ^{\oplus (1+md)}\oplus 0 & \overset{\circled{1}}\longrightarrow & \QQ^{\oplus 2(1+md)} & \longrightarrow &\\
3 & \longrightarrow & H_c^3(Y_\Phi, \QQ) & \longrightarrow & \QQ^{\oplus 2g(\Sigma)}\oplus\QQ^{\oplus 2g}\oplus0 & \overset{\circled{2}}\longrightarrow & \QQ^{\oplus2(2g)} & \longrightarrow &\\
4 & \longrightarrow & H_c^4(Y_\Phi, \QQ) & \longrightarrow & \QQ\oplus\QQ\oplus\QQ^{\oplus(1+md)} & \overset{\circled{3}}\longrightarrow & \QQ^{\oplus 2} & \longrightarrow &\\
5 & \longrightarrow & H_c^5(Y_\Phi, \QQ) & \longrightarrow & 0\oplus0\oplus\QQ^{\oplus 2g} & \longrightarrow & 0 & \longrightarrow &\\
6 & \longrightarrow & \QQ & \longrightarrow & 0\oplus0\oplus\QQ & \longrightarrow & 0 & &
\end{tikzcd}
\end{equation*}

Note that direct sums in the middle column correspond to contributions of the three terms from Lemma~\ref{directimage}.

\begin{proposition}\label{topology}
\begin{equation}
  H_c^\bullet(Y_\Phi, \QQ) = \QQ[-3]^{\oplus m(m+2)(2g-2+d)}\oplus \QQ[-4]^{\oplus (md+1)}\oplus\QQ[-5]^{\oplus 2g}\oplus\QQ[-6]
\end{equation}
\begin{equation}
H^\bullet(Y_\Phi, \QQ) =\QQ[0]\oplus \QQ[-1]^{\oplus 2g}\oplus\QQ[-2]^{\oplus (md+1)}\QQ[-3]^{\oplus m(m+2)(2g-2+d)}
\end{equation}
\end{proposition}
It is clear that the second equality is a consequence of the first by Poincar\'e duality. The proposition follows from elementary rank computations using the following
\begin{lemma}
  \begin{enumerate}[label = (\alph*)]
    \item Arrow \circled{$1$} is an injection;
    \item Arrow \circled{$2$} is a surjection;
    \item Arrow \circled{$3$} is a surjection.
  \end{enumerate}
\end{lemma}
\begin{proof} The third column in the long exact sequence corresponds to cohomology of two copies of $T_\Phi$. The composition of the projection $\bar\beta:X_\Phi\rightarrow T_\Phi$ with the embedding of each copy $T_\Phi^{\{A = 0\}}$ or $T_\Phi^{\{C = 0\}}$ is the identity. Hence the diagonal lies in the image of restriction map
$$\underline\QQ_{X_\Phi}
\longrightarrow i_*\left(\underline{\QQ}_{T_\Phi\sqcup T_\Phi}\right)$$
More precisely, the map between last two columns of the long exact sequence can be obtained by applying $R\Gamma_c(-)$ to the natural morphism
$$R\bar\beta_*\left(\underline{\QQ}_{X_\Phi}\longrightarrow i^*\underline{\QQ}_{T_\Phi\sqcup T_\Phi}\right)$$
By Lemma~\ref{directimage} the source of the morphism consists of three direct summands and its restriction to $\underline{\QQ}_{T_\Phi}[0]$ is the diagonal embedding.

To prove parts (a) and (c) in it suffices to show that the image of the extra copy of~$\QQ$ coming from $(i_\Sigma)_*(\underline\QQ_{\Sigma})$ lands off diagonal. This is clear from the description of the direct image given in the proof of Lemma~\ref{directimage}. It says that the image of the group $H_c^{\bullet - 2}(\Sigma, \QQ)$ in $H_c^\bullet(T_\Phi\sqcup T_\Phi,\QQ)$ belongs to anti-diagonal.

For part (b) we observe that there is a chain of isomorphisms $$H^3_c(T_\Phi, \QQ) \simeq H_1(T_\Phi, \QQ)\simeq H_1(S,\QQ),$$
where the second identification is induced by the projection $T_\Phi \overset{\kappa}\rightarrow S$. The image of the map~\circled{$2$} restricted to the summand $\QQ^{\oplus 2g(\Sigma)}$ belongs to anti-diagonal and it can be described by the action of the spectral cover map $\Sigma\longrightarrow S$ on the first homology groups of curves. The latter is known to be surjective and the statement follows.
\end{proof}

\smallskip
We compute the third cohomology groups in Proposition~\ref{topology}, the rest is similar:
\begin{equation}
\begin{split}
\operatorname{rk}\left(H^3_c(Y_\Phi,\QQ)\right) = - \underbrace{(1 + 1 + md)}_{H^2_c(X_\Phi)} + \underbrace{2(1+md)}_{H^2_c(T_\Phi\sqcup T_\Phi)}
+ \underbrace{(2g(\Sigma) + 2g)}_{H^3_c(X_\Phi)} -& \underbrace{2(2g)}_{H^3_c(T_\Phi\sqcup T_\Phi)} = \\ =
2g(\Sigma) + md - 2g = (m+1)^2(2g-2) + m(m+1)d+2 +&md - 2g = \\ &=  m(m+2)(2g-2+d)
\end{split}
\end{equation}
where in the third equality uses Lemma~\ref{spectralgenus}.

\subsection{Lagrangian $3$-spheres in $Y_\Phi$}\label{spheres} The rank of the third homology groups $H_3(Y_\Phi,\QQ) \simeq H^3_c(Y_\Phi,\QQ)$ coincides with the number of vertices of quivers $Q_{\mathcal T, m}$ (see Section~\ref{clusterstr}). In this subsection we propose a way to produce \emph{topological} $3$-spheres in $Y_\Phi$ that coincides with the description of Smith \cite{S13} in the case $m=1$. The symplectic part of the story as well as many other details is subject to future research. This section has expository role and contains several informal claims.

In \cite{S13} every vertex gives rise to a special Lagrangian sphere $S^3\subset Y_\Phi$. We refer the reader to Smith's paper for details. The construction uses saddle-connections of the quadratic differential $\varphi \in H^0(S, K_S(D)^{\otimes 2})$. Every saddle-connection gives rise to a special Lagrangian sphere as outlined below. On the other hand, saddle-connections can be matched with vertices of quiver $Q_{\mathcal T, 1}$ by explicit geometric considerations for $\varphi$.

Roughly speaking, Smith's description of Lagrangian spheres is as follows. Recall that away from zeroes and poles of $\varphi$ a fiber of the projection $Y_\varphi\rightarrow S$ is a smooth quadric. Over the zeroes of the quadratic differential the quadric degenerates to an affine cone. A saddle-connection $\gamma(t)$ of $\varphi$ is a path on $S$ connecting two zeroes with constant $\varphi$-phase:
$$\operatorname{Im} \left(e^{-i\theta}\frac{\partial}{\partial t}\int_{\gamma_{[0,t]}}\sqrt\varphi\right) = 0$$
Thus, over interior points of a saddle-connection the fiber is a smooth quadric isomorphic to the cotangent bundle $T^*S^2$, at endpoints it degenerates to the cone. Special Lagrangian sphere associated to a saddle connection is formed by the zero section of the cotangent bundle that shrinks to a dot over endpoints.

\medskip
Recall the map $\beta:Y_\Phi\rightarrow T_\Phi$. Description of $Y_\Phi$ as a conic fibration gives an alternative view on Lagrangian spheres. Assume that there is an embedding of a closed topological $2$-disk $f: B^2 \rightarrow T_\Phi$ such that $f(\partial B^2)\subset \Sigma$ and no points in the interior are mapped to $\Sigma$. Then the preimage $F = \beta^{-1}\circ f(B^2)$ consists of affine conics
$$\{ac = \epsilon\}\simeq \CC^\times $$
over points of $f(B^2\setminus \partial B^2)$. Over the image of the boundary of $B^2$ these conics degenerate to crossings $ac = 0$. Then there is a $3$-sphere $L\subset F$ such that for any $x\in f(D\setminus \partial D) $ the intersection $L\cap (Y_\Phi)_x$ is a circle $S^1$ generating $H_1((Y_\Phi)_x)$; this circle shrinks to a point over the boundary $f(\partial D)\subset T_\Phi$. In particular, this means that any loop $\gamma\subset\Sigma$ that is contractible in $T_\Phi$ gives rise to a $3$-sphere $L$ described above.
\medskip

To establish a link between quivers $Q_{\TT,m}$ from Section~\ref{clusterstr} and collections of Lagrangian spheres as in~\cite{S13} one has to incorporate the intersection form on $H^3(Y_\Phi, \QQ)$. It is not difficult to show that for topological $3$-spheres described above the intersection numbers can be identified with the intersection form on $H_1(\Sigma,\QQ)$. Thus, from the topological perspective one has to find the collection of disk embeddings to $T_\Phi$ such that their boundaries are mapped to $\Sigma$ and the intersection pattern of the corresponding collection of loops is described by $Q_{\TT,m}$.

We note that in Smith's paper a \emph{canonical} collection of Lagrangian spheres is associated to every generic quadratic differential. For any $\phi_2$ intersection numbers correspond to a quiver $Q_{\TT,1}$ where the triangulation is explicitly constructed from the horizontal trajectories of quadratic differential. For $m>1$ this picture is related to the geometry of higher order differentials and might be considerably more complicated. It has been partially studied in the series of papers by Gaiotto, Moore and Neitzke~\cite{GMN08}, \cite{GMN12},\cite{GMN12p2}. In particular, for $m>1$ it may no longer be true that the collection of Lagrangian spheres associated to a polydifferential corresponds to some triangulation. However, we emphasize that collections of topological $3$-spheres still exist and define subcategories of Fukaya categories in a more flexible way.

\smallskip
In~\cite{G16} Goncharov defines ``topological spectral cover'' $\mathbf{\pi}: \mathbf{\Sigma}\rightarrow \SSS$ that can be associated to a triangulation (or more generally to an object called ``ideal web''). Conjecturally, in certain cases this map encodes topology of actual spectral curve $\Sigma$ with its natural map to $S$ arising from $\Phi\in\mathcal{B}_{S,D,m}$. Conditional on that relationship, one can exhibit a collection of loops on $\Sigma$ that lie in the kernel of associated map $H_1(\Sigma,\QQ) \longrightarrow H_1(S,\QQ)$ and with intersection pattern given by $Q_{\TT,m}$. Indeed, in Goncharov's construction the surface $\SSS$ is tiled by disks that he calls faces and their boundaries can be lifted to non-trivial cycles in $\mathbf{\Sigma}$. The claim about the quiver is immediate from details of the construction of topological spectral cover that describes $\mathbf{\Sigma}\longrightarrow \SSS$ as a result of specific gluing procedure (see loc. cit.). Hence, the associated collection of $3$-spheres in $Y_\Phi$ will have desired properties from Conjecture~\ref{conject}.

\newpage

\begin{center}
\begin{tabular}{ ||c|c|c||}
\hline
\multicolumn{2}{||m{0.66\textwidth}|}{\centering\vspace{5pt}\textbf{Triangle Terms}}&\multicolumn{1}{m{0.33\textwidth}||}{\centering\vspace{5pt}\textbf{Edge Terms}} \\[1.2ex]
\hline
TI.\includegraphics[width=.28\textwidth]{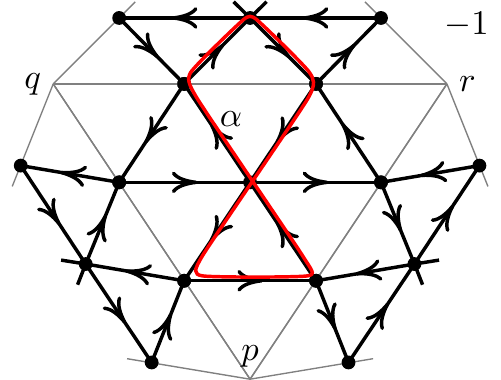}&
TII.\includegraphics[width=.28\textwidth]{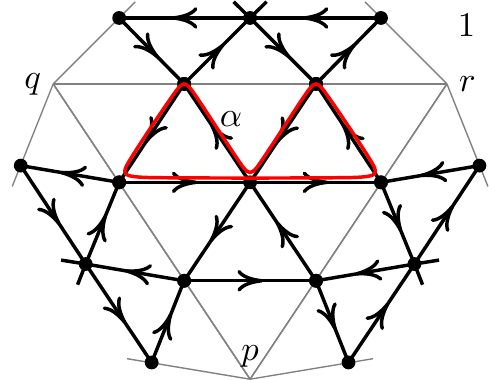}&
E. \includegraphics[width=.28\textwidth]{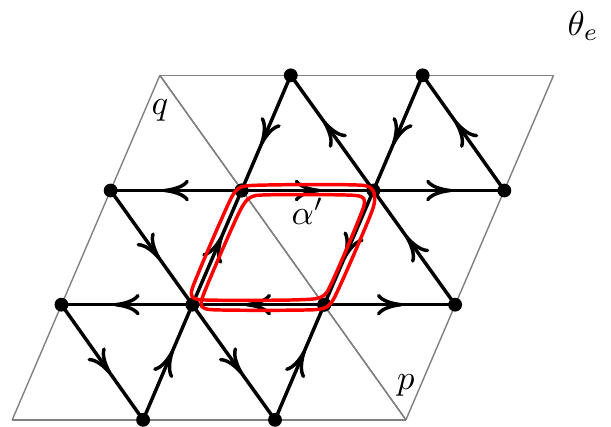}\\
\hline
\multicolumn{1}{||m{0.3\textwidth}|}
{\centering\vspace{4pt}$\overleftarrow{llfrrr}(\al)$}&
\multicolumn{1}{m{0.3\textwidth}|}
{\centering\vspace{4pt}$\overleftarrow{llfll}(\al)$}&
\multicolumn{1}{m{0.3\textwidth}||}
{\centering\vspace{4pt}$\overleftarrow{rrrrrrr}(\al')$}\\[2ex]
\hline
\end{tabular}
\captionof{table}{Table of Triangle and Edge Terms}
\label{table1}
\end{center}

\begin{center}
\begin{tabular}{ ||c|c|c||}
\hline
\multicolumn{3}{||m{\textwidth}||}{\centering\vspace{5pt}\textbf{Vertex Terms, Part 1}} \\[1.2ex]
\hline
I.\includegraphics[width=.28\textwidth]{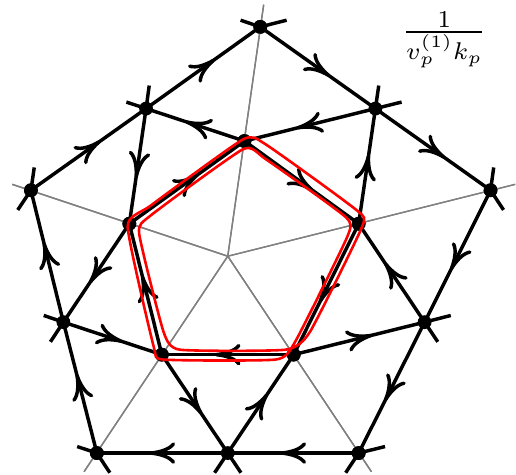}&
II.\includegraphics[width=.28\textwidth]{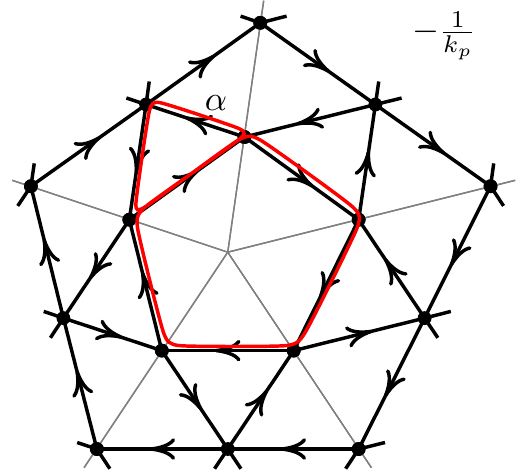}&
III.\includegraphics[width=.28\textwidth]{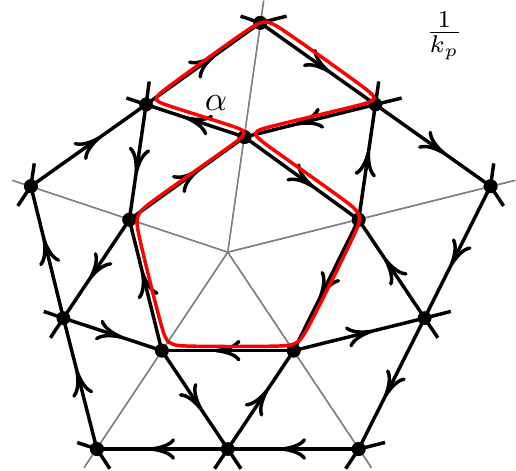}\\
\hline
\multicolumn{1}{||m{0.3\textwidth}|}
{\centering\vspace{4pt}$\left(L^{(1)}_p\right)^2$}&
\multicolumn{1}{m{0.3\textwidth}|}
{\centering\vspace{4pt}$L^{(1)}_p \overleftarrow{ll}(\al)$}&
\multicolumn{1}{m{0.3\textwidth}||}
{\centering\vspace{4pt}$L^{(1)}_p \overleftarrow{rrr}(\al)$}\\[2ex]
\hline
IV.\includegraphics[width=.28\textwidth]{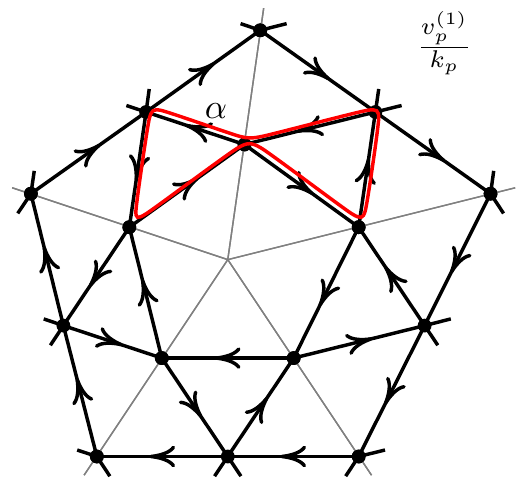}&
V.\includegraphics[width=.28\textwidth]{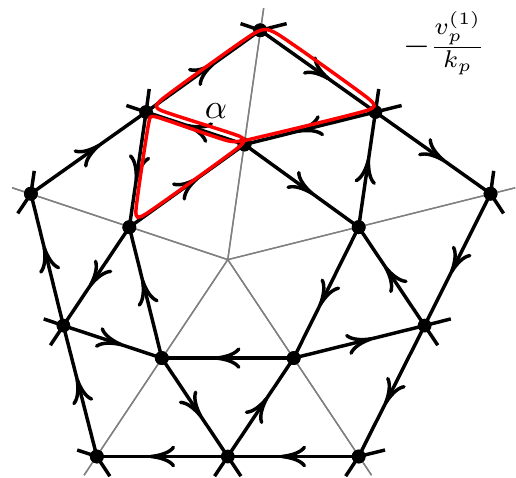}&
VI.\includegraphics[width=.28\textwidth]{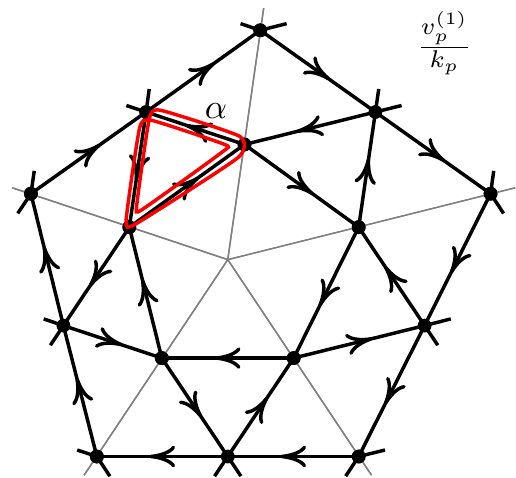}\\
\hline
\multicolumn{1}{||m{0.3\textwidth}|}
{\centering\vspace{4pt}$\overleftarrow{llrll}(\al)$}&
\multicolumn{1}{m{0.3\textwidth}|}
{\centering\vspace{4pt}$\overleftarrow{rrrlll}(\al)$}&
\multicolumn{1}{m{0.3\textwidth}||}
{\centering\vspace{4pt}$\overleftarrow{lllll}(\al)$}\\[2ex]
\hline
\end{tabular}
\captionof{table}{Table of Vertex Terms Part 1}
\label{table2}
\end{center}

\begin{center}
\begin{tabular}{ ||c|c|c||}
\hline
\multicolumn{3}{||m{\textwidth}||}{\centering\vspace{5pt}\textbf{Vertex Terms, Part 2}} \\[1.2ex]
\hline
VII.\includegraphics[width=.28\textwidth]{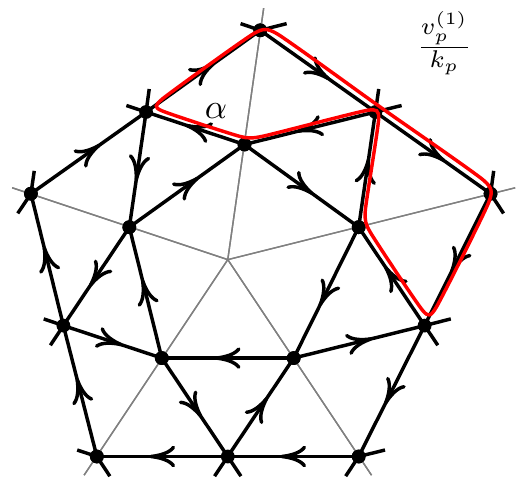}&
VIII.\includegraphics[width=.28\textwidth]{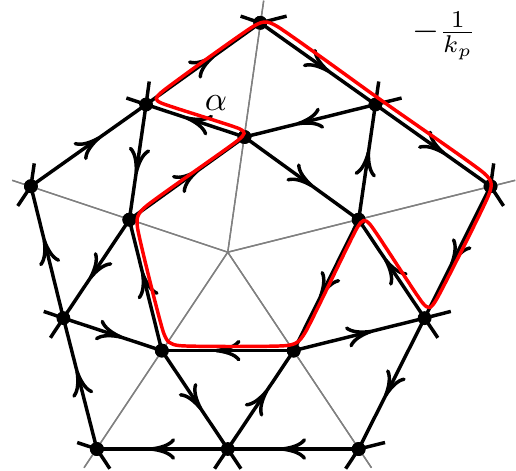}&
IX.*\includegraphics[width=.28\textwidth]{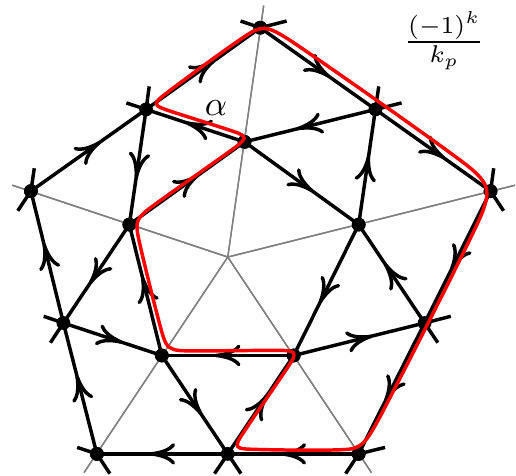}\\
\hline
\multicolumn{1}{||m{0.3\textwidth}|}
{\centering\vspace{4pt}$\overleftarrow{lrrrfrr}(\al)$}&
\multicolumn{1}{m{0.3\textwidth}|}
{\centering\vspace{4pt}$\overleftarrow{\underbrace{r...r}_{\val p - 2}lrfffr}(\al)$}&
\multicolumn{1}{m{0.3\textwidth}||}
{\centering\vspace{4pt}$\overleftarrow{\underbrace{r...r}_{\val p -k -1}lr\underbrace{f...f}_{2k+1}g}(\al)$}\\[2ex]
\hline
X.\includegraphics[width=.28\textwidth]{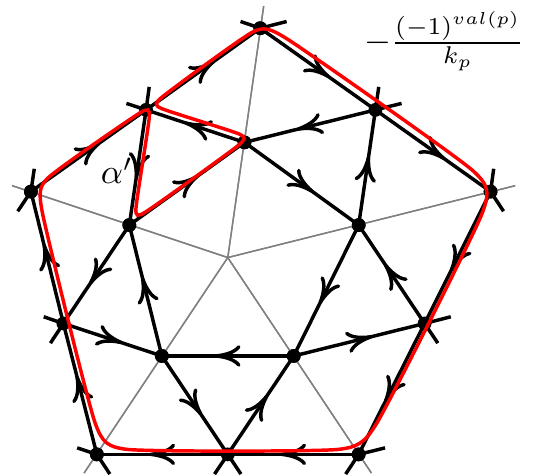}&
XI.\includegraphics[width=.28\textwidth]{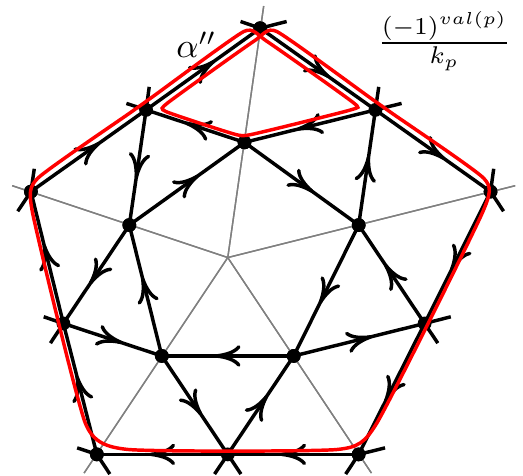}&
XII.\includegraphics[width=.28\textwidth]{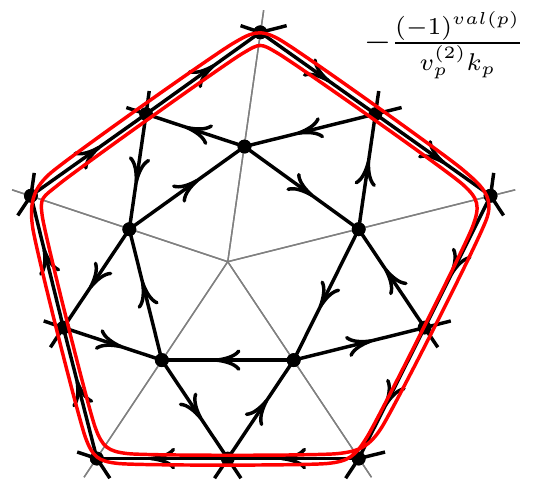}\\
\hline
\multicolumn{1}{||m{0.3\textwidth}|}
{\centering\vspace{4pt}$\overleftarrow{ll}(\al')L^{(2)}_p$}&
\multicolumn{1}{m{0.3\textwidth}|}
{\centering\vspace{4pt}$\overleftarrow{rrr}(\al'')L^{(2)}_p$}&
\multicolumn{1}{m{0.3\textwidth}||}
{\centering\vspace{4pt}$\left(L^{(2)}_p\right)^2$}\\[2ex]
\hline
\end{tabular}
\captionof{table}{Table of Vertex Terms Part 2}
\label{table3}
\end{center}

\bibliographystyle{alphaurl}
\bibliography{references}
\end{document}